%% file: main.tex
\documentclass[12pt,oneside,a4paper]{amsart}
\pdfoutput=1
\usepackage[dvipsnames]{xcolor}
\usepackage{mathtools,amssymb}
\usepackage[normalem]{ulem} 
\usepackage{tikz}
\usepackage{caption}
\usepackage{subcaption}
\usepackage{hyperref}

\usepackage[T1]{fontenc}
\usepackage{newtxtext,newtxmath}
\usepackage{microtype}

\usepackage{todonotes}
\makeatletter
\providecommand\@dotsep{5}
\makeatother
\setuptodonotes{backgroundcolor=white,bordercolor=white,noline}

\renewcommand{\epsilon}{\varepsilon}
\renewcommand{\phi}{\varphi}

\mathcode`\;="303B
\newcommand{\mres}{ \! \mathbin{\vrule height 1.4ex depth 0pt width
    0.13ex\vrule height 0.13ex depth 0pt width 0.8ex} \!}

\newcommand{\tempalpha}{\alpha}

\newcommand{\dom}{\mathop{\rm dom}}
\newcommand{\intr}{\mathop{\rm int}}
\newcommand{\ri}{\mathop{\rm r.i.}}
\newcommand{\R}{{\mathbf R}}
\newcommand{\Rn}{{\R^n}}

\newcommand{\p}{{\partial}}

\newtheorem{theorem}{Theorem}[section]
\newtheorem{proposition}[theorem]{Proposition}
\newtheorem{corollary}[theorem]{Corollary}
\newtheorem{lemma}[theorem]{Lemma}
\newtheorem{definition}[theorem]{Definition}
\newtheorem{remark}[theorem]{Remark}
\newtheorem{example}[theorem]{Example}

\newtheorem{assumption}{Assumption}

\makeatletter
\renewcommand\subsection{\@startsection{subsection}{2}%
  \z@{.5\linespacing\@plus.7\linespacing}{.1\linespacing}%
  {\normalfont\scshape}}
\makeatother

\title[The monopolist's free boundary problem in the plane]{The monopolist's free~boundary problem in the plane}
\author{Robert J. McCann
  \and
  Cale Rankin
  \and
Kelvin Shuangjian Zhang}
\address{Departments of Mathematics and of Economics, University of Toronto, Toronto Ontario M5S 2E4 Canada {\tt mccann@math.utoronto.ca}}
\address{School of Science, UNSW Canberra ACT 2600 Australia 
  {\tt c.rankin@unsw.edu.au}}
\address{School of Mathematical Sciences and Center for Applied Mathematics,
Fudan University, 
220 Handan Road, Yangpu District, Shanghai 200433, P.R. China 
{\tt ksjzhang@fudan.edu.cn}}

\thanks{
MSC 2020: 
35R35 % Free boundary problems
49N10 % Linear-quadratic optimal control
91B41 % Contract theory (including moral hazard and adverse selection)
Secondary:
35Q91 % PDE in economics and social sciences
49Q22 % Optimal transport
90B50 % Management Decision-making
91A65 %Heirarchical games (including Stackelberg games)
91B43 %Principal Agent 
\\ 
$^*$Robert McCann's work was supported in part by the Canada Research Chairs program CRC-2020-00289, Natural Sciences and Engineering Research Council of Canada Discovery Grants RGPIN--2020--04162 and 2026-04906, and
a grant from the Simons Foundation (923125, McCann). 
Cale Rankin's was supported in part by postdoctoral fellowships from the Fields Institute for Research in Mathematical Sciences and the University of Toronto. Kelvin Shuangjian Zhang's was supported by a Start-Up Grant from Fudan University. The authors are grateful to Guillaume Carlier, Shibing Chen, Philippe Chon\'e, Krzysztof Ciosmak, Alessio Figalli, Lucas O'Brien, Jean-Charles Rochet, Ting-Kam Leonard Wong and Yi Ru-Ya Zhang for fruitful exchanges,
and to Australian National University and the Sydney Mathematical Research Institute for their hospitality in 2026.
    \copyright \today}

\begin{document}
\begin{abstract}
We study the Monopolist's problem with a focus on the free boundary separating
bunched from unbunched consumers, especially in the plane, and give a 
description of $C^{1,1}(\overline{\Omega})$ solutions for the family of square domains $\{(a,a+1)^2\}_{a \ge
0}$.  The Monopolist's problem is fundamental in economics, yet widely
considered analytically intractable when both consumers and products have more
than one degree of heterogeneity.  Mathematically, the problem is to minimize a
smooth, uniformly convex Lagrangian over the space of nonnegative convex
functions.  What results is a free boundary problem between the regions of
strict and nonstrict convexity.  We study both the structure of the free boundary problem on convex domains in $\R^n$ showing 
 that each bunch extends to this boundary and providing a proof in $\R^2$ that
the interior free boundary can only fail to be smooth in one of four specific
ways (cusp, high frequency oscillations, stray bunch, nontransversal bunch); in addition we offer a 
description of $C^{1,1}(\overline{\Omega})$ solutions to Rochet and Chon\'e's example on the
family of squares $\Omega = (a,a+1)^2$, where we discover bifurcations first to
targeted and then to blunt bunching as the distance $a \ge 0$ to the origin is
increased.  To do this, we extend the localization for measures in convex-order
to accommodate potential discontinuities in the product allocation map at the
fixed boundary. We also employ techniques from the study of the Monge--Amp\`ere
equation and the obstacle problem.
\end{abstract}
\maketitle
\setcounter{tocdepth}{1}
\tableofcontents

\section{Introduction} The Monopolist's problem is a principal-agent model for
making decisions facing asymmetric information; it has fundamental importance
in microeconomic theory. A simple form from \cite{RochetChone98} capturing
multiple dimensions of heterogeneity that we rederive below is to
\begin{equation}
\label{eq:monopolist} \begin{dcases} &\text{minimize}\ L[u]:= \int_{\Omega}
\left(\frac{1}{2}|Du-x|^2 + u \right) \, dx,\\ &\text{on}\  \mathcal{U} := \{u : \overline\Omega \rightarrow
 [0,\infty)\ \text{convex and lower semicontinuous}
\}.
  \end{dcases}
\end{equation}
We always take $\Omega \subset \mathbf{R}^n$ to be open and convex with compact closure
$\overline \Omega$. Our goal in this paper is to elucidate 
 the properties of
solutions to this minimization problem, with an eventual focus on the
two-dimensional setting. Because the minimization takes place over the set of
convex functions the problem has a free boundary structure. The free boundary
separates the region where the function is convex, but not strictly convex,
from the region where the function is strictly convex. In this paper we refine
our understanding of the free boundary structure in all dimensions by showing
regions where $u$ is not strictly convex always extend to the fixed boundary $\partial
\Omega$. 
In two-dimensions, we show in a neighbourhood of a 
{newly identified} class of
free boundary points (that we call tame), the minimizer solves 
{a classical
obstacle problem, but with a low regularity endogenous obstacle}. From this we obtain that the tame free boundary is locally
piecewise Lipschitz except at accumulation points of its local maxima; 
it has Hausdorff dimension strictly less than two and is the graph of a
continuous function.
We also establish a bootstrapping procedure: if the free boundary is suitably Lipschitz, it is $C^\infty$;  see \cite{ChenFigalliZhang26+,McCannOBrienRankin26+} for subsequent developments.
As an application of our techniques we completely describe $C^{1,1}(\overline{\Omega})$ solutions on the square domains $\Omega =(a,a+1)^2 \subset \mathbf{R}^2$ with $a \ge
0$. Despite significant numerical
\cite{BoermaTsyvinskiZimin22+,CarlierLachand-RobertMaury01,CarlierDupuis17,CarlierDupiusRochetThanassoulis24,EkelandMoreno-Bromberg10,MerigotOudet14,Mirebeau16,Oberman13},
and analytic attempts
\cite{RochetChone98,KolesnikovSandomirskiyTsyvinskiZimin22+,McCannZhang23+} the
description of the solution on the square has previously remained incomplete,
and has come to be regarded as analytically intractable
\cite{BoermaTsyvinskiZimin22+,CarlierDupiusRochetThanassoulis24,Kolesnikov23+}.
At least on the plane, we rebut this view by confirming for $a \ge \frac{11}{2} -\sqrt{2}$ the
solution recently hypothesized by McCann and Zhang \cite{McCannZhang23+}.  We
show how their solution can also be modified to accommodate smaller values of
$a$ and other convex, planar domains. We show that the nature of the bunching
undergoes unanticipated changes --- from absent to targeted to blunt --- as $a \ge 0$
is increased.  
{Here {\em blunt} bunching refers to bunches orthogonal to
the diagonal line of symmetry, and {\em targeted} bunching refers to all other one-dimensional bunches
\cite{BoermaTsyvinskiZimin22+}.}
We rigorously prove the support $\overline{Du({\Omega})}$ of the
unknown distribution of products consumed has a lower boundary which is concave
nondecreasing --- as the above-mentioned numerics and stingray description suggest --- and that all
products selected by more than one type of consumer lie on this boundary or its
reflection through the diagonal.

The problem \eqref{eq:monopolist} arises from the question of how a monopolist
should price goods for optimal profit in the face of information
asymmetry. Here is a simple derivation of \eqref{eq:monopolist}. We assume a
closed set of products $\Omega^* \subset \mathbf{R}^n$ where each coordinate represents some
attribute of the product, and an open set of consumers $\Omega \subset \mathbf{R}^n$ where
each coordinate represents some attribute of the consumer. Consumers are
distributed according to a Borel probability measure $\mu$ on $\Omega$; the set of such measures is denoted by $\mathcal{P}(\Omega)$.  The
monopolist's goal is to determine a price $v(y)$ at which to sell product
$y$ in a way which maximizes her profit. If consumer $x \in \Omega$ attains benefit $b(x,y)$ from  product $y \in \Omega^*$ then the consumer
will choose the product $y$ which maximizes his utility
\begin{equation}
  \label{eq:b-transform} u(x) := \sup_{y \in \Omega^*} b(x,y) - v(y).
\end{equation}
Provided it exists, we denote the $y$ which realizes this supremum by $Yu(x)$.
Assuming the monopolist pays cost $c(y)$ for product $y$, then the monopolist's
goal is to maximize {her} profit, the integral of price {she sells}  for minus cost {she pays},
\[ \int_{\Omega} [v(Yu(x)) - c(Yu(x))] \, d \mu(x). \]
The problem has been considered in this generality in e.g.~\cite{Basov05}
\cite{Carlier01} \cite{FigalliKimMcCann11} \cite{mccann2023c}, and for even more
general (non quasilinear) utility functions in \cite{NoldekeSamuelson18}, \cite{McCannZhang19}.
For this paper, to highlight the mathematical properties of most interest, we
adopt several standard simplifying assumptions proposed by Rochet and Chon\'e
\cite{RochetChone98}: that products lie in the nonnegative orthant
$\Omega^*=[0,\infty)^n$ and the monopolist's direct cost to produce them is quadratic
$c(y) = |y|^2/2$, furthermore, that consumers are uniformly distributed on their domain $\Omega \subset \R^n$ and their product preferences are bilinear $b(x,y) = x \cdot y$.  In this
case \eqref{eq:b-transform} implies $u$ is the Legendre transform of $v$ (and
thus a convex function), $Yu(x) \in \partial u(x)$, and when $\Omega \subset [0,\infty)^n$ the
Monopolist's goal becomes to maximize
\[ \int_{\Omega} \left( x \cdot Du(x) - u(x) - \frac{|Du(x)|^2}{2} \right) \, dx,\]
over nonnegative convex functions which, up to an irrelevant constant, is the
problem \eqref{eq:monopolist}. Since convex functions are differentiable almost
everywhere the integrand is well-defined. The nonnegativity constraint on $u$
represents the additional requirement that $v(0)=0$, meaning consumers need not
consume if the monopolist raises prices too high, or equivalently, are always
free to pay nothing by choosing the zero product as an outside option.

This problem was first considered by Mussa and Rosen in the one-dimensional
setting \cite{MussaRosen78}, (after related models of taxation
\cite{Mirrlees71}, matching \cite{Becker73}, and signaling \cite{Spence73}
were introduced and analyzed by Mirrlees, Becker and Spence). The
multidimensional problem was considered by Wilson \cite{Wilson93} and Armstrong
\cite{Armstrong96}, while our formulation above is essentially that of Rochet
and Chon\'e \cite{RochetChone98}.  Although this model is of significant
importance to economists it presents serious mathematical difficulties. Indeed,
were there only the nonnegativity constraint in \eqref{eq:monopolist} we would have a variant of the
obstacle problem \cite{Caffarelli98,Figalli18}; with the \textit{nonnegativity
and convexity constraint} we have a free boundary problem for three different
regions. In the region where the function is positive
but not strictly convex, the fundamental tool of two-sided perturbation by an
arbitrary test function is no longer applicable. As a result, until recent work
of the first and third authors \cite{McCannZhang23+} it has not even been
possible to write down the Euler--Lagrange equation in the region of nonstrict convexity. Despite
this, other aspects of the problem have been studied, notably by Rochet and
Chon\'e \cite{RochetChone98}, who derived a necessary and sufficient condition
for optimality in terms of convex-ordering between the positive and negative
parts of the variational derivative of the objective functional  conditioned on product selected, Basov
\cite{Basov05} who advanced a control theoretic approach to such problems,
Carlier \cite{Carlier01,Carlier02} who considered existence and first-order
conditions for the minimizer, and Carlier and Lachand-Robert
\cite{CarlierLachand-Robert01a,CarlierLachand--Robert01,CarlierLachandRobert08}
who studied regularity and gave a description of the polar cone. 

In this paper we prove results of mathematical and economic interest. We invoke
tools from diverse areas of mathematics:  the theory of sweeping and convex orders of
measures, Monge--Amp\`ere equations, regularity theory for the obstacle problem, and the theory of optimal transport (which has deep
links to the Monopolist's problem). We also indicate a striking
connection to the classical obstacle problem: Locally the minimizer $u$
solves an obstacle problem where the obstacle is the minimal
convex extension of $u$ from its region of nonstrict convexity.  We now outline
our results.

\subsection{Multidimensional results}

If $u$ solves \eqref{eq:monopolist} and $\Omega$ is a convex open subset of
$\mathbf{R}^n$ it is 
known from the work of {Rochet--Chon\'e
\cite{RochetChone98}, 
Carlier--Lachand-Robert \cite{CarlierLachand--Robert01},
and Caffarelli and Lions \cite{CaffarelliLions06+} 
that $u \in
(C^{1,1}_{\text{loc}}\cap C^{0,1})(\Omega)$;
see \cite{mccann2023c}\cite{ChenFigalliZhang26+} and Remark \ref{R:ChenFigalliZhang}.

{
To allow for the possibility $u \not\in C^1(\overline\Omega)$
, let $u$ also denote the minimal convex extension of $u$ from $\Omega$ to all of $\R^n,$ and set
\begin{align*}
\p u &:= \{(x,y) \in \R^{2n} ; u(z) \ge u(x) + y \cdot (z-x) \quad \forall z \in \R^n\},
\\ \p u (x) &:= \{y \in \R^n ; (x,y) \in \p u\},
\\ \dom Du &:= \{ x \in \R^n ; \#(\p u(x))=1\}.
\end{align*}
Note $u$ is differentiable at $x$ if and only if $x \in \dom Du$. Since it is not clear that $\dom Du$ includes much of
$\p \Omega$, 
we extend $Du$ to ($\mathcal H^{n-1}$ almost all of) 
$\p \Omega$ as follows.  Let
$\dom D_{\p\Omega} u \subset \p \Omega$ denote the set of differentiability points of $\p \Omega$ where the restriction $u|_{\p \Omega}$ of $u$ to $\p \Omega$ is differentiable on $\p \Omega$, and set 
\begin{align}
\label{dom D^-u}
\dom D^-u &:= (\overline\Omega \cap \dom Du) \cup \dom D_{\p \Omega}u,
\\ D^-u(x) &:= 
\begin{cases}
    y(x) & {\rm if}\ x \in \dom D_{\p \Omega}u
\\ Du(x) & {\rm if}\ x \in \dom D^-u \setminus \dom D_{\p \Omega}u
\end{cases}
\label{D^-u}
\end{align}
where $\lambda(x)\ge 0$ and $y(x)$ are defined by
$\p u(x) = [y(x),y(x) + \lambda(x) \mathbf n]$ and
$\mathbf n$ is the outer normal to $\p \Omega$
at each $x \in \dom D_{\p \Omega} u$.
By Appendix~\ref{sec:rochet-chones-use} or \cite[Lemmas 1.6 and 2.2]{GangboMcCann00},
$D^-u$ is a well-defined Borel selection from $\p u$ and also coincides $\mathcal H^{n-1}$-a.e. on $\p \Omega$ with the inner boundary trace of the $BV$ function 
$Du:\Omega \longrightarrow \R^n$ from \cite[Theorem 3.77]{AmbrosioFuscoPallara00}.

As a non-tangential limit, $D^-u$ depends only on the local behaviour of $u$ on $\Omega$,
(whereas the outer boundary trace $D^+u(x):= D^-u(x) + \lambda(x)\mathbf n$ may depend on the behaviour of $u|_\Omega$ near the entire boundary facet of $\p \Omega$ containing $x$;
of course $D^\pm u=Du$ and $\dom D^-u=\overline \Omega$ when $u \in C^1(\overline \Omega)$).  
We define
an equivalence relation on $\dom D^-u$,  whose equivalence class 
for each $x_0 \in \dom D^-u$ is given by the convex set
\begin{align} 
\tilde{x_0} &:=  (D^- u)^{-1}(D^-u(x_0)) \subset \overline \Omega.
\label{equivalence classes}
\end{align}
In other words, $x_0 \sim x_1$ are equivalent if and only if $D^-u(x_0)=D^-u(x_1)$.}
We call an
equivalence class trivial if 
$\tilde{x} = \{x\}$. 
We call equivalence classes {\em leaves},
since they foliate the interior of $\Omega_i$.
They are also called {\em isochoice
sets} \cite{ChiapporiMcCannPass17} or {\em bunches} if nontrivial \cite{RochetChone98}. We
also call one-dimensional leaves {\em rays}. We set 
\begin{equation}
  \label{stratify} \Omega_{i} = \{x \in   \dom D^-u \subset \overline\Omega; 
  \tilde x \text{ is ($n-i$)-dimensional} \}.
\end{equation}
In particular --- outside of $\overline\Omega \setminus \dom D^-u$ (which is $\mathcal H^{n-1}$ negligible) --- $\Omega_n$ consists of all points at which $u$ is
strictly convex and $\Omega_0$ consists of all points 
lying in the closure of
an open subset of $\Omega$ on which $u$ is affine. These disjoint sets partition 
$\dom D^-u \subset \overline \Omega$,  
and our first result describes the qualitative behavior in each
set. 

\begin{theorem}[Partition into foliations by leaves that extend to the
boundary]\label{thm:structure-convexity} 
Let $u$ solve 
\eqref{eq:monopolist}
where $\Omega \subset \R^n$ is bounded, open and convex and set 
$Z:= \{x \in {\overline \Omega} ; u(x)=0\}$. 
Then $\Omega \subset \Omega_0 \cup \ldots \cup \Omega_n  =\dom D^-u $ and
\begin{enumerate}
  \item {either the closed convex set 
  $Z=\overline \Omega_0$ and 
  hence has positive volume,  or else
$\Omega_0$ is empty
and $Z \subset \p \Omega$ has positive area $\mathcal{H}^{n-1}(Z)>0$;}
  \item $\Omega_1,\dots,\Omega_{n-1}$ are a union of equivalence classes {$\tilde x$ 
  on which $u$ is
affine and the closure $\overline {\tilde x}$ 
of $\tilde x$ 
intersects the fixed} boundary $\partial \Omega$;
  \item $\Omega_n \cap \Omega$ is an open set on which $u \in C^\infty(\Omega_n)$ solves $\Delta u =n+1$.
\end{enumerate}
\end{theorem}

Note Theorem \ref{thm:structure-convexity}(1) refines Armstrong's result on the desirability of exclusion \cite[Proposition 1]{Armstrong96}, 
by relaxing strict convexity of $\Omega$ and showing that if $Z=u^{-1}(0)$ fails to have positive volume, then it lies on the boundary and has positive area;
c.f. \cite[Theorem 4.3]{FigalliKimMcCann11}.

\begin{remark}[An analytic interface]\label{R:analyticity}
 Any portion of $\p \Omega_0$ lying outside $\overline{\Omega_1 \cup \cdots \cup \Omega_{n-1} \cup \p\Omega}$ is
smooth --- in fact locally analytic --- by the theory of the obstacle problem
\cite{Caffarelli77,Caffarelli98}. Indeed, on any open ball in $\Omega$ exhausted by
$\Omega_0$ and $\Omega_n$, $u$ is a convex solution of $ \frac1{n+1}\Delta u =
1_{\{u>0\}}(x)$. Thus $\Omega_0$ is convex, and subsequently has Lipschitz boundary
which improves to locally analytic by the regularity for the obstacle problem
(see also \cite[Theorem 7.3]{Figalli18}).  Theorem
\ref{thm:description-on-square} 
below gives examples of square domains
in the plane for which this portion is nonempty.
\end{remark}

It is immediate from the definition that $\Omega_i$ for $1 \leq i \leq n-1$
are a union of nontrivial equivalence classes; the key conclusion is these
extend to the boundary (i.e. if $x \in \Omega_i$ for $1 \leq i \leq n-1$ then
$\overline{\tilde{x}} \cap \partial \Omega \neq \emptyset$).
 This offers some hope that the interior solution can be inferred from its boundary behaviour.
The PDE in point (3) has been considered in more
generality by Rochet and Chon\'e \cite{RochetChone98}.
Note the economic interpretation of (1): no bunches of positive measure are
sold {apart from} the null product; as we shall see during the proof of (2), any
product sold to more than one consumer lies on the boundary of the set of
products sold. Thus the entire interior of the set of products sold consists of
individually customized products.
Our proof of Theorem \ref{thm:structure-convexity} and some subsequent results requires a new proposition asserting that a.e. on the boundary of a convex domain $\Omega$,
the minimizer of \eqref{eq:monopolist} 
satisfies the boundary condition 
$({D^{ +}u}(x)-x) \cdot  \mathbf{n} \geq 0$, 
 where $\mathbf{n}$ is the outer unit normal to $\Omega$. 
Established in Proposition \ref{prop:neumann-sign}  (see Proposition \ref{P:neumann-sign-} for replacing $D^+u$ by $D^-u$), it can be interpreted to mean that the normal component of any boundary distortion in product selected can never be inward.

\subsection{Planar results}

The remainder of our results are restricted to the planar case $\Omega \subset
\mathbf{R}^2$. Our planar results are conditional on the following standing assumption:
\begin{assumption}\label{a:c1}
  When $\Omega \subset \mathbf{R}^2$ is convex the solution $u$ of \eqref{eq:monopolist} satisfies 
\begin{equation}\label{Assumption 1}
u \in C^1(\overline{\Omega}).
\end{equation}
\end{assumption}
See Remark \ref{R:ChenFigalliZhang} for a discussion of recent work which aims to establish this assumption.{}

In this case  $\overline \Omega = \dom D^-u$
and $D^\pm u = Du$,} so the equivalence classes take the form
\begin{align} 
\tilde{x_0} &= \{x \in \overline{\Omega} ; u(x) = p_{x_0}(x)\}
\label{support plane} 
\\ \nonumber \text{where}\quad 
p_{x_0}(x)&= u(x_0) + Du(x_0)\cdot(x-x_0), 
\end{align}
hence are compact and partition all of $\overline \Omega$.
Theorem \ref{thm:structure-convexity} provides a complete
description of the solution in $\Omega_0$ and $\Omega_2$: it remains to better understand
the behavior of the solution in $\Omega_1$ as well as the properties of the domains
$\Omega_1$ and $\Omega_2$, (noting $\Omega_0$ is, by Theorem \ref{thm:structure-convexity}, a
closed convex set).

By Theorem \ref{thm:structure-convexity}, the free boundary $\Gamma := \partial \Omega_1 \cap \partial \Omega_2
\cap \Omega$ between $\Omega_1$ and $\Omega_2$ consists only of points in rays which
extend to $\partial \Omega$. For our study of the free boundary we will assume we are working at $x_1 \in \Gamma $, with $\tilde{x_1} \cap \partial \Omega = \{x_0\}$ and assume a $C^{1,1}$ estimate in a neighbourhood of this ray.

\begin{assumption}\label{a:localc11}
 We say the minimizer of \eqref{eq:monopolist} satisfies a $C^{1,1}$ assumption in a neighbourhood of a ray $\tilde{x_0}$ provided there is $\epsilon > 0$ such that 
 on \[ \mathcal{N}_\epsilon(\tilde{x_0}):= \{x \in \overline{\Omega} ; \text{dist}(x,\tilde{x_0}) < \epsilon\},\] there holds
  \begin{equation}\label{Assumption 2}
  \Vert u \Vert_{C^{1,1}(\mathcal{N}_\epsilon(\tilde{x_0}))} \leq C.      
  \end{equation} 
\end{assumption}
All our free boundary results are stated in a neighbourhood of a ray $\tilde{x_0}$  and are conditional on Assumption \ref{a:localc11}. 

In \S\ref{sec:neum-estim} we prove that the Neumann condition
\begin{equation}
  \label{Neumann} (Du(x_0)-x_0) \cdot \mathbf{n} = 0,
\end{equation}
where $\mathbf{n}$ is the outer unit normal to the fixed boundary $\p \Omega$, can be used to characterize the
presence of these rays.  Namely if \eqref{Neumann} is not satisfied at $x_0 \in \partial
\Omega$ then $\tilde{x_0} \ne \{x_0\}$ (and  
nearby rays foliate the interior of $\Omega_1$ locally).  Conversely, if $x \in \partial \Omega$ lies in a boundary
neighbourhood $B_\epsilon(x) \cap \partial \Omega$ on which the Neumann condition \eqref{Neumann} is
satisfied and a $C^{1,1}$ estimate holds then $x$ is a point of strict convexity for $u$.  The remaining case ---
rays $\tilde x_0 \ne \{x_0\}$ which satisfy \eqref{Neumann} --- is subtle: we call
such rays {\em stray} and conjecture the union $ \mathcal S$ of stray rays has zero area in
general, Lemma~\ref{L:Neumann unimodality} shows convex polygonal domains $\Omega \subset [0,\infty)^2$ admit at most countably many stray rays;
the squares $\Omega=(a,a+1)^2$ admit none.

Let $x_1$ be a point in the free boundary $\Gamma$ which, necessarily, lies on the ray $\tilde{x_1}$. Let $x_0 \in \tilde{x_1} \cap \partial \Omega$ be the {only fixed-}boundary endpoint of $\tilde{x_1}$ provided by Theorem \ref{thm:structure-convexity} and let 
\eqref{Assumption 2} hold. Note that if $\partial \Omega$ is $C^1$ in a neighbourhood of $x_0$ and $(Du(x_0) - x_0)\cdot \mathbf{n} > 0$, the same inequality holds for all $x \in B_\epsilon(x_0) \cap \partial \Omega$ and thus such $x$ are also the boundary endpoints of nontrivial rays. In this case 
we call $x_1$ a {\em tame} free boundary point (and $\tilde x_1$ a {\em tame} ray).   We denote the set of tame free boundary points by $\mathcal{T} \subset \Gamma = \Omega \cap \p \Omega_1 \cap \p \Omega_2$ {(and reiterate that
$\#(\tilde x_1 \cap \p \Omega) = 1$ 
for all $x_1 \in \mathcal T$).} 

\begin{theorem}
[Regularity results for the 
free boundary
]
\label{thm:fb-regularity}
Let $u$ solve \eqref{eq:monopolist} where $\Omega \subset \R^2$ is  bounded, open, and
convex.
If
$x_1 \in \mathcal{T}$ with $\tilde x_1 \cap \p \Omega =\{x_0\}$  
and $\p \Omega$ is 
smooth at $x_0$ then: there exists 
$\epsilon > 0$ such that 
\\
(1) $\Gamma:= \p \Omega_1 \cap \p\Omega_2 \setminus \p \Omega$ has Hausdorff dimension less than $2$ in $B_\epsilon(x_1)$; \\
(2) the function $D(x):= \text{diam}(\tilde{x})$ is continuous in $B_\epsilon(x_0)\cap \p \Omega$; \\
(3) if $A \subset \p\Omega$ denotes the closure of all points where $D$ attains its local maxima, then
$\Gamma \cap \{x' \in \tilde x : x \in B_\epsilon(x_0) \cap \p\Omega \setminus A\}$ is Lipschitz;
\\
(4) if $D$ is Lipschitz on $B_\epsilon(x_0)\cap\p\Omega$, 
then a bootstrapping procedure
yields $\Gamma \cap B_\epsilon(x_1)$ is a $C^\infty$ curve and $u \in C^\infty(B_\epsilon(x_1) \cap \text{int } \Omega_1)$.
\end{theorem}

Establishing the Lipschitz regularity of $D$, which
permits the above-mentioned bootstrapping to a $C^\infty$ free boundary, remains an interesting open 
problem.\footnote{Subsequently however, 
Chen--Figalli--Zhang and McCann--O'Brien--Rankin 
built independently on the developments below 
to show locally that
the tame free boundary $\mathcal T$ 
is continuously differentiable 
\cite{ChenFigalliZhang26+}
(outside a discrete set,
\cite{McCannOBrienRankin26+} showed $C^{0,\alpha}$ smoothness
for all $0<\alpha<1$),
and moreover, 
apart from a closed nowhere dense subset,  is $C^\infty$-smooth \cite{McCannOBrienRankin26+}.
}

\begin{remark}[Lipschitzianity, convexity, and smoothness]
\label{R:convex?}
Note the Lipschitz requirement on 
$D|_{B_\epsilon(x_0) \cap 
\p \Omega}$ from (4)
is not necessarily satisfied even when the corresponding portion of $\Gamma$ lies in a Lipschitz submanifold given by (3).  If $\{ x' \in \tilde x : x \in B_\epsilon(x_0) \cap \p\Omega\}$ 
happens to be convex this distinction disappears for every smaller value of $\epsilon$; simulations \cite{Mirebeau16} suggest this occurs in the square examples from Theorem~\ref{thm:description-on-square}.
Thus if the region $\Omega_1^+$ depicted in 
Figure~\ref{fig:three-behaviours} is convex,  as Mirebeau's simulations lead us to conjecture, Theorem~\ref{thm:fb-regularity} guarantees the curved portion of its boundary is smooth (away from its endpoints).  See alternately
Remark \ref{R:Hoelder to smooth} below.
\end{remark}

Theorem \ref{thm:fb-regularity} is proved using new coordinates 
for the problem, new Euler--Lagrange equations, and a new observation: That in a neighbourhood of a tame free boundary point the difference between the minimizer $u$ and the minimal convex extension of $u|_{\Omega_1}$ solves the classical obstacle problem.  A priori, the obstacle is  $C^{1,1}_{loc}(\Omega)$, i.e. has a merely $L^\infty$ Laplacian, and thus, without first improving the regularity of $u|_{\partial \Omega_1}$, the above results are the best one can obtain from the 
theory of the obstacle problem. 

\subsection{Bifurcations along the family of symmetrical squares}

For general convex domains it is difficult to study the structure of the
stray set $\mathcal{S} = \Gamma \setminus \mathcal{T}$ which may include points in the relative interior of
rays.
Except on polygonal domains, we have not ruled out the possibility that the relative boundary of
\[ \partial\Omega_{\neq} := \{x \in \partial \Omega ; (Du-x) \cdot \mathbf{n} \neq 0\},\]
in $\p \Omega$ might have positive $\mathcal{H}^1$-measure and the corresponding free boundary
be nonsmooth; {(nor have  \cite{ChenFigalliZhang26+}\cite{McCannOBrienRankin26+})}. However, in specific cases, a more complete description is
possible. For example, for $C^{1,1}(\overline{\Omega})$ solutions on the squares $\Omega = (a,a+1)^2$,  
 we establish unimodality of $(Du-x)\cdot \mathbf n$ and use it to show $\partial\Omega_{\neq}$ is a
single connected component of $\partial \Omega$, 
{a property which does not occur in general and 
may be lost after small perturbations}. In fact, 
we are able to provide the first explicit 
description of $C^{1,1}$ solutions
on $\Omega=(a,a+1)^2$, including an unexpected trichotomy for $a=0$, $a$ sufficiently
small, and $a$ sufficiently large. To describe the solution we label the edges
of $\Omega$ by their compass direction and set
\begin{align*} 
&\Omega_{N} =[a,a+1]\times\{a+1\} &&\Omega_{E} = \{a+1\}\times [a,a+1] 
\\ &\Omega_{W} = \{a\} \times [a,a+1] &&\Omega_S = [a,a+1]\times\{a\}.
\end{align*}

The minimizer is described by the following bifurcation theorem (illustrated schematically in Figure \ref{fig:three-behaviours}).

\begin{figure}
     \centering
  \begin{subfigure}[b]{0.9\textwidth}
         \centering
         \begin{tikzpicture}[thick,scale=6]
		\draw[domain=0:1] plot (0, \x) ;
		\draw[domain=0:1] plot (\x, 1) ;
		\draw[domain=0:1] plot (1,\x) ;
		\draw[domain=0:1] plot (\x, 0) ;
		\node at (0.2, 0.2) {$\Omega_0$};
		\node at (0.6, 0.6) {$\Omega_{2}$};

\draw plot coordinates {(0., 0.4) (0.005, 0.4) (0.01, 0.399998) (0.015, 0.399993) (0.02,  0.399983) (0.025, 0.399967) (0.03, 0.399944) (0.035,  0.399911) (0.04, 0.399867) (0.045, 0.39981) (0.05,  0.399739) (0.055, 0.399653) (0.06, 0.399549) (0.065,  0.399427) (0.07, 0.399284) (0.075, 0.399119) (0.08,  0.39893) (0.085, 0.398716) (0.09, 0.398475) (0.095,  0.398206) (0.1, 0.397906) (0.105, 0.397574) (0.11,  0.397208) (0.115, 0.396806) (0.12, 0.396367) (0.125,  0.395889) (0.13, 0.39537) (0.135, 0.394807) (0.14,  0.3942) (0.145, 0.393545) (0.15, 0.392841) (0.155,  0.392086) (0.16, 0.391278) (0.165, 0.390413) (0.17,  0.389491) (0.175, 0.388508) (0.18, 0.387461) (0.185,  0.386349) (0.19, 0.385167) (0.195, 0.383914) (0.2,  0.382586) (0.205, 0.38118) (0.21, 0.379693) (0.215,  0.37812) (0.22, 0.376458) (0.225, 0.374703) (0.23,  0.372851) (0.235, 0.370897) (0.24, 0.368835) (0.245,  0.366661) (0.25, 0.364368) (0.255, 0.361951) (0.26,  0.359402) (0.265, 0.356715) (0.27, 0.353881) (0.275,  0.350891) (0.28, 0.347735) (0.285, 0.344403) (0.29,  0.340883) (0.295, 0.337161) (0.3, 0.333222) (0.305,  0.32905) (0.31, 0.324624) (0.315, 0.319922) (0.32,  0.31492) (0.325, 0.309586) (0.33, 0.303886) (0.335,  0.297778) (0.34, 0.291212) (0.345, 0.284124) (0.35,  0.276439) (0.355, 0.268057) (0.36, 0.258851) (0.365,  0.248648) (0.37, 0.237207) (0.375, 0.224174) (0.38,  0.20899) (0.385, 0.190684) (0.39, 0.167281) (0.395,  0.133329) (0.4, 0.)};
                
	\end{tikzpicture}
         \caption{$a=0$}
         \label{fig:a0}
     \end{subfigure}      
     \begin{subfigure}[b]{0.9\textwidth}
         \centering
	\begin{tikzpicture}[thick,scale=6]
		\draw[domain=0:1] plot (0, \x) ;
		\draw[domain=0:1] plot (\x, 1) ;
		\draw[domain=0:1] plot (1,\x) ;
		\draw[domain=0:1] plot (\x, 0) ;
           %     		\draw[domain=0:1] plot (\x, \x) ;
                                \draw[domain=0:0.2] plot (\x, -0.36397*\x+0.5) ; %line at -20 deg (tan 20 = 0.36397)
                                \draw[domain=0:0.2] plot ( -0.36397*\x+0.5,\x) ; %mirrored

                                \draw (0.4274,0.2) arc [start angle=20, end angle=70, radius=3.8mm]; %curved \Omega_0 boundary

                                \draw[domain=0:0.2] plot(\x, -4.31985*\x*\x+0.6); %upper parabola
                                \draw[domain=0:0.2] plot(-4.31985*\x*\x+0.6,\x); %mirrored
		\node at (0.2, 0.2) {$\Omega_0$};
		\node at (0.06, 0.52) {$\Omega_{1}^{-}$};
		\node at (0.53, 0.05) { $\Omega_{1}^{+}$};
		\node at (0.6, 0.6) {$\Omega_{2}$};
	\end{tikzpicture}
         \caption{$a>0$ sufficiently small}
         \label{fig:asmall}
     \end{subfigure}
     
 \begin{subfigure}[b]{0.9\textwidth}
         \centering
	\begin{tikzpicture}[thick,scale=6]
		\draw[domain=0:1] plot (0, \x) ;
		\draw[domain=0:1] plot (\x, 1) ;
		\draw[domain=0:1] plot (1,\x) ;
		\draw[domain=0:1] plot (\x, 0) ;
		\draw[domain=0:0.34] plot (\x,  -1.2+2*sqrt{1-\x*\x*4}) ;
		\draw[domain=0:0.34] plot (-1.2+2*sqrt{1-\x*\x*4}, \x);
		\draw[domain=0:0.3] plot (\x, 0.3-\x);
		\draw[domain = 0:0.68] plot (\x, 0.68-\x);
		
		\node at (0.09, 0.09) {$\Omega_0$};
		\node at (0.23, 0.28) { $\Omega_{1}^{0}$};
		\node at (0.13, 0.65) {$\Omega_{1}^{-}$};
		\node at (0.65, 0.13) { $\Omega_{1}^{+}$};
		\node at (0.6, 0.6) {$\Omega_{2}$};
	\end{tikzpicture}
         \caption{$a \ge \frac{11}{2} -\sqrt{2}$}
         \label{fig:alarge}
     \end{subfigure}          
     \caption{ Bifurcation to targeted then blunt bunching
 (Thm \ref{thm:description-on-square})     as distance $a\ge0$ 
 of $\Omega=(a,a+1)^2$ to zero is increased.}
        \label{fig:three-behaviours}
\end{figure}
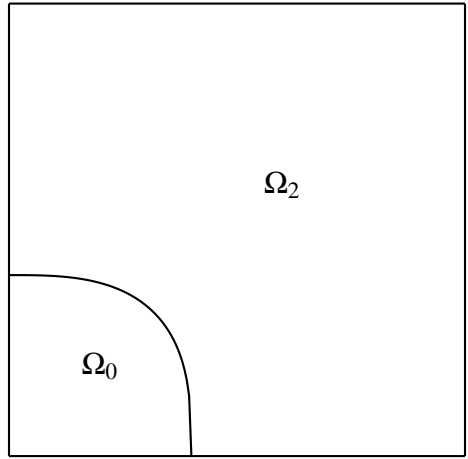
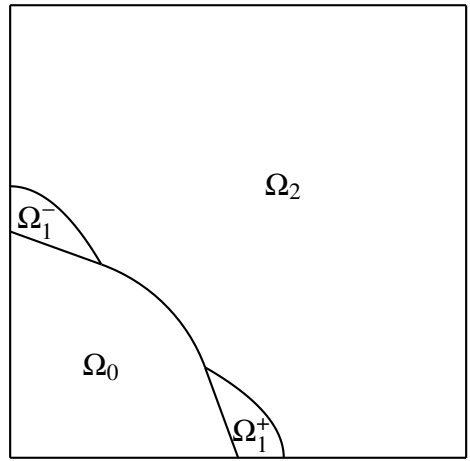
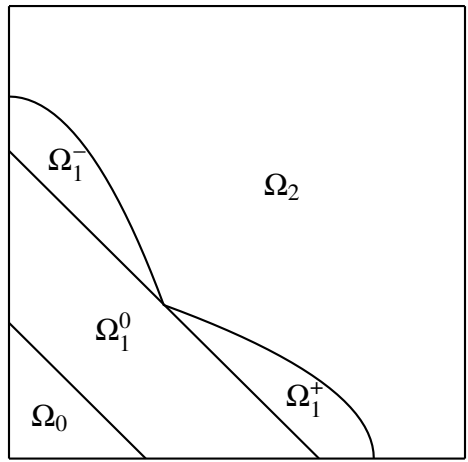

\begin{theorem}[Blunt bunching is a symptom of a seller's market]
\label{thm:description-on-square}
  Let $u \in C^{1,1}(\overline{\Omega})$ solve \eqref{eq:monopolist} with $\Omega = (a,a+1)^2$ where $a \geq 0$. Then
  \begin{enumerate}
  \item $\Omega_0$ is a convex set which includes a neighbourhood of $(a,a)$ in $\overline \Omega$.
  \item   The 
set $\Omega_1^0 := \{x \in \Omega_1 ; \#(\tilde x \cap \p\Omega)=2\}$ 
  of rays intersecting the boundary in precisely two points is connected.  It is nonempty
when $a\ge \frac{11}{2} -\sqrt{2}\approx 4.1$.  All rays in $\Omega_1^0$ are orthogonal to the diagonal
and have one endpoint on $\Omega_{W}$ and the other on $\Omega_{S}$.  On the other hand there
is $\epsilon_0>0$ such that  $ \Omega_1^0$ is empty when $a < \epsilon_0$.

  \item For all $a > 0$ there are exactly two disjoint  connected components of $
  {(\Omega \cap \Omega_{1}) \setminus \Omega_1^0}$.  In these regions, each ray has only one endpoint $x_0 \in \p \Omega$ on the boundary; it lies in $\Omega_S \cup \Omega_W$, violates the Neumann condition \eqref{Neumann},  and the solution $u$ is described by the Euler--Lagrange equations of McCann and Zhang \cite{McCannZhang23+}; c.f. \eqref{D:m,b}--\eqref{extras}. 
  For $a = 0$, $\Omega_1$ is empty.
  \item  The set $\Omega_2$ of strict convexity of $u$ contains 
  $B_{1/3}(a+1,a+1) \cap (\Omega_E\cup \Omega_N)$  
  and the Neumann condition \eqref{Neumann} holds along $\Omega_N \cup \Omega_E$ and also at all $x_0 \in \Omega_2 \cap \p \Omega$ apart from the 3 vertices. 
  \end{enumerate}
\end{theorem}

The following corollary may be of purely mathematical interest from the point of view of calculus of variations and partial differential equations.
The  smoothness asserted follows from 
{Remark \ref{R:analyticity}.}

\begin{corollary}[Convexity of solution to, and contact set for, an obstacle problem]
For $\Omega = (-1,1)^2$,  the minimizer of $L(u)$ over non-negative functions $0 \le u \in W^{1,2}(\overline{\Omega})$ is convex.  
Its zero set is smooth, convex, has positive area, and is compactly contained in the centered square $\Omega$.
\end{corollary}

\begin{remark}[Concave nondecreasing profile of stingray's tail]
Numerical simulations of the square example show the region $Du(\Omega)$ of products
consumed to be shaped like a stingray, e.g. Figure 1 of \cite{CarlierDupiusRochetThanassoulis24}.  Theorem \ref{thm:description-on-square} combines with Lemma
\ref{lem:boundary-portions-convex} below to provide a rigorous proof that the
lower edge of stingray is concave non-decreasing --- as the simulations suggest ---
while Theorem \ref{thm:structure-convexity} shows that every product selected
by more than one type of consumer lies on this boundary or its mirror image
across the diagonal.
\end{remark}

\begin{remark}[Absence and ordering of blunt vs targeted bunching]
The potential absence of blunt bunching from the square example 
--- established on a nonempty interval $a \in (0,\epsilon_0)$ by the preceding theorem
--- has been overlooked in all previous investigations that we are aware of.  It can be understood as a symptom of a buyers' market,
in which a lack of enthusiasm on the part of qualified buyers incentivizes the monopolist to sell to fewer buyers but cater more to the tastes of those who do buy. The persistence of targeted bunching $\Omega_1^\pm \ne \emptyset$ for all $a>0$ reflects the need to transition continuously from vanishing Neumann condition \eqref{Neumann} 
--- satisfied on the customization region 
$\Omega_2 \cap (\Omega_S \cup \Omega_W)$ where $u$ is strictly convex --- to the uniformly positive Neumann condition $(Du -x)\cdot \mathbf n=a$ on the exclusion region $\Omega_0$ where $u$ vanishes,  in light of the 
assumed $C^1(\overline{\Omega})$ regularity \eqref{Assumption 1}.
Such bunching is neither needed nor present when $a=0$:
in this case $x_0 \cdot \mathbf{n}=a$ on $\Omega_S \cup \Omega_W$ shows the Neumann conditions in $\Omega_0$ and $\Omega_2$ coincide.
When the blunt bunching region $\Omega_1^0$ is present, our proof of Theorem \ref{thm:description-on-square} shows it separates $\Omega_0$ from $\Omega_1^\pm$, which in turn separate all but one point of $\Omega_1^0$ from $\Omega_2$.  In particular, blunt bunching implies $\Omega_0$ is a triangle, which is exceedingly rare in its absence.
\end{remark}

The hypothesis $u\in C^{1,1}(\overline \Omega)$ allows us to
characterize the solution to \eqref{eq:monopolist}
on $\Omega=(a,a+1)^2$ for every value of $a \ge 0$.  Namely $u \in \mathcal{U}$ minimizes
\eqref{eq:monopolist} if and only if (A) bunching is absent
($\Omega_1 =\emptyset$, as for $a=0$), in which case $u$ solves $\frac13 \Delta u = 1_{\{u>0\}}$, i.e
the classical obstacle problem
\cite{PetrosyanShahgholianUraltseva12,Figalli18,RealRosOton22} and $(Du-x) \cdot \mathbf{n} = 0$ on $\partial \Omega$ (Figure \ref{fig:a0}),  or (B) bunching is present but blunt bunching is absent, ($\Omega_1^0 = \emptyset \ne \Omega_1$, as for $a \ll 1$) in which case we derive below necessary conditions, whose sufficiency can be confirmed as in \cite{McCannZhang23+} (Figure \ref{fig:asmall}), or (C) blunt bunching is present, ($\Omega_1^0 \ne \emptyset$, as for $a \ge \frac{11}{2} -\sqrt{2}$) (Figure \ref{fig:alarge}), in which case the sufficient conditions for a minimum established by two of us \cite{McCannZhang23+} are also shown to be necessary below. 

If instead (B) blunt bunching is absent but bunching is present,
$\Omega_1^0 = \emptyset \ne \Omega_1$, 
Theorem \ref{thm:description-on-square} asserts that $\Omega \cap \Omega_1 =
\Omega_1^+ \cup \Omega_1^-$ split into two connected components
\begin{align}
\label{Omega_1^pm} \Omega_1^\pm &:= \{(x_1,x_2) \in \Omega \cap \Omega_1 \setminus \Omega_1^0 : \pm (x_1-x_2) > 0 \},
\end{align} placed symmetrically below and above the diagonal.
The region $\Omega_1^-$ and its reflection $\Omega_1^+$ below the diagonal
are foliated by isochoice segments making continuously varying angles $\theta$ with
the horizontal.  The limit of these segments is a segment of length $R_0>0$
lying on the boundary of the convex set $Z= \Omega_0$, having endpoint $(a,h_0)$ and making angle $\theta_0
\in [-\pi/4,0)$ with the horizontal.

Fix any closed convex neighbourhood $\Omega_0$ of $(a,a)$ in $\overline\Omega$ which is reflection symmetric around the diagonal and contains such a segment in its boundary.
We describe the solution $u=u_1^-$ in $\Omega_1^-$ using an Euler-Lagrange equation \eqref{slope E-L} from
 \cite{McCannZhang23+}, rederived more simply in Section \ref{sec:rc-example} below. 
 Index each isochoice segment in $\Omega_1^-$ by its angle $\theta\in (-\frac{\pi}{4}, 
 0)$; (angles which are less than $-\pi/4$ or non-negative are ruled out in the proof of Theorem \ref{thm:description-on-square}).
Let $(a,{h(\theta)})$ denote its left-hand endpoint and parameterize the segment by 
distance ${  r} \in [0,R(\theta)]$ to this boundary point $(a,h(\theta))$. 
Along the hypothesized length $R(\theta)$ of this segment assume $u$ 
increases linearly with slope $m(\theta)$ and offset $b(\theta)$:
\begin{equation}\label{D:m,b}
u_1^-\Big((a,h(\theta)) + { r} (\cos \theta,  \sin \theta)\Big) = {m(\theta)}{  r} + {b(\theta)}.
\end{equation}
Given the initial (angle, height, length) triple $(\theta_0,h_0,R_0) \in [-\pi/4,0)\times (a,a+1) \times (0,\sqrt{2}/2)$ corresponding to the segment in $\Omega \cap \p\Omega_0$, and 
${R}: \left[\theta_0,0\right] \to \left[0, 1\right)$ 
piecewise Lipschitz  with $R(\theta_0)=R_0$, 
solve 
\begin{equation}\label{slope BC}
\textstyle  
m(\theta_0) = 0, \qquad m'(\theta_0) = 0\qquad  \mbox{\rm such that}
\end{equation}
\begin{equation}\label{slope E-L}
({m''(\theta) + m(\theta)} - {2R(\theta)})({m'(\theta)} \sin \theta - {m(\theta)} \cos \theta +a) = \frac32 {R^2(\theta) \cos\theta};
\end{equation}
then set
\begin{eqnarray}
\label{D:h}
{h(\theta)} &=& h_0+ \frac13 \int_{\theta_0}^{\theta} (m''(\vartheta) + m(\vartheta) - 2{  R(\vartheta)}) \frac{d\vartheta}{\cos \vartheta},
\\ 
{b(\theta)} &=& 
\int_{\theta_0}^{\theta} (m'(\vartheta) \cos \vartheta + m(\vartheta) \sin \vartheta) h'(\vartheta) d\vartheta.
\label{offset E-L}
\end{eqnarray}
	Given $(\theta_0,h_0,R_0)$ and $R(\cdot)$ {as above},  
	the triple $(m, b, h)$ satisfying 
	\eqref{slope E-L}--\eqref{offset E-L} exists and is unique {on the interval where $R(\cdot)>0$. Thus} 
	the shape of $\Omega_{1}^{-}$ 	
	and the value of $u_1^-$ on it will be uniquely determined by $\Omega_0$ and $R:\left[\theta_0,0\right] \to \left[0, 1\right)$.  We henceforth restrict our attention to choices of $\Omega_0$ and $R(\cdot)$ for which the resulting set $\Omega_1^-$ 
	lies above the diagonal.  
	In this case $\Omega_{1}^{+}$ and the value of $u = u_1^+$ on $\Omega_1^+$ are determined 
	by reflection symmetry $x_1 \leftrightarrow x_2$ across the diagonal.
	This defines $u=u_1$ on $\Omega_1$ and provides the boundary data on 
	$
	\p \Omega_1 \cap \p \Omega_2$ needed for the  mixed Dirichlet / Neumann boundary value problem for Poisson's equation,
 \begin{flalign}\label{mixedBVP}
	\begin{cases}
		\Delta u_2  = 3, & \text{ on } \Omega_2,\\
		(Du_2(x)-x)\cdot \mathbf{n} = 0, & \text{ on } \p \Omega_2 \cap \p \Omega,\\
		u_2 - u_1 = 0,  & \text{ on } \p \Omega_2 \cap \p \Omega_1, \\
        u_2 = 0 & \text{ on } \p \Omega_2 \cap \p \Omega_0,		\end{cases}
\end{flalign} 
 which determines $u=u_2$  on $\Omega_2:=\Omega \setminus (\Omega_0 \cup \Omega_1)$.
The duality discovered in \cite{McCannZhang23+}, implies that for at most one choice of  $\Omega_{0}$ and $R(\cdot)$ Lipschitz can convex $u$ (pieced together from  $u_1,u_2$ as above with $u_0:=0$ on $\Omega_0$) satisfy the supplemental Neumann conditions
 \begin{flalign}\label{extra}
        	 D(u_2-u_1) \cdot \hat{\mathbf n} =0, & \text{ on } \p \Omega_2 \cap \p \Omega_1
  \\ Du_2 \cdot \hat{\mathbf n} =0, & \text{ on } \p \Omega_2 \cap \p \Omega_0
  \label{extras}
  \end{flalign}
  required on the free boundaries (since $u\in C^1({\Omega})$ \cite{CarlierLachand--Robert01}\cite{RochetChone98}); here $\hat{\mathbf n}$ denotes the outer unit normal to $\Omega_2$ at $x \in \p \Omega_2$. In the course of proving Theorem~\ref{thm:description-on-square}
in Section \ref{sec:rc-example} below we complete this circle of ideas 
--- apart from the Lipschitz hypothesis 
on $R(\cdot)$ which Lemma \ref{L:R continuity} renders inessential{} --- 
by showing at least one such choice exists;  this choice uniquely solves \eqref{eq:monopolist} on the square in case (B).
In case (C), Theorem \ref{thm:description-on-square} 
 shows at least (and therefore exactly \cite{McCannZhang23+}) one choice exists satisfying
 the free boundary problem from  \cite{McCannZhang23+} in the analogous sense.
 
\subsection{Plan of the paper}

We conclude this introduction by outlining the structure of the paper. Section
\ref{sec:pert-ineq} contains preliminaries: the variational inequality
associated to \eqref{eq:monopolist}, some background on Alexandrov second
derivatives, and localization results of Rochet--Chon\'e.  In Section
\ref{sec:structure-convexity} we prove Theorem \ref{thm:structure-convexity}
by combining localization with 
perturbation techniques previously used to study the Monge--Amp\`ere
equation.
 In Section~\ref{sec:neum-estim} we quantify how at points of nonstrict
convexity the Neumann boundary condition fails to be satisfied.
Section~\ref{sec:coord-argum-two} and Section~\ref{sec:lipschitz-fb} 
establish Theorem~\ref{thm:fb-regularity} using techniques from the study of the
obstacle problem. Here we indicate a new connection to the classical obstacle
problem. Namely, that the Monopolist's problem gives rise to an obstacle
problem where the obstacle is the minimal convex extension of the function
defined on $\Omega_1$.  The proof of Theorem \ref{thm:description-on-square} is
completed in Section~\ref{sec:rc-example} using a case by case analysis based
on a careful choice of coordinates to deduce monotonicity of $(Du(x)-x)\cdot \mathbf{n}$ along the upper left boundary $\Omega_W \cup \Omega_N$ of the square. It confirms the economic intuition that the
degree to which product selection (hence bunching) is influenced by the market
presence of competing consumers decreases as we move away from the exclusion
region, i.e. from the lower left toward the upper right region of the square,
while on the other hand, increasing as we move the entire square of consumer
types away from the outside option by increasing $a \ge 0$.  We conclude with
an appendix 
{devoted to developing necessary background material, including the extension of Rochet and Chon\'e's localization results to $u \not\in C^1(\overline\Omega)$}.  Table
\ref{tbl:notation} contains a list of notation.

\begin{table}
  \begin{tabular}{ll}
    Notation & Meaning\\
    \hline
    $\Omega$ & A bounded open convex subset of $\mathbf{R}^n$.\\
    $\overline\Omega$ & The closure of $\Omega$.\\
    $\text{int } \overline \Omega$ & The interior of $\overline\Omega$.\\
    $\Omega^c$ & Set complement $\Omega^c:=\mathbf{R}^n \setminus \Omega$ of $\Omega$.\\
    $\mathbf n$ & Outer unit normal at a point where $\p \Omega$ is differentiable. 
    \\  $\text{spt}{f}$ & The support of $f$, i.e.\ $\text{spt}{f} = \text{closure}\{x ; f(x) \neq 0\}$.
    \\$ f|_\Omega$ & The restriction of $f$ to $\Omega$.      \\  $id$ & the identity map $id(x)=x$.
     \\  $\p u$ &   
$\p u := \{(x,y) \in \R^{2n}; u(z) \ge u(x) + (z-x)\cdot y \ \forall z \in \Rn\}$.
\\ $ \dom Du\subset \R^n$ & 
Set where minimal convex extension of $u$ is differentiable.
    \\ $ D^\pm u$ & Outer and inner extensions of $Du \in [BV(\Rn)]^n$ to $\p \Omega$,  \eqref{D^-u}.\\
    $ \dom D^- u$ & $\dom Du \cup \{\mbox{\rm differentiability points of 
$\p \Omega$ and $u|_{\p \Omega}$}
                        \}$, \eqref{dom D^-u}. \\
    $\tilde{x}$ & Bunch $\tilde {x} := \{z \in { \dom D^-u \subset \overline \Omega}; D^-u(z)=D^-u(x)\}$.\\
    $\ri(\tilde x)$ & The relative interior of the convex set $\tilde x$.\\
    $\Omega_i$ & Subset \eqref{stratify} of $\overline\Omega$ foliated by ($n-i$)-dimensional {bunches}.\\
    $\subset \subset$ & Compact containment.\\
    $f_{+}$ & The positive part of a function, $f_{+}(x) := \max\{f(x),0\}$.\\
    $\mu_{+}$ & The positive part of a measure $\mu$.
    \\$\mathcal{P}(\Omega)$ & The set of Borel probability measures on $\Omega$.\\
    $\mres$ & Measure restriction: $(\mu \mres A)(B) = \mu (A \cap B)$.\\
    $ \tilde \mu = F_\#\mu$ & Push-forward $\tilde \mu = \mu \circ F^{-1}$ of measure $\mu$ through $F$. \\
    $ \mu=\int \mu_y d\tilde \mu(y)$ & 
    Disintegration     of the measure 
    $\mu$ with respect to $F$, \eqref{B disintegrate-}.\\
    $\mathcal{H}^d$ & $d$-dimensional Hausdorff measure.\\
 $dx :=d\mathcal{H}^n(x)$ & $n$-dimensional Lebesgue measure, i.e. volume measure. \\
 $dS:=d\mathcal{H}^{n-1}$ & Surface area measure (or arclength in  special case $n=2$).
\\  $|A|$ &  Volume $|A|:=\mathcal H^n(A)$ of $A \subset \R^n$. 

  \end{tabular}
  \caption{Table of notation.}
  \label{tbl:notation}
\end{table}

\subsection{Note added in October 2026 revision}
\begin{remark}[On regularity up to the fixed boundary]\label{R:ChenFigalliZhang}
Although the minimizer of \eqref{eq:monopolist} has been claimed to satisfy $u \in C^1({\overline\Omega})$ 
for all convex bounded
$\Omega \subset \R^n$ \cite{RochetChone98} \cite{CarlierLachand--Robert01},
and it is true that $u \in C^{0,1}(\overline \Omega)$
(c.f.~Lemma \ref{L:Lipschitz}), 
these authors were unable to provide us with proof that the continuous differentiability of $u$  they establish in $\Omega$ extends to $\overline\Omega$ 
\cite{CarlierRochetChone}. 
This gap was noticed by Chen, Figalli, and Zhang \cite{ChenFigalliZhang26+},
who also provided examples of convex planar domains 
on which the modulus of continuity of $Du$ could be arbitrarily bad.

The original version of the present manuscript claimed to extend $C^{1,1}$ regularity of $u$
to (a) (smooth parts of) the boundary of a convex polyhedron $\Omega \subset \R^n$,  and  (b) to both ends of any ray having an endpoint in the interior of a more general convex domain $\Omega \subset \R^n$.  
Following our June 2026 revision, Lucas O'Brien noted a gap in our proof of the $C^{1,1}$ regularity for one ended rays and on polyhedra. Since then we have discovered new counterexamples 
that indicate both results (a) and (b) to be incorrect as stated. Namely,
on the triangle $\Omega$ having vertices $(c,0),(c,1)$ and $(c+1,0)$,  
with $c \gg 1$ sufficiently large, the $C^{1,1}$ norm of $u$ blows up along a sequence of one ended rays with endpoints in the left-hand side of the triangle, as we shall show elsewhere.
The $C^{1}(\overline{\Omega})$ and $C^{1,1}(\overline{\Omega})$ results asserted for planar domains by Chen, Figalli, and Zhang \cite[Theorems 1.1--1.3]{ChenFigalliZhang26+} also relied on our erroneous $C^{1,1}$ estimates and thus appear currently incomplete. 
Even for planar domains $\Omega \subset \R^2$, the best regularity we are presently aware of is therefore $u \in (C^{0,1} \cap C^{1,1}_{loc})(\Omega)$ \cite{Chone99}\cite{CaffarelliLions06+}\cite{mccann2023c}.
However, Section \ref{sec:rc-example} below considers only $C^{1,1}(\overline{\Omega})$ solutions --- and classifies these on the $a \ge 0$ squares $\Omega = (a,a+1)^2$.  In spite of the counterexamples mentioned above, we believe solutions on these squares are in fact $C^{1,1}(\overline{\Omega})$ and  aim to establish this in future work.

We nevertheless consider it a significant achievement that the present manuscript identifies $C^{1,1}$ boundary regularity of $u$ as 
a clear route to making our conditional planar results --- and those of \cite{McCannZhang23+} --- unconditional. The lack of dependence of the biLipschitz constant in Corollary~\ref{C:biLipschitz} on $\|u\|_{C^{1,1}(\overline \Omega)}$ suggests it may not be the only route however. 
\end{remark}

{\bf Statement on the use of AI:}  {These results have been obtained and this manuscript prepared without the use of AI.}

\section{Variational inequalities and Alexandrov second derivatives}
\label{sec:pert-ineq}

\subsection{Variational inequalities}
\label{sec:vari-ineq}

Our basic tools for studying the unique minimizer of the functional
\begin{equation}
  \label{eq:1}
  L[u] = \int_{\Omega} \left(\frac{1}{2}|Du-x|^2 + u \right)\, dx,
\end{equation}
over
\[ \mathcal{U} = \{ u : \overline{\Omega} \rightarrow \mathbf{R} ; \text{ $u$ is nonnegative and convex}\},\]
are the variational inequalities stated in the following lemma.  Because the
gradient of a convex function $u \in C^{0,1}(\Omega)$ has bounded variation, the
integration by parts against the gradients of Lipschitz test functions below is
justified, where $\Delta u$ is interpreted as a Radon measure vanishing outside $\Omega$,
and $D^-u$ is from \eqref{D^-u}. We have the equality $D^-u(x) \cdot \mathbf{n} =
\lim\limits_{t \downarrow 0}( u(x) - u(x-t\mathbf{n}))/t$ for $\mathcal{H}^{n-1}$ a.e. $x$ on $\partial \Omega$
and $\mathbf{n}$ the outer normal. Moreover, when $u \in C^{1,1}_{loc}(\Omega)$ --- as
{for the minimizer \cite{CaffarelliLions06+}\cite{mccann2023c} of
\eqref{eq:monopolist} --- we have $\Delta u$ absolutely continuous with respect to
Lebesgue}, and when $u \in C^1(\overline{\Omega})$ --- as assumed for $n =2$ (\eqref{Assumption 1} and Remark
\ref{R:ChenFigalliZhang}) --- we have $D^-u = Du$.

\begin{lemma}[Variational inequalities]
  \label{lem:variational-inequalities}
Let $u$ solve \eqref{eq:monopolist}. Let $\bar{u}\in \mathcal{U}$ be Lipschitz and $w=\bar u -u$. Then each of the following inequalities hold:
  \begin{align}
    \label{eq:perturb-0} 0 &\leq 
    L'_u(w) := \int_{\Omega}(n+1-\Delta u)  \, w dx + \int_{\partial\Omega}(D^- u-x) \cdot \mathbf{n} \, w dS,\\
    \label{eq:perturb-2} 0 &\leq L'_u(\bar u) =\int_{\Omega}(n+1-\Delta u) \bar{u} \, dx + \int_{\partial\Omega} (D^-u-x) \cdot \mathbf{n} \, \bar u dS,
  \end{align}
Moreover, if $Du \not \equiv D\bar{u}$ on a set of positive $\mathcal{H}^n$ measure, 
  then
\begin{align}
  \label{eq:perturb-1} 0 &< L'_{\bar u}(w) = \int_{\Omega}(n+1-\Delta \bar{u}) \, wdx + \int_{\partial\Omega}(D^-\bar{u}-x) \cdot \mathbf{n} \, wdS.
\end{align}
{Here $\Delta \bar u$ is 
a Radon measure 
on $\Omega$ and
$D^- u \cdot \mathbf n$ 
agrees with the one-sided derivative $\lim\limits_{t \downarrow 0}(u(x) - u(x-t\mathbf n))/t$ 
which exists by convexity of $u$.}
\end{lemma}

\begin{proof}
  We begin with \eqref{eq:perturb-0}. Let $u$ be the minimizer and observe $\mathcal{U}$ is convex. Thus for any $\bar{u} \in \mathcal{U}$, $w= \bar u- u$ and $t \in [0,1]$
 we have
\[ L[u] \leq L[u+t w],\]
so, in particular,
\begin{align}
\nonumber  0 &\leq \frac{d}{dt}\Big\vert_{t=0} L[u+tw] \\
  \label{eq:no-parts}  &= \int_{\Omega} (Du-x) \cdot Dw + w \, dx.
\end{align}
Note $u \in C^{0,1}(\overline{\Omega}) \cap C^{1,1}_{\text{loc}}(\Omega)$ with $\p^2_{ii} u\geq 0$ so we may apply the divergence theorem and obtain
\begin{align}\label{variational derivative}
{\frac{d}{dt}\Big\vert_{t=0} L(u+tw)
=} \int_{\Omega}(n+1-\Delta u)  \, wdx + \int_{\partial\Omega} (D^{-}u-x) \cdot \mathbf{n} \, wdS. 
\end{align}
where $\mathbf{n}$ is the outer unit normal to $\Omega$ which exists $\mathcal{H}^{n-1}$-a.e. for the convex domain $\Omega$.

  Inequality \eqref{eq:perturb-2} follows by performing the same argument with $u+\bar{u}$ in place of $\bar{u}$. For \eqref{eq:perturb-1} we perform similar calculations but use that $h(t):= L[\bar{u}  - tw]$ is strictly convex with a minimum at $1$ so $h'(0) < 0$. For the divergence theorem we take the one-sided directional derivatives and use that $Du$ and $D\bar u$  are of bounded variation 
  {\cite{AmbrosioFuscoPallara00} \cite[Ch. 5--6]{EvansGariepy92}}.
\end{proof}

\begin{remark}
  (1)  It is straightforward to see, again by arguing using a perturbation, that inequality \eqref{eq:perturb-2} holds not just for $\bar{u} \in \mathcal{U}$ but for any convex $\bar{u}$ with $\text{spt} \, \bar{u}_{ -}$ (the support of the negative part of $\bar{u}$) disjoint from the set $\{ u = 0\}$. The key observation is that for sufficiently small $t$, $u + t(\bar u - u) \in \mathcal{U}$.\\
(2) In any neighbourhood where $u$ is $C^2$ and uniformly convex, that is $u$ 
satisfies an estimate $D^2u \geq \lambda I >0$, one may perturb --- as is standard in the calculus of variations --- by smooth compactly supported functions and obtain $\Delta u =n+1$ in the interior and $(D^-u-x) \cdot \mathbf{n} = 0$ on the fixed boundary of $\Omega$.
Without a local uniform convexity estimate, even for smooth functions $\bar{u}$, it may be that there is no $t>0$ small enough to ensure $u+t\bar{u}$ is convex.
\end{remark}

Inequality \eqref{eq:perturb-1} is useful when one chooses $\bar{u}$ as paraboloid with prescribed Laplacian. We give an example now ---  the result we prove is required in subsequent sections. It is interesting to contrast the following result with the one-dimensional case, in which minimizers on domains $\Omega \subset [0,\infty) 
$ satisfy $u'(x) \leq x$.
Take the minimal convex extension of $u \in C^1(\Omega)$
to $\Rn$. 
In terms of the inner and outer boundary traces 
$D^\pm u \in [(L^\infty\cap BV)(\Rn)]^n$ from \eqref{dom D^-u}--\eqref{D^-u}:

\begin{proposition}[ Outer normal distortion is not inward]
\label{prop:neumann-sign}
  Let $u$ solve \eqref{eq:monopolist} where $\Omega \subset\mathbf{R}^n$ is bounded, open, and convex. Then 
  $\mathcal H^{n-1}(\p \Omega \setminus \dom D^-u)=0$
  and each $x_0 \in \p \Omega \cap \dom D^- u$ satisfies 
  \begin{equation}\label{neumann-sign+}
  (D^{ +}u(x_0) - x_0) \cdot \mathbf{n} \geq 0.
  \end{equation}
\end{proposition}

\begin{proof}
 By approximation  it suffices to prove the result for smooth \textit{strictly} convex domains. Indeed, \cite[Corollary 4.7]{FigalliKimMcCann11} and its proof imply if $\Omega^{(k)} \supset \Omega$
is a sequence of smooth strictly convex approximating domains and $u^{(k)}$ is the solution of \eqref{eq:monopolist} on $\Omega^{(k)}$ then $Du^{(k)}(x) \rightarrow Du(x)$ at every $x$ where $u$ and each $u^{(k)}$ is differentiable. Interior $C^{1,1}_{\text{loc}}$ regularity for $u^{(k)}$ at $x_0 \in \Omega^{(k)}$ and 
$x_0 \in \dom D^- u$ implies 
$Du^{(k)}(x_{0})$ accumulates onto $\p u(x_0)$ as desired.

Thus we take $\Omega$ to be smooth and strictly convex.  Up to a choice of
coordinates we assume $x_0 = 0$ and $\mathbf{n} = e_1$ (see Figure
\ref{fig:neumann-cond}). Recall $p_{x_0}(x):= u(x_0)+D^{ +}u(x_0)\cdot(x-x_0)$ is an
affine support at $x_0$.  For $t>0$ sufficiently small and $x=(x^1,\dots,x^n) \in
\mathbf{R}^n$ we consider the family of admissible perturbations (see Figure
\ref{fig:neumann-cond})
   \begin{align*}
     \hat{u}_t(x)&:= \frac{n+1}{2}\big([x^1 +t]_{+}\big)^2 + p_{x_0}(x),\\
     \bar{u}_t(x) &:= \text{max}\{u(x),\hat{u}_t(x)\},\\
     \Omega_{t} &:= \{x \in \overline{\Omega} ; \bar{u}_t(x) > u(x)\}.
   \end{align*}
Note $\Omega_t$ has positive measure since at $x_0=0$, $\hat{u}_t(0) =u(0)+
{(n+1)}t^2/2 > u(0)$. Moreover if $x^1 \leq - t$ then $\bar{u}_t(x) \leq u(x)$ and
thus $\Omega_{t} \subset \{x \in \overline{\Omega} ; x^1 \geq -t\}$. Thus, strict convexity of $\Omega$
implies $\overline{\Omega}_t \rightarrow \{0\}$ in Hausdorff distance \cite[\S 1.8]{Schneider92}
as $t\rightarrow 0$. Clearly $\Delta \bar{u}_t = n+1$ on $\Omega_{t}$.  To derive a contradiction assume
$(D^{ +}u(x_0)-x_0) \cdot \mathbf{n} <0$. Then for $t$ sufficiently small we also have
   \begin{equation}
     \label{eq:yields-contra}
     (D^{ +}\bar{u}_t(x) - x) \cdot \mathbf{n} < 0 \quad \text{on}\quad \Omega_{t} \cap \partial \Omega,
   \end{equation}
due to the upper semicontinuity of $\mathbf n \cdot D^+ u$ that follows from closedness of $\p u$, and also because
$|D\bar{u}_{t}(x_0)-Du(x_0)| = O(t)$. But \eqref{eq:perturb-1} with $\Delta \bar u = n+1$ implies
   \[ 0 < \int_{\Omega_t \cap \partial \Omega} (\bar{u}_t-u)(D^{ -}\bar{u}_t(x) - x) \cdot \mathbf{n} \, dS.\]
Since $(D^+\bar u - D^-\bar u)\cdot \mathbf n \ge 0$ this contradicts \eqref{eq:yields-contra} to conclude the proof.
         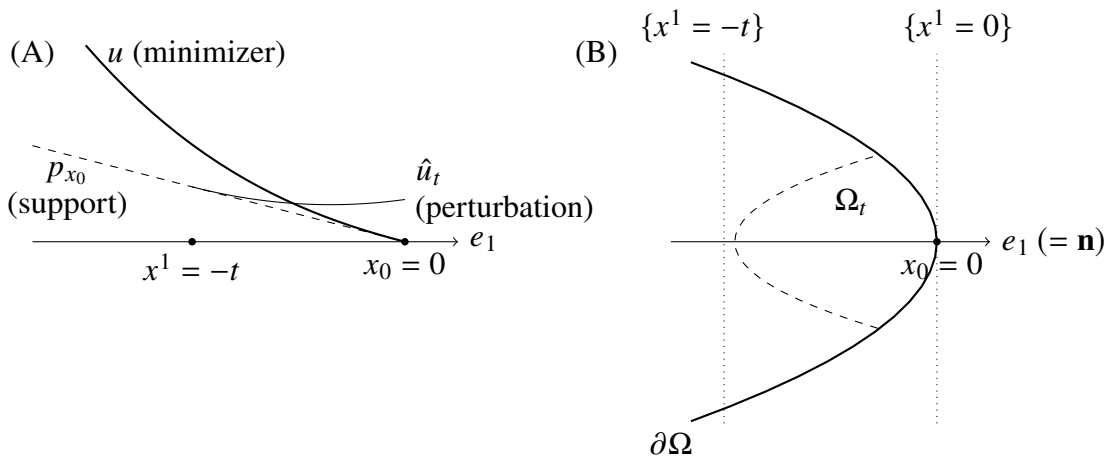
\begin{figure}
        \centering
        \centerline{\begin{tikzpicture}[scale=1.4]
%%%%% Left subfigure            
            \node at (-4.5,2)  {(A)};
%%%%%%%%%% Function $u$            
          \draw[thick,domain=-4:-1] plot(\x, -0.02*\x*\x*\x-0.2*\x);         
          \node [right] at (-3.9,2) {$u$ (minimizer)};
%%%%%%%%%% Support (dashed)          
          \draw[dashed,domain=-4.5:-1] plot(\x, -0.26*\x-0.04);
          \node [left,text width=3cm, align=center] at (-3,0.7) {$p_{x_0}$ \\ (support)}; 
%%%%%%%%%% Perturbation 
          \draw[domain=-3:-1] plot(\x, {-0.26*\x-0.04+0.1*(\x+3)*(\x+3)});
%          \draw[domain=-4.5:-3] plot(\x, -0.26*\x-0.04); solid left half perturbation
          \node [right,text width=3cm, align=left] at (-1,0.7) {$\hat{u}_t$ \\ (perturbation)};
%%%%%%%%%% Labels          
          \draw[thin,->] (-4.5,0.22) -- (-0.5,0.22);
          \node [right] at (-0.5,0.22) {$e_1$};
           
           \node [below] at (-1, 0.22) {$x_0=0$};
           \node[circle, fill=black, inner sep=1pt] at (-1, 0.22) {};

           \node [below] at (-3, 0.22) {$x^1 = -t$};
           \node[circle, fill=black, inner sep=1pt] at (-3, 0.22) {};

%%%%% Right subfigure           
           \node at (0.8,2) {(B)};
            \draw[thick,domain=-1.7:1.7] plot({-0.8*\x*\x+4}, \x+0.22);
            \node [above] at (1.5,-1.9) {$\partial \Omega$};
                      \draw[thin,->] (1.5,0.22) -- (4.5,0.22);
                      \node [right] at (4.5,0.22) {$e_1 \, (=\mathbf{n})$};

                      \node [below] at (4.05, 0.22) {$x_0=0$};
                      \node[circle, fill=black, inner sep=1pt] at (4, 0.22) {};

                      \draw [dotted] (4,-1.5) -- (4,2);
                      \node [above] at (4.2,2) {$\{x^1 = 0\}$};

           \draw [dotted] (2,-1.5) -- (2,2);
           \node [above] at (1.8,2) {$\{x^1 = -t\}$};

           \draw[dashed,domain=-0.82:0.82] plot({2*\x*\x+2.1}, \x+0.22);
           \node at (3.2,0.6) {$\Omega_t$};

        \end{tikzpicture}}
        \caption{Illustrates the constructions in the proof of Proposition~\ref{prop:neumann-sign}. Subfigure (A) shows a cross-section of the minimizer $u$, its support $p_{x_0}$, and the perturbation $\hat{u}_t$. Subfigure (B) illustrates that because $\Omega_{t} \subset \{x^1 \geq -t\}$ and $\Omega$ is strictly convex with outer normal $\mathbf{n} = e_1$ at $0$ we have $\overline{\Omega_t} \rightarrow \{0\}$ in the Hausdorff distance.}
        \label{fig:neumann-cond}
      \end{figure}
   \end{proof}

 An essential tool is that the variational inequality \eqref{eq:perturb-2} holds not only on $\overline \Omega$ but restricted to  enough of the leaves $\tilde{x}$ 
 to cover $\mathcal H^{n-1}$ almost all of $\overline \Omega$.
 This novel and powerful technique was pioneered in this context by 
 Rochet and Chon\'e~\cite{RochetChone98}, based on the sweeping theory of measures in convex order.
 In this section we recall and exploit a statement of their localization result; we relax their requirement that $u \in C^1( \overline\Omega)$ in Appendix~\ref{sec:rochet-chones-use}.

 We introduce the following notation for the variational derivative
 $\sigma:=\delta L/\delta u$ of our objective $L(u)$: 
 \begin{equation}
   \label{eq:variational-derivative}
   d \sigma (x) = (n+1-\Delta u)d\mathcal{H}^n\mres \Omega + (D^{ -}u-x) \cdot \mathbf{n} d\mathcal{H}^{n-1} \mres \partial\Omega,
 \end{equation}
 which turns out to be a measure with finite total variation.

Recall we extended $Du$ to the boundary as $D^-u$ and defined the equivalence
relation $x_1 \sim x_2$ if and only if $D^-u(x_1) = D^-u(x_2)$.  Here $\dom D^-u =
\cup_{i=0}^n \Omega_{i}$ covers $\mathcal H^{n-1}$-almost all of $\overline\Omega$.  Letting
$y=D^-u(x)$ index the equivalence class of~$x$, we can disintegrate $\sigma$ by
conditioning on the value of $D^-u$ as follows.  For a unique $z \in \{u=0\}$,
Theorem \ref{T:strictly convex disintegration} asserts the measure $\bar \sigma = \sigma -
|\Omega|\delta_z$ has vanishing zeroth and first moments.  Defining its conditional
measures $\bar \sigma_{y} = \bar \sigma_{+,y}- \bar \sigma_{-,y}$ by disintegrating separately
the positive and negative parts of $\bar \sigma$ with respect to the equivalence
relation $\sim$, for $\tilde \sigma_+ := (D^-u)_\#\sigma_+$ almost every $y$ we show
\[
  0 \leq \int_{(D^-u)^{-1}(y)} w(x') d\bar \sigma_{y} (x') 
\]
for all convex functions $w:\R^n\longrightarrow \R$,  
much as Rochet and Chon\'e \cite{RochetChone98}
did for $u \in C^1(\overline \Omega)$.

{We now exploit this technique to show \eqref{neumann-sign+} can be improved by replacing
$D^+u$ by $D^-u$ --- at least on the set $\Omega_0$ of top dimensional leaves \eqref{stratify}.  It is interesting to contrast the radically different methods used to prove the previous proposition and the next one.
\begin{proposition}
    [Top dimensional inner normal distortion is not inward]
\label{P:neumann-sign-}
  Let $u$ solve \eqref{eq:monopolist} where $\Omega \subset\mathbf{R}^n$ is bounded, open, and convex. 
  Then $\mathcal H^{n-1}$-a.e. 
  $x \in {\Omega_0} \cap \p\Omega$ satisfies
  \begin{equation}\label{neumann-sign-}
  (D^{ -}u(x) - x) \cdot \mathbf{n} \geq 0.
  \end{equation} 
\end{proposition}

\begin{proof}
Recall $\Omega_0$ is a countable disjoint union $\cup_{i=1}^\infty \tilde z_i$ of convex leaves, 
each with positive volume $|\tilde z_i|>0$ and $u$ being affine
with slope $y_i=D^-u(z_i)$
on $\tilde z_i$.  Thus $(D^-u)_\#\sigma_+$ assigns 
mass at least $(n+1)|\tilde z_i|> 0$ to $y_i$ by \eqref{eq:variational-derivative}.
The probability measures conditioned on $D^-u=y_i$ are therefore given by $\bar \sigma_{\pm,y_i}=c_i^\pm\bar \sigma_\pm \mres \tilde z_i$,  and share the same proportionality constant $c_i^\pm=:c_i$ in light of Theorem \ref{T:strictly convex disintegration} and Lemma \ref{L:identifying disintegration},
which moreover imply that $\bar \sigma \mres \tilde z_i$ has nonnegative integral against any convex Lipschitz function,
(recalling $\bar \sigma = \sigma - |\Omega|\delta_z$
for some $z \in \dom D^-u$ with $u(z)=0$).

As in the proof of Theorem \ref{T:strictly convex disintegration},
for each $i$, Strassen's theorem \cite{Strassen65} asserts the existence of Borel random variables $X^\pm$ distributed like $\bar \sigma_{y_i,\pm}$ respectively, such that $X^-=E(X^+|X^-)$ a.s.; in other words, such that $(X^-,X^+)$ is a martingale.  Let $D(\Omega) \subset \p\Omega$ denote the subset of points where $\p \Omega$ is differentiable, or equivalently has a unique supporting hyperplane.
The Gauss map $\mathbf n: D(\Omega) \longrightarrow \p B_1(0) \subset \R^n$ is Borel,  and its preimages ${\mathbf n}^{-1}(\omega)$ represent disjoint convex (but not necessarily compact) `facets' whose union is $D(\Omega)$ hence covers $\mathcal H^{n-1}$ almost all of $\p \Omega$.
Let $\bar {\mathbf n}$ denote the extension of $\mathbf n$
obtained by setting $\bar {\mathbf n}(x) =0$ if $x \in \Omega$.
Set $F^\pm= \bar {\mathbf n} \circ X^\pm$ and
let $\mu^\pm_{\omega}$ be the law of $X^\pm$ conditioned on 
$F^\pm=\omega$,  i.e. 
$\mu^\pm :=\bar \sigma_{\pm,y_i} = \int \mu^\pm_\omega d\tilde \mu^\pm(\omega)$
where $\tilde \mu^\pm = F^\pm_\#\bar \sigma_{y_i,\pm}$ in the notation
of Definition \ref{D:disintegrate}.  
Since $X^-\not\in \Omega$ lies on the boundary of the convex set $\Omega$ and $X^+ \in \overline \Omega$, the martingale property forces 
$F^+= F^-$ (a.s.) unless $X^- \in \Omega$,
so $\tilde \mu^+ \ge \tilde \mu^-$ except perhaps at $\omega=0$.
Both $\mu^\pm_\omega$ 
therefore vanish outside the facet $\bar {\mathbf n}^{-1}(\omega)$
for $\tilde \mu^-$-a.e. $\omega \ne 0$
by Theorem~\ref{T:abstract disintegration}.

Recall $d\mu^+ = (n+1-\Delta u)_+d\mu^n + [(D^-u -id)\cdot \mathbf n]_+  d\mu^{n-1}$,  where $\mu^n := c_i \mathcal H^{n}\mres (\tilde x_i \cap \Omega)$
and $\mu^{n-1}:= c_i \mathcal H^{n-1}\mres (\tilde x_i \cap \p \Omega)$. 
As in Lemma \ref{L:identifying disintegration},
for $\tilde \mu^-$-a.e. $\omega$ in the Gauss sphere $\p B_1(0)\subset \R^n$ 
there is a normalization constant $c^+_\omega>0$ such that $c^+_\omega d\mu^+_\omega = [(D^-u -id)\cdot \omega]_+ d\mu^{n-1}_\omega$ where 
$\mu^{n-1} = \int \mu_\omega^{n-1} d\tilde \mu^{n-1}(\omega)$
and $\tilde \mu^{n-1}:={\mathbf n}_\#\mu^{n-1}$ are also provided by Theorem \ref{T:abstract disintegration}.

Now $[(D^-u -id)\cdot \mathbf n]_+ d\mu^{n-1}_\omega$ cannot vanish identically when $c^+_\omega>0$.  
Since $D^-u$ and $id\cdot \omega$ are both constant on 
the facet $\bar {\mathbf n}^{-1}(\omega) \cap \dom D^-u$
normal to $\omega$, it follows that 
$[(D^-u-id)\cdot \mathbf n]_-=0$ throughout the same facet 
for $\tilde \mu^-$-a.e. $\omega \ne 0$.
The representation 
$c^-_\omega d\mu^-_\omega = [(D^-u -id)\cdot \omega]_- d\mu^{n-1}_\omega + \hat c1_{\bar {\mathbf n}^{-1}(\omega)}(z)|\Omega|d\delta_z$ 
of $\mu^-\mres \p \Omega = \int_{\p B_1(0)} \mu^-_\omega d\tilde \mu^-(\omega)$ using constants $\hat c>0$ and $c^+_\omega>0$
for $\tilde \mu^-$-a.e. $\omega \ne 0$ now shows 
$d\mu^- =(n+1-\Delta u)_-d\mu^n + [(D^-u -id)\cdot \mathbf n]_-  d\mu^{n-1} + |\Omega|d\delta_z$ 
does not charge $\p \Omega$ except perhaps
at the single point $z$ because of the vanishing of the terms in square brackets (at $\mu^{n-1}$-almost every Lebesgue point of $\mu^- \mres \p \Omega$ with respect to $\mu^{n-1}$).
The proposition is therefore established.
\end{proof}
}

\subsection{Legendre transforms and Alexandrov second derivatives}
\label{sec:legendre-transf-alek}

Recall if $u : \Omega \rightarrow \mathbf{R}$ is a convex function, then its Legendre transform is defined by
\begin{equation}
\label{Legendre}
v(y) = \sup_{x \in \Omega} x \cdot y - u(x).
\end{equation}
A function is called Alexandrov second differentiable at $x_0$, with Alexandrov Hessian $D^2u(x_0)$ (an $n \times n$ matrix), provided as $x \rightarrow x_0$ that
\[ u(x) = u(x_0)+Du(x_0)\cdot(x-x_0) + \frac{1}{2}(x-x_0)^TD^2u(x_0)(x-x_0) + o(|x-x_0|^2).\]
Alexandrov proved convex functions are twice differentiable in this sense $\mathcal{H}^n$ almost everywhere.

It's well known that if a differentiable convex function $u$ is Alexandrov differentiable at $x_0$ and its Legendre transform is Alexandrov second differentiable at $y_0 := Du(x_0)$ then
\[D^2v(y_0) = \big[D^2u(x_0)\big]^{-1}. \]
We have an analogous result even when Alexandrov second differentiability is not assumed. 

\begin{lemma}[Legendre transform of Hessian bounds]
\label{lem:aleksandrov-derivatives}
  Assume $u:\Omega \rightarrow \mathbf{R}$ is a convex function with Legendre transform $v$, that $x_0 \in \Omega$ and $M$ is an invertible symmetric positive definite matrix. Assume $y_0 \in \partial u(x_0)$. Then
   \[ u(x_0+\delta x) \geq u(x_0) + y_0\cdot \delta x +  \delta x^TM\delta x/2 +o(|\delta x|^2) 
    \quad {\rm as}\ \delta x \rightarrow 0,\] if and only if
    \[v(y_0 + \delta y) \leq v(y_0) + x_0\cdot \delta y + \delta y^TM^{-1}\delta y/2+o(|\delta y|^2)
    \quad {\rm as}\ \delta y \rightarrow 0.\]
  \end{lemma}
  \begin{proof}
We prove the ``only if'' statement; the ``if'' statement is proved similarly. Up to a choice of coordinates and subtracting an affine support we may assume $x_0=0=y_0$ and $u(x_0) = 0$. Whereby we're assuming
\begin{equation}
u(x) \geq x^TMx/2 +o(|x|^2).  
\label{eq:ox2}
\end{equation}
It is straightforward to see that \eqref{eq:ox2} holds if and only if for
every $\epsilon > 0$ there is a neighbourhood $\mathcal{N}_\epsilon$ of $0$ on which 
\begin{equation}
u(x) \geq (1-\epsilon)x^TMx/2.
\label{eq:ox2-equiv}
\end{equation}
Now, for $y \in \partial u(\mathcal{N}_\epsilon)$ we have
\begin{align*}
  v(y) &= \sup_{x \in \Omega} x \cdot y - u(x)\\
       &= \sup_{x \in \mathcal{N}_\epsilon} x \cdot y - u(x)\\
  &\leq \sup_{x \in \mathcal{N}_\epsilon} x \cdot y - (1-\epsilon)x^TMx/2.
\end{align*}
Provided $y$ lies in the possibly smaller neighbourhood $Y_{\epsilon}:=[(1-\epsilon)M\mathcal{N}_\epsilon] \cap \partial u(\mathcal{N}_\epsilon)$ the supremum is obtained at $x = [(1-\epsilon)M]^{-1}y$ and
\[ v(y) \leq y^T[(1-\epsilon)M]^{-1}y/2.\]
The aforementioned equivalence between \eqref{eq:ox2-equiv} and \eqref{eq:ox2} gives the desired result. 
\end{proof} 

\section{Partition into foliations by leaves 
that 
intersect  fixed boundary}
\label{sec:structure-convexity}
In this section we present the proof of Theorem \ref{thm:structure-convexity}. 
 We use localization (Appendix \ref{sec:rochet-chones-use}) to obtain part (1) 
 {(and later the vanishing Neumann condition complementing (3) throughout $\Omega_n \cap \p \Omega$). Apart}
 from that the only technique we use is energy comparison, and the variational inequalities \eqref{eq:perturb-0} -- \eqref{eq:perturb-2} coupled with careful choice of the comparison functions, many of which are inspired by those used to study the Monge--Amp\`ere equation \cite{Figalli17, LiuWang15}. We denote a subsection to each point of Theorem \ref{thm:structure-convexity}. 

\subsection{Point 1: $\Omega_0 = \{u = 0\}$ if nonempty.}

The fact that $Z:=u^{-1}(0) \subset \p\Omega$ and has positive area when $\Omega_0$ is empty is proved in Corollary \ref{lem:u=0_set_dimension}. Therefore we shall only prove $\Omega_0 \subset \{u= 0\}$; equality follows easily if $\Omega_0$ is nonempty, which is known to be true on strictly convex domains \cite{Armstrong96}. 

 \begin{proof}[Proof of Theorem \ref{thm:structure-convexity} (1)]
Fix $x_0 \in \Omega_0$,  meaning $|\tilde x_0|>0$. We recall from Theorem~\ref{T:strictly convex disintegration} that there is $z \in \{u = 0\}$ such that $\sigma - |\Omega|\delta_z$ has nonnegative integral against all real convex functions. The first step is to establish the result under the assumption  $z \not\in \tilde{x_0}$.  In this case we may apply Theorem~\ref{T:strictly convex disintegration} equation \eqref{consistency and local dominance+} with $w(x)=\pm p_{x_0}(x)$ and the characterization of the variational derivative provided by Lemma~\ref{L:identifying disintegration} to obtain
  \begin{equation}
    \label{eq:rc-ineq-2}
    0 = \int_{\tilde{x_0}}(n+1-\Delta u) u \, dx + \int_{\overline{ \tilde{x_0}} \cap \partial \Omega} u (D^{ -}u-x) \cdot \mathbf{n} \, dS,
  \end{equation}
  {noting $w=u$ on $\tilde x_0$.}
  Using $\Delta u = 0$ on $\tilde{x_0}$, $u \geq 0$,   disjointness of $\overline {\tilde x_0} \setminus \tilde x_0$ from $\dom D^-u$ and $(D^{ -}u-x) \cdot \mathbf{n} \geq 0$ holding $\mathcal H^{n-1}$ a.e. on $\tilde x_0 \cap \partial \Omega$ (from Proposition~\ref{P:neumann-sign-}) 
  then yields $u = 0$ on $\tilde{x_0}$ as required.

The second case to consider is $z \in \tilde{x_0}$ and thus $u = 0$ at $z \in \tilde{x_0}$. Of course, if $z \in \Omega$ then $Du(x_0)=Du(z)=0$ and the proof is complete. If not, by disintegrating $\Omega \setminus \overline{\tilde{x_0}}$ into equivalence classes, none of which contain $z$ and on which Theorem \ref{T:strictly convex disintegration} applies with both $w = \pm p_{x_0}$, Lemma \ref{L:identifying disintegration} yields
  \begin{equation}
    \label{eq:ineq-away-x0}
    0 = \int_{\Omega\setminus \tilde{x_0}} (n+1-\Delta u) p_{x_0} \, dx + \int_{\partial \Omega \setminus \overline { \tilde{x_0}}}p_{x_0}(D^-u-x) \cdot n \, dS.
  \end{equation}
  Now apply Lemma \ref{lem:variational-inequalities} inequality \eqref{eq:perturb-2} with $\overline{u} = u-p_{x_0}$, which is admissible and equals $0$ on $\tilde{x_0},$ to obtain
  \[ 0 \leq \int_{\Omega \setminus \tilde{x_0}} (n+1-\Delta u)(u-p_{x_0}) \, dx +\int_{\partial \Omega \setminus \overline{\tilde{x_0}}}(u-p_{x_0}) (D^-u-x) \cdot \mathbf{n} \, dS. \]
  Addition of \eqref{eq:ineq-away-x0} gives
  \[ 0 \leq \int_{\Omega \setminus \tilde{x_0}} (n+1-\Delta u)u \, dx +\int_{\partial \Omega \setminus \overline{ \tilde{x_0}}}u (D^-u-x) \cdot \mathbf{n} \, dS. \]
Combined with
  \[ 0 =  \frac{d}{dt}\Big\vert_{t = 1}L[tu] = \int_{\Omega } (n+1-\Delta u)u \, dx +\int_{\partial \Omega} u (D^-u-x) \cdot \mathbf{n} \, dS, \]
  yields
\[         0 \geq \int_{\tilde{x_0}}(n+1-\Delta u) u \, dx + \int_{\overline{ \tilde{x_0}} \cap \partial \Omega} u (D^{ -}u-x) \cdot \mathbf{n} \, dS\]
and we conclude using the same reasoning following \eqref{eq:rc-ineq-2}.
\end{proof}

\subsection{Point 3 of Theorem \ref{thm:structure-convexity}}

Now we present point 3: that $\Delta u = n+1$ in $\Omega_{n}$ and that $\Omega_n \cap \Omega$ is open.  One could obtain $\Delta u = n+1$ a.e. in $\Omega_n$ directly from Rochet--Chon\'e's localization (Theorem \ref{T:strictly convex disintegration} 
and its sequelae). 
However for point 3 of Theorem \ref{thm:structure-convexity} we require in addition, an inequality for $\Delta u(x_0)$ at all points where $Du(x_0) \in \text{int}Du(\Omega)$. Thus we prove the following lemma  using perturbations
instead. 

\begin{lemma}[Sub- and super-Poisson 
for interior vs.\ customized consumption]
\label{lem:leq/geq}
  Assume $u:\Omega \rightarrow \mathbf{R}$ solves \eqref{eq:monopolist} and $x_0 \in \Omega$ is a point of Alexandrov second differentiability satisfying $Du(x_0) \in \intr(Du(\Omega))$, $u(x_0)>0$. Then
  \begin{enumerate}
\item\label{I:interior} There holds $\Delta u(x_0) \geq n+1$.
\item\label{I:customize} If, in addition $u$ is strictly convex at $x_0$, then $\Delta u(x_0) \leq n+1$.
  \end{enumerate}
\end{lemma}
\begin{proof}
  (1) We assume $x_0 \in \Omega$ to be a point of Alexandrov second differentiability
  with $Du(x_0) \in \intr Du(\Omega)$. For convenience translate and subtract the affine support at $x_0$ so that $x_0,u(x_0)$ and $Du(x_0)$ all vanish. 

  For a contradiction assume $\Delta u(x_0) < n+1$ and take $\epsilon > 0$ satisfying  $\Delta u(x_0)+n \epsilon < n+1$. There is a neighbourhood of $x_0=0$ on which $u(x) < x^T(D^2u(x_0)+\epsilon I)x/2$. 

 Let $v$ denote the Legendre transform \eqref{Legendre} of $u$ and set  $\tilde{v} = y^T[D^2u(0)+\epsilon I]^{-1}y/2$. 
 Lemma \ref{lem:aleksandrov-derivatives} implies  $\tilde{v} < v$ in a punctured neighbourhood of the origin. Thus as $h \rightarrow 0$ the connected component of $\{x ; v < \tilde{v} + h\}$ containing the origin, which we denote $\Omega^*_h$, converges to $\{0\}$ in the Hausdorff distance.  Set
  \begin{equation}
    \label{eq:perturb}
    v_h(y) =
    \begin{cases}
      \tilde{v}(y)+h & y \in \Omega^*_h\\
      v(y) & y \not\in \Omega^*_h.
    \end{cases}
  \end{equation}
  Let $u_h$ be the Legendre transform of $v_h$. Note $u_h \leq u$ and this
inequality is strict at $x_0$. Moreover $u(x_0)>0$ implies for $h$ sufficiently small $u_h \geq 0$. Moreover because $D^2v_h \geq (D^2u(0)+\epsilon I)^{-1}$
at each $y \in \Omega^*_h$ Lemma \ref{lem:aleksandrov-derivatives} implies $\Delta u_h \leq \Delta
u(x_0)+n\epsilon < n+1$ on the set $\{u_h < u\}$.  This contradicts inequality
\eqref{eq:perturb-1} to establish \eqref{I:interior}, where $\partial \Omega \subset \{u_h = u\}$
for $h$ small enough has been used.

(2) Now suppose, in addition, $x_0$ is a point of strict convexity for $u$ and, for a contradiction, that $\Delta u(x_0) > n+1$.  Set $\bar{u}(x) = (1-\epsilon)x^TD^2u(0)x/2$ with $\epsilon>0$ chosen so small that $\Delta \bar{u} = (1-\epsilon)\Delta u(x_0) > n+1$. Note that $\bar{u} < u$ in a punctured neighbourhood of $0$; (this relies on the strict convexity of $u$ at $0$ in the case $D^2u(0)$ has a zero eigenvalue). Thus for sufficiently small $h>0$ the connected component of $\{x; u(x) < \bar{u}(x)+h\}$ containing $x_0=0$, which we call $ \Omega^{(h)}$, converges to $\{0\}$ in the Hausdorff distance. Set
\[ \bar{u}_{h} =
  \begin{cases}
    \bar{u}(x)+h &x \in \Omega^{(h)},\\
    u(x) &x \not\in \Omega^{(h)}.
  \end{cases}
\]
Then $\bar{u}_{h}$ is an admissible interior
perturbation of $u$ with $\Delta \bar{u}_{h}>n+1$ on $\{\bar{u}_h >u\}$. Once again we contradict inequality \eqref{eq:perturb-1}. 
\end{proof}

At any interior point of strict convexity, $x \in \Omega_n \cap \Omega$, we have $Du(x) \in \intr Du(\Omega)$  since $\p u$ is closed. It's now immediate that $\Delta u = n+1$ at each point of Alexandrov second differentiability in $\Omega_n \cap \Omega$. It remains to show $\Omega_n \cap \Omega$ is open.  We prove in the next subsection that $D^{ -}u(\bigcup^{n-1}_{i=0}\Omega_i) \subset \partial Du(\Omega)$, that is if $x \in \bigcup^{n-1}_{i= 0}\Omega_i$ then $D^{ -}u(x)$ is in the boundary of the set of gradients. Combined with the $C^{1,1}_{\text{loc}}$ regularity of $u$ we obtain if $x \in \Omega_n \cap \Omega$ the same is true for all sufficiently close $ \bar{x}$ (this is because $Du(\bar{x})$ is also in $\intr Du(\Omega)$) and this is a sufficient condition for $\bar{x} \in \Omega_n \cap \Omega$. We conclude $\Omega_n \cap \Omega$ is open. 

Since we now know $u$ is a $W^{2,\infty}_{\text{loc}}$ (equivalently, $C^{1,1}_{\text{loc}}$) solution of $\Delta u = n+1$ almost everywhere on the open set $\Omega_n \cap \Omega$ the elliptic regularity \cite[Theorem 9.19]{GilbargTrudinger83} implies $u \in C^\infty(\Omega_n)$.

\subsection{Point 2 of Theorem \ref{thm:structure-convexity}}
\label{sec:point-3-theorem}
To conclude the proof of Theorem \ref{thm:structure-convexity} we show if $x \in \Omega_{i}$ for $i = 1,\dots,{n-1}$ then $\tilde{x}$ extends to the boundary.

For a contradiction assume otherwise. Because $u$ is convex, then
there exists $x_0$ with $\{x_0\} \ne \overline{\tilde{x_0}}  \subset\subset \Omega$  and $y_0 := Du(x_0) \in \intr Du(\Omega)$ since $\p u$ is closed. Because $u$ is $ C^{1,1}$ 
and the sections
\[ \Omega^{(h)} := \{x  \in \overline\Omega; u(x) \leq  u_h(x) := u(x_0)+Du(x_0) \cdot (x-x_0)+h\}\]
converge to $\overline{\tilde{x_0}}$ in the Hausdorff distance as $h \rightarrow 0$, we obtain  $Du(\Omega^{(h)}) \subset \subset Du(\Omega)$ for $h$ sufficiently small. In particular by Lemma \ref{lem:leq/geq} we have $\Delta u \geq n+1$ in $\Omega^{(h)}$. Using $\hat u= \max\{u,u_h\}$ as a perturbation function in inequality  \eqref{eq:perturb-0} we see $\Delta u = n+1$ almost everywhere in $\Omega^{(h)}$ (if $\Delta u > n+1$ on a subset of $\{u_{h} > u\}$ with positive measure, inequality \eqref{eq:perturb-0} is violated). 
As in the previous subsection the elliptic regularity implies $u \in C^\infty(\Omega^{(h)})$ is a classical solution of $\Delta u =n+1$ in $\Omega^{(h)}$. Differentiating this PDE twice implies the second derivatives of $u$ are harmonic (and nonnegative by convexity). The strong maximum principle for harmonic functions says in fact $\p^2_{jj}u > 0$
in $\Omega^{(h)}$  for all $j = 1,2, ..., n$, so $u$ cannot be affine on $\tilde{x}$. This contradiction completes the proof. Note our use of the strong maximum principle requires $ \p^2_{jj}u > 0$ at some point in $\Omega^{(h)}$. This, however, follows by considering $v = u - u_h$. 
If $\p^2_{jj} u=0$ throughout $\Omega^{(h)}$ then $Dv(x_0) = 0$, $v(x_0) = -h$ and $v$ is independent
of $x_j$, hence $\Omega^{(h)} \cap \p \Omega \ne \emptyset$,  which would contradict $\Omega^{(h)} \to \overline{\tilde x_0} \subset\subset \Omega$
as $h \downarrow 0$.

In the course of the above proof we've proved the following lemma which we record here since we require it again and again. 

\begin{lemma}[Interior regularity and strong maximum principle]
\label{lem:mp-argument}
  Assume $u : \Omega \rightarrow \mathbf{R}$ is a $C^{1,1}_{\text{loc}}$ convex function. Let $C \in \mathbf{R}$. Assume, in the sense of Alexandrov second derivatives, $\Delta u = C$ almost everywhere in $B_\epsilon(x_0) \subset \Omega$. Then $u \in C^\infty(B_\epsilon(x_ 0))$ and satisfies that for each unit vector $\xi$ either $u_{\xi\xi} \equiv 0$ throughout $B_\epsilon(x_0)$ or $u_{\xi\xi} > 0$ throughout $B_\epsilon(x_0)$.
\end{lemma}

\section{Relating strict convexity to the Neumann condition}
\label{sec:neum-estim}

One of our key techniques for studying the planar Monopolist's problem is the introduction of a new coordinate system defined in terms of the rays which foliate $\Omega_{n-1}$. These new coordinates are a powerful tool for studying the behaviour of the minimizer $u$ but their justification requires the technical {preparations} we undertake in this section and Assumption \ref{a:localc11}, i.e. \eqref{Assumption 2}.

For planar domains in neighbourhoods where 
\eqref{Assumption 2} holds we prove the equivalence between the Neumann condition and strict convexity stated  precisely in Propositions \ref{P:no distortion implies strict convexity} and \ref{P:strictly-convex-implies-neumann}.  We begin with two lemmas concerning convex functions in~$\mathbf{R}^n$. The first states the upper semicontinuity of the function $x \mapsto \text{diam}(\tilde{x})$ and the second yields the convexity of $\partial^{2}_{ii}u$ when restricted to a contact set~$\overline{\tilde{x}}$.

\begin{lemma}[Upper semicontinuity of leaf diameter]
\label{lem:upper-semicontinuity-diam}
  Let $\Omega$ be a bounded open convex subset of $\mathbf{R}^n$ and $u 
  {:\Rn\longrightarrow \R}$
  a convex function. Then the function 
  $x \in \dom D^-u  \mapsto \text{diam}(\tilde x)$ is upper semicontinuous. 
\end{lemma}
\begin{proof}
  We fix a sequence $(x_k)_{k \geq 1} \subset 
     \dom D^-u$ 
  converging to some $x_{\infty} \in 
   \dom D^-u$ 
  and note it suffices to prove that
  \[ \limsup_{k \rightarrow \infty} \text{diam}(\tilde{x_k}) \leq \text{diam}(\tilde{x_\infty}).\]
  To this end, let $p_k = D^-u(x_k)$ and take $x^{(1)}_k$, $x^{(2)}_k$ in $\overline{\Omega}  \cap \overline{\tilde x_k}$ realizing
  \begin{align*}
    |x^{(1)}_k - x^{(2)}_k| &= \text{diam}(\tilde{x_k})\\
  \text{and}\quad \lim_{k \rightarrow \infty} |x^{(1)}_k - x^{(2)}_k| &= \limsup_{k \rightarrow \infty} \text{diam}(\tilde{x_k}).
  \end{align*}
The fact that $\p u$ is closed implies
  that a (nonrelabelled) subsequence 
   $D^-u(x_k)$ 
   converges to a limit $p_\infty \in \p u(x_\infty)$ with 
   $x^{(i)}_k \rightarrow x^{(i)}_\infty \in \overline{\Omega}$ for $i=1,2$. Thus we may send $k \rightarrow \infty$ in the identity
  \begin{align}
    u(x_k^{(i)}) = u(x_k) + D^-u(x_k) \cdot (x_k^{(i)} - x_k)  
  \end{align}
  to obtain that for $i=1,2$ we have 
  $x^{(i)}_\infty \in \widetilde{x_\infty}$ and thus $\tilde{x}$ has diameter greater than or equal to
  \[ |x^{(1)}_\infty - x^{(2)}_\infty| = \lim_{k \rightarrow \infty} |x^{(1)}_k - x^{(2)}_k| = \limsup_{k \rightarrow \infty} \text{diam}(\tilde{x_k}).  \]
\end{proof}

Let $\ri(\tilde x)$ denote the relative interior of the convex set $\tilde x$.

\begin{lemma}[Existence a.e.\ of $D^2 u$ on a leaf implies
convexity of $\partial^2_{ii} u$] 
\label{lem:convex-laplacian} Let $u:\Omega \rightarrow \R$ be a differentiable convex function
defined on an open convex subset $\Omega \subset \mathbf{R}^n$. Fix any $x \in \Omega$.  If
$\mathcal{H}^{\dim \tilde x}(\tilde x \setminus \dom D^2 u) = 0$, where $\dom D^2 u \subset \Omega$ denotes
the set of  Alexandrov second differentiability of $u$, then $u|_{\ri(\tilde x)} \in C^2_{\text{loc}}(\ri(\tilde x)
)$
and $\partial^2_{ii}u|_{\tilde{x}}$
is a convex function for each $i=1,\dots,n$.
\end{lemma}

\begin{proof}
Fix $x$ satisfying $\mathcal{H}^{\dim \tilde x}(\tilde x \setminus \dom D^2 u) = 0$.
After subtracting the support at $x$, we may assume $u(x)=0, Du(x) = 0$ and that $\tilde{x}$ is not a singleton 
(since otherwise the result holds trivially). We will show $\partial^2_{ii} u$ is convex along any of those line segments contained in $\tilde{x}$ for which Alexandrov second differentiability holds a.e.  To this end fix $x_0, x_1 \in \tilde{x} \cap \dom D^2 u$
along such a segment.
Choose orthonormal coordinates such that $x_1 = x_0 + T e_1$ for some $T > 0$. Then since $u$ is affine on $\tilde{x}$ and $\{x_0 + t e_1 ; 0 \leq t \leq T \} \subset \tilde{x}$ we have $\partial^2_{11} u = 0$ a.e.\ on $\{x_0 + t e_1 ; 0 \leq t \leq T \}$. 

Next, for $i = 2,\dots,n$ and any $t \in (0,1)$, convexity of $u$ implies for $r>0$ sufficiently small 
\begin{equation}\label{eq:conv-u}
u((1-t)x_0 + t x_1 + re_i) \leq (1-t)u(x_0 + re_i) + t u(x_1+re_i). 
\end{equation}
Here, by $r$ sufficiently small we mean small enough to ensure the above arguments of $u$ are contained in $\Omega$. 
The definition of Alexandrov second differentiability along with $u, Du = 0$ on $\tilde{x}$ implies
\[ u((1-t)x_0 + t x_1 + re_{ i}) = r^2\partial^2_{ii}u((1-t)x_0 + t x_1)/2 + o(r^2), \]
for a.e.\ $t \in (0,1)$ and similarly at $x_0+re_i$ and $x_1 + re_i$. 
Thus \eqref{eq:conv-u} becomes
 \[ r^2\partial^2_{ii}u((1-t)x_0 + t x_1)/2 + o(r^2) \leq (1-t)r^2\partial^2_{ii}u(x_0)/2+t r^2\partial^2_{ii}u(x_1)/2 + o(r^2). \]
 Dividing by $r^2$ and sending $r \rightarrow 0$ yields that $\partial^2_{ii} u$ is the restriction of a convex function to the segment $[x_0,x_1]$ --- hence continuous 
 on $[x_0,x_1]$. The polarization identity 
 implies the continuity of mixed second order partial derivatives.   It follows that $u|_{\ri(\tilde x)} \in C^2_{\text{loc}}$.
\end{proof}

\begin{proposition}[No normal distortion nearby implies strict convexity]
  \label{P:no distortion implies strict convexity} Let $u$ minimize
\eqref{eq:monopolist} where $\Omega \subset\subset \mathbf{R}^{ 2}$ is open and convex. Let
$x_0 \in \partial \Omega$ be a point where $u(x_0)>0$ and $\tilde{x_0} \cap \partial \Omega = \{
x_0\}$. Assume there is $\epsilon > 0$ with
      \[ (Du(x) - x ) \cdot \mathbf{n} = 0\quad  {\mathcal H^{1} a.e.}\text{ on }B_\epsilon(x_0) \cap \partial \Omega.\] Then
      $\tilde{x_0} = \{x_0\}$, that is $u$ is strictly convex at $x_0$.
\end{proposition}

\begin{proof}
Using Assumption \ref{a:c1} ($u \in C^1(\overline{\Omega})$) and Lemma \ref{lem:upper-semicontinuity-diam}
implies the upper semicontinuity of $R(x):={\text diam} \tilde x$, 
we may find a possibly smaller $\epsilon>0$ such
that $u(x) > 0$ and $\tilde{x} \cap \partial \Omega = \{x\}$ for each $x \in \mathcal{N} := B_\epsilon(x_0) \cap \partial
\Omega$. {We'll now show} Rochet and Chon\'e's localization (Theorem~\ref{T:strictly convex disintegration}) 
with
$0$ boundary term 
{(Lemmas \ref{L:identifying disintegration} and \ref{L:absolute continuity})}
implies $\Delta u = 3$ almost everywhere on
\[ \tilde{\mathcal{N}} := \{ x' \in\ \tilde{x} ; x \in \mathcal{N}\}.\] This is certainly true for any $x \in \Omega_2$. Now, assume there is a positive $\mathcal{H}^2$ measure subset of $\tilde{\mathcal{N}}$ consisting of rays in $\Omega_1$. Either $\Delta u = 3$ on almost all of this set, or this set 
is non-negligible for the variational derivative, in which case we can apply Rochet and Chon\'e's localization as follows: By Lemma \ref{lem:convex-laplacian}, $\Delta u$ restricted to any $\tilde{x}$ for
which $\mathcal{H}^{\dim \tilde x}(\tilde x \setminus \dom D^2 u) = 0$ is a convex function.
Thus
$w = -(3-\Delta u)$ is a permissible test function in the localization
Theorem~\ref{T:strictly convex disintegration}. Combining with identification of the disintegration in Lemma~\ref{L:identifying disintegration} and the mutual absolute continuity of the disintegration with $\mathcal{H}^{1}\mres \tilde{x}$ obtained in Lemma~\ref{L:absolute continuity} yields
\begin{equation}
\label{eq:sign-contra}
    0 \leq - \int_{\tilde{x}}(3-\Delta u)^2 \, d\mu^2_{\tilde{x}}, 
\end{equation} 
where $\mu^2_{\tilde x}$ denotes the disintegration of the Lebesgue measure
with respect to the contact sets (an explanation of disintegration is provided in Appendix~\ref{sec:rochet-chones-use}). Inequality \eqref{eq:sign-contra} implies
that $\mathcal{H}^{\dim\tilde{x}}$ almost everywhere on $\tilde{x}$ there holds
$(3-\Delta u)^2 = 0$ and thus the same equality holds on $\tilde{\mathcal{N}}$.

Let $x_1$ be the interior endpoint of $\tilde{x_0}$ (if $\tilde{x_0}=\{x_0\}$
then we are already done).  In a sufficiently small ball $B_\delta(x_1)$, we have
just shown $\Delta u = 3$ a.e. in $B_\delta(x_1) \cap \tilde{\mathcal{N}}$. Moreover, Theorem
\ref{thm:structure-convexity} implies $\Delta u =3$ in $B_\delta(x_1) \setminus \tilde{\mathcal{N}} =
B_\delta(x_1) \cap \Omega_2$, {\em which is nonempty}. 
Our usual maximum principle argument, Lemma \ref{lem:mp-argument}, implies $u$
is strictly convex inside $B_\delta(x_1)$, contradicting that $x_1$ is the endpoint
of a ray.
    \end{proof}

The next proof requires Corollary \ref{lem:Du-balances} which gives that {for some $z \in u^{-1}(0)$ the pushforward of the variational derivative satisfies 
$Du_\#(\sigma)={ |\Omega|}\delta_{D^-u(z)}$  }--- and is proved in Appendix
\ref{sec:rochet-chones-use} by combining the neutrality implied by localization
away from the excluded region $\{u=0\}$ with the fact that our objective
responds proportionately to a uniform increase in indirect utility.  We use
these to estimate the following:

\begin{proposition}[One-ended ray lengths bound normal distortion]
\label{P:strictly-convex-implies-neumann}
  Let $u$ solve \eqref{eq:monopolist} where $\Omega \subset\subset \mathbf{R}^2$ is open and
convex.  Let $\{x_0\} =\tilde{x_0} \cap \partial \Omega $ where $\partial \Omega$ is {differentiable} in a neighbourhood
of $x_0$, and \eqref{Assumption 2} is satisfied at $x_0$ with $u(x_0)> 0$.  Set $R(x_0) = \text{diam}(\tilde{x_0})$.  Then
  \begin{equation}
    \label{eq:neumann-estimate}
    0 \leq (Du(x_0) - x_0) \cdot \textbf{n} \leq CR(x_0),
  \end{equation}
  where $C$ depends only on a $C^{1,1}$ bound for $u$ in a neighbourhood
of $\tilde{x_0}$.
  \end{proposition}
  \begin{proof}

The lower bound \eqref{eq:neumann-estimate} was established in
Proposition \ref{prop:neumann-sign},
{noting $u \in C^1(\overline \Omega)$ from Assumption \ref{a:c1}}. 
We first prove \eqref{eq:neumann-estimate} assuming $\tilde{x_0}$ nontrivial and at the conclusion of the proof explain why it holds for all $x_0$ in the theorem. Let $x_0 \in \partial \Omega$ and let $\partial \Omega$ be locally represented by a smooth curve with an arc length parametrization $\gamma: (-\epsilon,\epsilon ) \rightarrow \mathbf{R}^2$ traversing $\partial \Omega$ in the anticlockwise direction with $x_0 = \gamma(0)$ and without loss of generality $\dot{\gamma}(0) = e_2$.   

 The upper semicontinuity of $R$ from Lemma \ref{lem:upper-semicontinuity-diam} implies
    \[ \limsup_{ x \rightarrow x_0}R(x) \leq R(x_0). \]
    On the other hand we know from Lemma \ref{lem:mp-argument} that no subinterval of $\tilde{x_0}$ can be exposed to $\Omega_2$ by which we mean there is no $x \in \tilde{x_0}$ and $\delta>0$  with $B_{\delta}(x) \setminus \tilde{x_0} = \Omega_2$. Note in two-dimensions, if $(x_n)_{n \geq 1} \subset \partial \Omega$ satisfies $x_n \rightarrow x_0$ and $R(x_n) \rightarrow R(x_0)$ then $\tilde{x_n} \rightarrow \tilde{x_0}$ in the Hausdorff distance (and such sequences can be found).

    As a result there exists sufficiently small $\alpha,\beta > 0$ such that
    \begin{enumerate}
     \item The leaves $\widetilde{\gamma(-\alpha)}$ and $\widetilde{\gamma(\beta)}$  have length at least $3R(x_0)/4$.
    \item The leaves $\widetilde{\gamma(-\alpha)}$  and $\widetilde{\gamma(\beta)}$ can be chosen to make fixed but arbitrarily small angle with $\tilde{x_0}$, by 
    e.g.~\cite[Lemma 16]{CaffarelliFeldmanMcCann00}.
      \item All leaves intersecting the boundary in $\gamma\big([-2\alpha,2\beta]\big)$  have length less than $9R(x_0)/8$ (this holds by the upper semicontinuity of $R$).
      \end{enumerate}
      Moreover $\alpha,\beta$ can be taken as close to $0$ as desired. With the smoothness of $\gamma$, our two-dimensional setting, and the fact that leaves cannot intersect, this significantly constrains the geometry of
\[ A := (Du)^{-1}\big(Du(\gamma\big([-\alpha,\beta]\big)\big) = \bigcup_{t \in [-\alpha,\beta]}\widetilde{\gamma(t)}. \]
      The set $A$ is strictly contained in a set  with left edge $\gamma\big([-2\alpha,2\beta]\big)$ and a {\em vertical right edge of lengths bounded by $2(\beta + \alpha)$,} and top and bottom side lengths bounded by $5R(x_0)/4$ (see Figure \ref{fig:set-geometry}).       Finally we note we can choose sequences $\alpha_k,\beta_k$ satisfying the above requirements and $\alpha_k,\beta_k \rightarrow 0$.
      \begin{figure}[h]
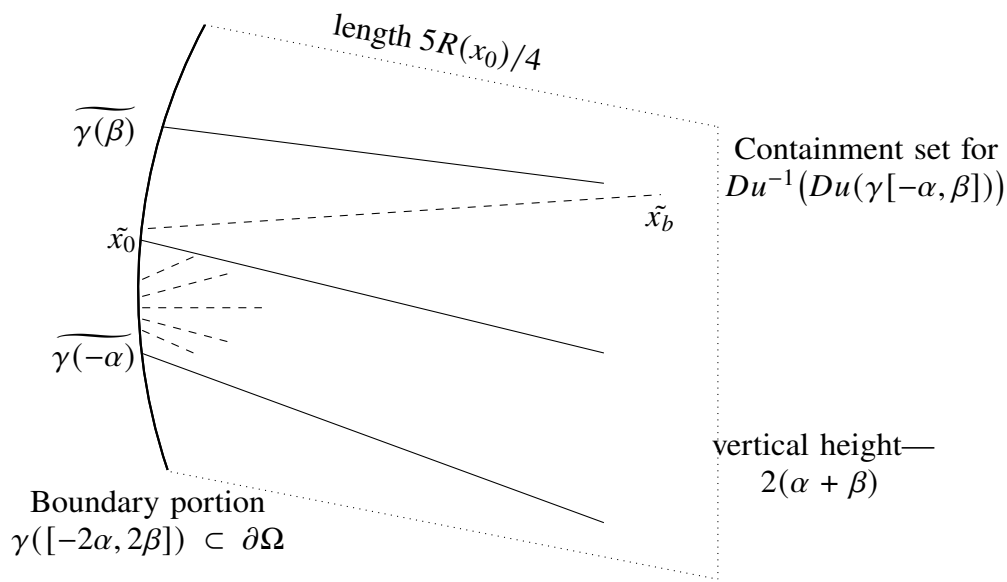

        \centering
        \include{neumann-tikz}
                \caption{Geometry of the constructed set $A$. Note a priori (though not expected) there may be errant leaves such as $\tilde{x_b}$ or those between $\widetilde{\gamma(-\alpha)}$ and $\tilde{x_0}$.  However we have constrained the length of such leaves as less than $9R(x_0)/8$ and, when long, their angles are constrained by the outer leaves $\widetilde{\gamma(-\alpha)}$ and $\widetilde{\gamma(\beta)}$ (which other leaves may not intersect). Thus we obtain the (crude) containment estimate indicated by dotted lines.}
        \label{fig:set-geometry}
      \end{figure}
      
Now, by Corollary \ref{lem:Du-balances}, $\sigma(A) = ((Du)_{\#}\sigma)\big(Du(\gamma([-\alpha,\beta]))\big) = 0$. That is,
\begin{equation}
  \label{eq:becoming}
  0 = \int_{A}(n+1-\Delta u) \, dx + \int_{\gamma(-\alpha,\beta)}(Du-x) \cdot \textbf{n} \, d \mathcal{H}^1.
\end{equation}
Using the local $C^{1,1}$ estimate (\ref{Assumption 2})\footnote{Because $\tilde{x_0}$ doesn't intersect $\partial \Omega$ at both endpoints and $R$ is upper semicontinuous the same is true for all sufficiently close leaves.}, the constrained geometry of $A$, and nonnegativity of $(Du-x)\cdot \mathbf{n}$ already established,
we see \eqref{eq:becoming} implies
      \begin{align*}
      0 &\leq \int_{\gamma(-\alpha,\beta)} (Du-x) \cdot \textbf{n} \, d \mathcal{H}^1  
      \\&\leq \sup_{A}|n+1-\Delta u| \mathcal{H}^2(A) 
      \\&\leq  C R(x_0)(\alpha+\beta)        
      \end{align*} 
      which is precisely the desired estimate. Indeed, employing this estimate with $\alpha_k,\beta_k$ in place of $\alpha,\beta$, dividing by  $\alpha_k+\beta_k$, and sending $\alpha,\beta \rightarrow 0$ we obtain \eqref{eq:neumann-estimate} (after dividing we have an average and $(Du(x_0)-x_0) \cdot \mathbf{n}$ is continuous).

      Now, we explain how to obtain the estimate when $\tilde{x}$ is trivial. If $x \in \partial \Omega$ is such that $\tilde{x}$ is trivial and there is a sequence of nontrivial leaves $\partial \Omega \ni x_k \rightarrow x$ then the estimate follows by the upper semicontinuity of $R$. If there is no such sequence then $x \in \partial \Omega$ lies in a relatively open subset of the boundary $\mathcal{N} = B_{\epsilon_0}(x) \cap \partial \Omega$ on which $u$ is strictly convex.
      Up to a possibly smaller choice of $\epsilon_0$ we may assume $B_{\epsilon_0}(x) \cap \Omega \subset \Omega_2$. Indeed, the alternative is there a sequence of rays, with one end outside $ B_{\epsilon_0}(x) \cap \partial \Omega$ approaching $x$. Since such rays have length bounded below $x$ would also lie in a ray --- a contradiction. This implies $\Delta u = 3$ on $B_{\epsilon_0}(x) \cap \Omega$. Thus the variational derivative satisfies
        \begin{equation}
          \label{eq:var-simplifies}
          \sigma(B_{\epsilon_0}(x)) = \int_{B_{\epsilon_0}(x) \cap \partial \Omega} (Du-x) \cdot \mathbf{n}\,d \mathcal{H}^{n-1}.
        \end{equation}
        On the other hand, strict convexity implies $Du$ is locally injective so that
        \begin{equation}
          \label{eq:balances-app}
          \sigma(B_{\epsilon_0}(x)) =     (Du_{\#}\sigma)(Du(B_{\epsilon_0}(x)\cap\overline\Omega)) = 0
        \end{equation}
where the latter equality follows from Corollary \ref{lem:Du-balances}.  Thus, since $(Du-x) \cdot \mathbf{n}$ is nonnegative, combining \eqref{eq:var-simplifies} and \eqref{eq:balances-app} implies $(Du-x) \cdot \mathbf{n}$  must be zero on $B_{\epsilon_0}(x) \cap \partial \Omega$.
    \end{proof}

\section{Leafwise coordinates parameterizing bunches in the plane}
\label{sec:coord-argum-two}

In this section and the next we study the behavior of the minimizer on $\Omega_1$
and the free boundary $\Gamma = \partial \Omega_1 \cap \partial \Omega_2 \cap \Omega$ in two-dimensions. We introduce
one of our main tools: a coordinate system to study the problem on $\Omega_1$ which
is flexible enough to include the coordinates proposed earlier by the first and
third authors \cite{McCannZhang23+}, and for which, under Assumption \ref{a:localc11}, i.e. \eqref{Assumption 2}, these coordinates may be justified by proving biLipschitz equivalence to Cartesian
coordinates.  Moreover, by combining these coordinates with 
Rochet and
Chon\'e's localization technique (Theorem \ref{T:strictly convex disintegration}) we are able
to provide a radically simpler derivation of the Euler-Lagrange equations \eqref{slope E-L}--\eqref{D:h}  first
expressed in \cite{McCannZhang23+};
c.f.~\eqref{eq:zeroth moment}--\eqref{eq:first moment} and \eqref{eq:spec-zeroth}--\eqref{eq:spec-first} below. Throughout this section we assume we work in a neighbourhood of rays where \eqref{Assumption 2} 
holds. 

Let $\gamma:[-a,b] \rightarrow \mathbf{R}^2$ be a curve parameterizing $\partial \Omega$ in the clockwise
direction with $\dot{\gamma}(t) \neq 0$ 
and write $\gamma(t) = (\gamma^1(t),\gamma^2(t))$.

First we give conditions to ensure a neighbourhood of a ray is foliated by rays.
\begin{lemma}[Local foliation around each tame ray]
\label{lem:foliates-neighbourhood}
  Let $\partial \Omega$ be smooth in a neighbourhood of $x_0 \in \R^2$ satisfying \eqref{Assumption 2}, and that
$(Du(x_0)-x_0) \cdot \mathbf{n} \neq 0$, $u(x_0) \ne 0$ and $\{x_0\} = \tilde{x_0} \cap \partial
\Omega$. Then there exist $\epsilon,r_0 > 0$ such that $\text{diam}(\tilde{x}) \geq r_0$ and
$\tilde{x} \cap \partial \Omega = \{x\}$ for all $x \in \partial\Omega \cap B_\epsilon(x_0)$.
\end{lemma}

\begin{proof}
   We assume without loss of generality that $\gamma(0) = x_0$ and thus there is $\epsilon
> 0$ such that $\gamma : (-\epsilon,\epsilon) \rightarrow \partial \Omega$ is a smooth curve. Proposition
\ref{P:strictly-convex-implies-neumann} implies $\tilde{x_0}$ is nontrivial.  For a possibly
smaller $\epsilon > 0$ and all $t \in (-\epsilon,\epsilon)$, the $C^1(\overline{\Omega})$ regularity of $u$
implies $\widetilde{\gamma(t)}$ is nontrivial, with length bounded below by some
$r_0$ determined by \eqref{eq:neumann-estimate}. Furthermore, the upper semicontinuity of $x
\mapsto\text{diam}(\tilde{x})$ (Lemma \ref{lem:upper-semicontinuity-diam}) implies
for a possibly smaller $\epsilon$, $\widetilde{\gamma(t)} \cap \partial \Omega = \{\gamma(t)\}$ for each $t \in
(-\epsilon,\epsilon)$.
\end{proof}

Recall $x_1 \in \Gamma$ is a tame point of the free boundary
$\Gamma:=\p \Omega_1 \cap \p \Omega_2 \cap \Omega$ provided $\tilde{x_1} \cap \partial \Omega = \{x_0\}$ for an $x_0$ satisfying the hypothesis of Lemma \ref{lem:foliates-neighbourhood}. 

We define $\xi(t) = (\xi^1(t),\xi^2(t))$ as the unit direction vector of the leaf
$\widetilde{\gamma(t)}$ pointing into $\Omega$. This means, with $R(t):=
\text{diam}(\widetilde{\gamma(t)})$,
\[ \widetilde{\gamma(t)} = \{\gamma(t) + r\xi(t) ; 0 \leq r \leq R(t)\},\] and subsequently we
can write a subset of the connected component of $\Omega_1$ containing $x_0$ as
\begin{align} 
\label{eq:n-def} 
\mathcal{N} = \mathcal{N} \cap \Omega_1 &= \bigcup_{t \in (-\epsilon,\epsilon)}\widetilde{\gamma(t)}\\
 \label{leafwise coordinates} &= \{x(r,t) = \gamma(t)+r\xi(t); -\epsilon < t < \epsilon, 0 \leq r \leq
R(t)\},
\end{align} and we take $(r,t)$ as new coordinates for $\mathcal{N}$. We assume $\mathcal{N}$ is a subset of the set on which Assumption \ref{a:localc11}, i.e. \eqref{Assumption 2}, holds.  Because each ray is
a contact set along which $u$ is affine there exists functions $b,m:(-\epsilon,\epsilon) \rightarrow
\mathbf{R}$ such that
\begin{equation}
  \label{eq:general-form} u(x(r,t)) = b(t)+rm(t),
\end{equation} and $Du(x(r,t))$ is independent of $r$.

Our goal is to derive the Euler--Lagrange equations of Lemma \ref{lem:new-euler-lagrange} below which describe the equations the minimizer satisfies in terms of $R$, $\xi$, $m$ and $b$ . First, we record the key structural equalities for the new coordinates in the following lemma, which holds under the biLipschitz hypothesis we eventually establish in Corollary \ref{C:biLipschitz}.   The quantities \eqref{J}--\eqref{eq:hessian} from this lemma also yield a formula for the Laplacian of $u \in C^{1,1}(\mathcal{N})$: 
 \begin{equation}
\label{eq:laplacian} \Delta u = \frac{\xi \times w'}{J(r,t)} =: \frac{\delta(t)}{J(r,t)} .
 \end{equation}

\begin{lemma}[Gradient and Hessian of $u$ in coordinates along tame rays]\label{L:coordinate derivatives}
Suppose $u$ solves \eqref{eq:monopolist} where $\Omega \subset \R^2$ is bounded, open and
convex. Let 
$\gamma,\xi,R,m,b$ on $\mathcal{N}$ be as above \eqref{leafwise coordinates}.   
If the transformation $x(r,t)=\gamma(t)+r\xi(t)$ is biLipschitz on $\mathcal{N}$, then its Jacobian determinant is positive and given by
\begin{equation}\label{J}
0< J(r,t) = \det\left( \frac{\partial(x^1,x^2)}{\partial(r,t)}\right) = \xi\times \dot{\gamma} + r \xi \times \dot{\xi}=: j(t) + rf_{\xi}(t) 
\end{equation}
where $\xi \times \dot{\gamma} = \xi^1\dot{\gamma}^2 - \xi^2 \dot{\gamma}^1$ and similarly $\xi \times \dot{\xi}$
are evaluated at $t$. In addition the following formulas for the gradient and
entries of the Hessian of \eqref{eq:general-form} hold $\mathcal{H}^2$-a.e.:
 \begin{align}
   \label{eq:u1u2}
   Du(x(r,t)) &=
  \begin{pmatrix}
    D_1u\\D_2u
  \end{pmatrix} = \frac{1}{\xi \times \dot{\gamma}}
  \begin{pmatrix}
    \dot{\gamma}^2&-\xi^2\\
    -\dot{\gamma}^1 &\xi^1
  \end{pmatrix}
  \begin{pmatrix}
    m(t)\\b'(t)
  \end{pmatrix}
  =:
  \begin{pmatrix}
    w_1(t)\\w_2(t)
  \end{pmatrix},\\[1ex]
 \label{eq:hessian}  D^2 u (x(r,t)) &= \begin{pmatrix}
     \p^2_{11} u&\p^2_{12}u\\[1ex]
     \p^2_{21}u&\p^2_{22}u
   \end{pmatrix} = \frac{1}{J(r,t)} 
   \begin{pmatrix}
     -\xi^2(t)w_{1}'(t)&\xi^1(t)w_{1}'(t)\\
     -\xi^2(t)w_{2}'(t)&\xi^1(t)w_{2}'(t) 
   \end{pmatrix}.
 \end{align}
\end{lemma}

\begin{proof}
Where the transformation $x(r,t)=\gamma(t) + r\xi(t)$ and  its inverse respectively are Lipschitz, their Jacobian derivatives are easily computed to be:
 \begin{align}
\label{eq:jac-1}
  \frac{\partial(x^1,x^2)}{\partial(r,t)} &=
  \begin{pmatrix}
    \xi^1(t)&\dot{\gamma}^1(t)+r\dot{\xi}^1(t)\\
    \xi^2(t) &\dot{\gamma}^2(t)+r\dot{\xi}^2(t)
  \end{pmatrix}\\
   \label{eq:jac-2}
  \frac{\partial (r,t)}{\partial(x^1,x^2)} &= \frac{1}{J(r,t)}  \begin{pmatrix}
    \dot{\gamma}^2(t)+r\dot{\xi}^2(t)& -\dot{\gamma}^1(t)-r\dot{\xi}^1(t) \\
    -\xi^2(t) &  \xi^1(t)
  \end{pmatrix},
 \end{align}
respectively, with $J(r,t)$ from \eqref{J}.
Next, to obtain the gradient expressions we differentiate equation \eqref{eq:general-form} with respect to $r$ to obtain
\begin{align}
\label{eq:m-comp}  m(t) &= \xi^1(t)u_{1}(x(r,t)) + \xi^2(t)u_{2}(x(r,t)) =\langle \xi,Du\rangle.
\end{align}
Similarly, differentiating \eqref{eq:general-form} with respect to $t$ and equating coefficients of $r$ yields
\begin{align}
\label{eq:m-dash-comp}  m'(t) &= \dot{\xi}^1(t)D_1u(x(r,t))  +\dot{\xi}^2(t) u_{2}(x(r,t)) =\langle\dot \xi, Du\rangle,\\
\label{eq:b-dash-comp}  b'(t) &= \dot{\gamma}^1(t)D_1u(x(r,t)) + \dot{\gamma}^2(t)D_2u(x(r,t)) = \langle\dot \gamma, Du\rangle,
\end{align}
where $D_iu = \frac{\partial u}{\partial x^i}$ and that $Du(x(r,t))$ is independent from $r$ has been used. We solve \eqref{eq:m-comp} and \eqref{eq:b-dash-comp} for $D_1u$, $D_2u$ and obtain \eqref{eq:u1u2}. 
Note the functions $w_1$ and $w_2$ in \eqref{eq:u1u2}  are Lipschitz because $u \in C^{1,1}(\mathcal{N})$. Thus, differentiating the expressions for $D_1u$ and $D_2u$ given by \eqref{eq:u1u2} with respect to, respectively, $x_1$ and $x_2$ and using the Jacobian \eqref{eq:jac-2} gives the formula \eqref{eq:hessian}. 
 
We note two facts about the functions $j$ and $f_{\xi}$ which determine the Jacobian determinant.  First
\begin{equation}
\label{eq:j-nonzero}
  j(t):= J(0,t) = \xi \times \dot{\gamma}>0,
\end{equation}
where the nonnegativity is by our chosen orientation: $\gamma$ traverses $\partial \Omega$ in a clockwise direction and $\xi$ points into the convex domain $\Omega$.  Inequality \eqref{eq:j-nonzero} is strict because $\widetilde{\gamma(t)}$ is nontrivial and  has only one endpoint on $\partial \Omega$ for each $t \in (-\epsilon,\epsilon)$. Next, because $\xi$ is a unit vector, whence $\dot{\xi}$ is orthogonal to $\xi$, we have
\begin{align}
  \label{eq:f-xi-def}
  f_\xi(t) = \xi(t) \times \dot {\xi}(t) = \pm |\dot{\xi}|,
\end{align}
where the value of $\pm$ is determined by the sign of $f_\xi(t)$. From our biLipschitz hypothesis (or nonnegativity of the Laplacian \eqref{eq:laplacian}) the sign of $J(r,t)$ is independent of $r$. Combined with \eqref{eq:j-nonzero} we obtain $J(r,t) > 0$ for $t \in (-\epsilon,\epsilon)$ and $0 \leq r \leq R(t)$.
\end{proof}

Now we combine our raywise coordinates $(r,t)$ with Rochet--Chon\'e's localization (Theorem~\ref{T:strictly convex disintegration}). We obtain the following Euler--Lagrange equations which are central to the remainder of our work.

\begin{lemma}[Poisson data along tame rays]
\label{lem:new-euler-lagrange}
 Let $u$ solve \eqref{eq:monopolist} where $\Omega \subset \R^2$ is bounded, open and
 convex. Let the coordinates $r,t$ and functions $\gamma,\xi,R,m,b$ be as in Lemma \ref{L:coordinate derivatives}. Then the minimality of $u$, more precisely Rochet--Chon\'e's localization, implies the Euler--Lagrange equation relating the fixed and free boundaries
  \begin{align}
  \label{eq:r-neumm} R^2(t)|\dot{\xi}(t)| = 2|\dot{\gamma}(t)| (Du-x) \cdot \mathbf{n} > 0,
  \end{align}
  and the Euler--Lagrange equation for $u$
  \begin{align}
    \label{eq:discont-est}
    3 - \Delta u = \frac{3j(t)+ 3r|\dot{\xi}| - \delta(t) }{J(r,t)} = 
    \frac{3r - 2R(t)}{r + j(t)/|\dot\xi(t)|}.
  \end{align}
\end{lemma}

\begin{proof}
The assumption $(Du-x)\cdot\mathbf{n}>0$ on a boundary segment consisting of ray endpoints implies any 
subset of these endpoints having $\mathcal{H}^1 \mres \partial \Omega$  positive mass is also nonnegligible for the variational derivative so we may  invoke the results of Theorem \ref{T:strictly convex disintegration}.
Applying Lemma~\ref{L:identifying disintegration} to compute the density of the disintegration we obtain that for $\mathcal{H}^1$-a.e. $t \in (-\epsilon,\epsilon)$ that
\[ 0 \leq v (Du-x)\cdot \mathbf{n} |\dot{\gamma}(t)|+ \int_{0}^{R(t)} \left(3-\frac{\delta(t)}{J(r,t)}\right) J(r,t) \, v dr .  \]
The terms outside the integral are evaluated at $x = x(0,t)$. We consider four choices of test functions, $v|_{\widetilde{\gamma(t)}} = \pm 1$ and $v|_{\widetilde{\gamma(t)}} = \pm r$. 
Using these in the localization formula we obtain the equalities
\begin{align}
\label{eq:zeroth moment}  \left((3j(t)-\delta(t))R(t) + \frac{3}{2}R^2(t)f_{\xi}(t) \right) + |\dot{\gamma}(t)|(Du-x)\cdot \mathbf{n} = 0\\
\label{eq:first moment}    (3j(t)-\delta(t))\frac{R(t)^2}{2} + R^3(t)f_{\xi}(t) = 0.
\end{align}
These combine to imply
\begin{align}
  \label{eq:r-neumm-0} R^2(t)f_{\xi}(t) = 2|\dot{\gamma}(t)| (Du-x) \cdot \mathbf{n} > 0.
  \end{align}
  We recall \eqref{eq:f-xi-def} and note \eqref{eq:r-neumm-0} determines the sign $f_{\xi}(t) = |\dot{\xi}|$. Thus
  the Euler--Lagrange equation \eqref{eq:r-neumm} for the free boundary holds. Equation
\eqref{eq:r-neumm} implies $|\dot{\xi}|$ is bounded away from zero and from above
depending only on estimates for the continuous function $(Du-x)\cdot \mathbf{n}$
and our previously given estimate $R(t) > r_0$ (Lemma
\ref{lem:foliates-neighbourhood}).

Combining \eqref{eq:laplacian} with \eqref{eq:first moment} we obtain \eqref{eq:discont-est}.
\end{proof}

\begin{remark}[Tame rays must spread as they leave the boundary] 
\label{R:rays spread}
Comparing \eqref{J} to \eqref{eq:discont-est} we recover 
$f_{\xi}(t)=\xi \times \dot \xi > 0$, which asserts that the Jacobian $J(r,t)$ is an increasing function of $r$: that is, tame rays spread out as they move away from the boundary. This implies all the rays may be extended into  $\Omega_2$ without intersecting each other.
\end{remark}

 Recalling that $\Delta u =3 $ in $\Omega_2$, we see \eqref{eq:discont-est} quantifies how Poisson's equation fails to be satisfied along leaves (with the equation only satisfied at the point $r = 2R(t)/3$). It also implies that when we move from $\Omega_1$ into $\Omega_2$ the Laplacian jumps discontinuously across their common boundary. 

 In Section~\ref{sec:lipschitz-fb} we show this discontinuity yields quadratic separation of $u$ from its contact sets and exploit this to obtain estimates on the Hausdorff dimension of $\Gamma \cap \mathcal{N}$.

Let us conclude by justifying the aforementioned Lipschitz continuity of $\xi$. Note  both \eqref{eq:discont-est} and \eqref{eq:r-neumm} yield Lipschitz estimates. However, since their derivation assumed $\xi$ was Lipschitz we need to redo these calculations with a perturbed, Lipschitz, $\xi_\delta$ and obtain uniform estimates as the perturbation parameter $\delta$ approaches $0$. Notice the Lipschitz constant from the following lemma does not depend on the Neumann values $\|\log ((Du-x)\cdot \mathbf{n})\|_{L^\infty(\mathcal{N}\cap \p\Omega)}$.

 \begin{lemma}[Tame ray directions are Lipschitz on the fixed boundary]
   \label{lem:theta-lipschitz} With $\xi$ defined as above, the function $t \mapsto \xi(t)$ is Lipschitz on the interval $(-\epsilon,\epsilon)$ provided by Lemma \ref{lem:foliates-neighbourhood}, with Lipschitz constant depending only on an upper bound for $\sup_{\mathcal{N}} \Delta u$ and a lower bound on $R(t)$. 
\end{lemma}

\begin{proof}
Assume a collection of non-intersecting rays foliate an open set in $\mathbf{R}^2$ and pierce a smooth curve $\gamma(t)$. Provided the intersection of each ray with the curve occurs some fixed distance $d$ from either endpoint of the ray, the assertion of Caffarelli, Feldman, and McCann \cite[Lemma 16]{CaffarelliFeldmanMcCann00} says the directions $\xi(t)$ (of the ray passing through $\gamma(t)$) is a locally Lipschitz function with Lipschitz constant depending on $\gamma$ and $d$. 

  Thus, in our setting, if we could extend each ray by length $\delta$ outside the domain, $\xi(t)$ would be locally Lipschitz with a constant depending on $\delta$. Then, once we've obtained \eqref{eq:discont-est}  the Lipschitz constant of $\xi$ is independent of $\delta$. A priori such an extension outside the domain may not be possible. Thus our strategy below will be to translate the boundary  distance $\delta$ and use the translated boundary to redefine the $r=0$ axis of the $(r,t)$ coordinates. In this setting the corresponding direction vector $\xi_{\delta}$ is locally Lipschitz, the above calculations are justified, and sending $\delta \rightarrow 0$ gives the Lipschitz estimate on~$\xi$.

  Thus, with $\gamma$ as above let $\mathbf{n}(0)$ be the outer unit normal at $\gamma(0)$ and set
  \[ \eta(t) = \gamma(t) - \delta \mathbf{n}(0).\]
  Because the length of rays intersecting $\gamma(-\epsilon,\epsilon)$ is bounded below by $r_0$, up to a smaller choice of $\epsilon,\delta$ we may assume for each $t \in (-\epsilon,\epsilon)$ that $\eta(t) \in \Omega_1$ and lies distance at least $\delta/2$ from each endpoint of $\widetilde{\eta(t)}$. Let $\xi_\delta(t)$ denote the unit vector parallel to $\widetilde{\eta(t)}$, where by \cite[Lemma 16]{CaffarelliFeldmanMcCann00} $\xi_\delta$ is a locally Lipschitz function of $t$. We may redefine the $(r,t)$ coordinates and write a connected component of $\Omega_1$ containing $\eta(0)$ as
\[\mathcal{N} =  \{x(r,t) = \eta(t) + r \xi_\delta(t) ; t \in (-\epsilon,\epsilon) \text{ and } -R_0(t) \leq r \leq R_1(t) \},\]
where $R_0(t)$ is defined so that $\eta(t) + R_0(t) \xi_\delta(t)$ is the point where the ray $\widetilde{\eta(t)}$ intersects $\partial \Omega$. The function $R_0(t)$ is locally Lipschitz because the curves $\gamma,\eta$ are smooth and $\xi_\delta$ is Lipschitz (though we don't assert that the Lipschitz constant of $R_0$ is independent of $\delta$). Thus all our earlier  computations may be repeated in these new coordinates. The computations leading to \eqref{eq:discont-est} now yield the equation
\begin{equation}
\label{eq:discont-est-new}
    3-\Delta u = |\dot{\xi}_\delta(t)|\frac{3r-2R_1+R_0}{ \xi_\delta \times \dot{\eta}+r|\dot{\xi}_\delta|}, 
\end{equation} 
which we note satisfies $3 = \Delta u$ when $r = (2R_1-R_0)/3$ and thus agrees with our earlier coordinate system (in our modified coordinates $r=(2R_1-R_0)/3$ is the point of distance  $2(R_1+R_0)/3$ from the endpoint of the ray). Similarly we may repeat the calculations which lead to \eqref{eq:r-neumm-0} and read off $f_{\xi_\delta}(t) := \xi_\delta \times \dot{\xi_\delta} \geq c > 0$ for a constant depending on ray length, the violation of the Neumann condition, and the parametrization $\gamma$.  

Evaluating \eqref{eq:discont-est-new} at $r=0$ we obtain a Lipschitz estimate on $\xi_\delta$ which is independent of $\delta$ (using, crucially, the Laplacian estimates provided by \eqref{Assumption 2}). 
The pointwise convergence of $\xi_\delta$ to $\xi$ ensures $\xi$ is Lipschitz and gives a lower bound on $\xi \times \dot{\xi}$.
\end{proof}

As a consequence of Lemma \ref{lem:theta-lipschitz} and computations from the proof of Lemma~\ref{L:coordinate derivatives}, we obtain the following crucial corollary.
Curiously, it shows the biLipschitz constant of the change of variables depends a posteriori on $\|Du\|_{L^\infty(\mathcal N)}$, as opposed to the a priori $\|\Delta u\|_{L^\infty(\mathcal N)}$ dependence from the preceding lemma.

\begin{corollary}[Raywise coordinates \eqref{leafwise coordinates} are biLipschitz]
\label{C:biLipschitz}
Let $\mathcal{N}$ denote the subset of $\Omega_1$ defined in \eqref{eq:n-def}, $(x^1,x^2)$ Euclidean coordinates, and $(r,t)$ the coordinates defined in \eqref{leafwise coordinates}. Then the change of variables from $(x^1,x^2)$ to $(r,t)$ is biLipschitz with Lipschitz constant depending only on $\sup|Du|$, 
$\sup_{t \in (-\epsilon,\epsilon)}|\dot{\gamma}|$, $\inf_{t \in (-\epsilon,\epsilon)} R(t)$, and $\inf_{t \in (-\epsilon,\epsilon)} \xi \times \dot{\gamma}$ (i.e. the transversality of the intersections of rays with the fixed boundary
$\{\gamma(t)\}_{t \in (-\epsilon,\epsilon)}$).
\end{corollary}

\begin{proof}
By Clarke's inverse function theorem for Lipschitz functions \cite{Clarke76},
it suffices to estimate each of the entries in the Jacobian $ \frac{\partial(x^1,x^2)}{\partial(r,t)}$ computed when $x(r,t)$ is Lipschitz in \eqref{eq:jac-1}, its determinant \eqref{J}, and its inverse matrix $\frac{\partial(r,t)}{\partial(x^1,x^2)}$ given by \eqref{eq:jac-2}. Note that applying Clarke's inverse function theorem requires invertibility of each matrix in the Clarke subdifferential. Observe by continuity of $\dot{\gamma}$ and $\xi$ we compute if $(r^i_k,t^i_k) \rightarrow (r,t)$ as $k\rightarrow \infty$
\[A^i := \lim_{(r^i_k,t^i_k) \rightarrow (r,t)} \frac{\partial(x^1,x^2)}{\partial(r,t)} =   \begin{pmatrix}
    \xi^1(t)&\dot{\gamma}^1(t)+r\lim_{t^i_k \rightarrow t }\dot{\xi}^1(t^i_k)\\
    \xi^2(t) &\dot{\gamma}^2(t)+r\lim_{t^i_k \rightarrow t }\dot{\xi}^2({ t^i_k})
\end{pmatrix}.  \]
Any element in the Clarke subdifferential is a convex combination of these $A^i$ and we may compute
\[ 
\det\left(\sum \lambda_i A^i \right) = \xi \times \dot{\gamma} + r \xi \times \left( \sum \lambda_i \lim_{t^i_k \rightarrow t } \dot{\xi}(t^i_k) \right). 
\]
This is nonzero thanks to our assumed positivity of $\inf_{t \in (-\epsilon,\epsilon)} \xi \times \dot{\gamma}$ and the computations in the proof of Lemma \ref{lem:theta-lipschitz} which imply $\xi \times \dot{\xi}$ is bounded below by a positive constant.    

From \eqref{eq:jac-1} it's clear that the entries of this Jacobian permit an estimate from above in terms of  $\sup_{t \in (-\epsilon,\epsilon)}|\dot{\gamma}|$ and $\sup_{t \in (-\epsilon,\epsilon)}|\dot{\xi}|$, where the latter may, in turn, be estimated in terms of $\sup|Du|$ and  $\inf_{t \in (-\epsilon,\epsilon)} R(t)$ thanks to \eqref{eq:r-neumm}. The only additional term we must estimate for the second Jacobian   $\eqref{eq:jac-2}$ is $1/J(r,t)$.  Since $1/J(r,t)$ is decreasing in $r$ by Remark \ref{R:rays spread},
   it suffices to estimate $J(0,t)$ and this is an immediate consequence of \eqref{J} which we recall says $j(t) = \xi \times \dot{\gamma} > 0$, and is justified for $x(r,t)$ Lipschitz as opposed to biLipschitz. 
\end{proof}

\section{On the regularity of the free boundary}
\label{sec:lipschitz-fb}
In this section we study local properties of the free boundary by transformation to an obstacle problem and prove Theorem \ref{thm:fb-regularity}. If $u_1$ denotes the minimal convex extension of $u|_{\mathcal{N}}$ we show that $v :=u - u_1$ solves an obstacle problem with the same free boundary as our original problem.  Standard results for the obstacle problem then imply the free boundary $\Gamma \cap \mathcal{N}$ has Lebesgue measure $0$ and, the stronger result, $\Gamma \cap \mathcal{N}$ has Hausdorff dimension strictly less than $n$.  

We use these estimates to establish that the function $t \mapsto R(t)$ from Section~\ref{sec:coord-argum-two} is continuous on $(-\epsilon,\epsilon)$, Theorem \ref{T:a.e.continuity}. In Proposition~\ref{P:smooth fb} we describe a bootstrapping procedure which shows that if the free boundary --- or rather the function $t \mapsto R(t)$ --- is Lipschitz, then it is $C^\infty.$  In Theorem~\ref{T:unimodal} and its corollary, we show $R$ has a Lipschitz graph away from its local maxima.

\subsection{Transformation to the classical obstacle problem}
Let  $x_1 \in \Omega$ denote the endpoint of a tame ray and $\tilde{x_1} \cap \partial \Omega = \{x_0\}$ as in Lemma \ref{lem:foliates-neighbourhood}. 
Using the coordinates from the 
previous section we consider a subset of $\Omega_1$,
\begin{align*}
  \mathcal{N} = \mathcal{N} \cap \Omega_1 &= \{ \widetilde{\gamma(t)} ; t \in (-\epsilon,\epsilon)\},\\
  &= \{x(r,t) = \gamma(t) + r \xi(t) ; t \in (-\epsilon,\epsilon)  \text{ and }0\leq r \leq R(t)\}.
\end{align*}

In  Remark \ref{R:rays spread} we observed that rays spread out as they leave the boundary. Thus the coordinates $(r,t)$ remain well-defined on an extension of $\mathcal{N}$. We denote this extension by $\mathcal{N}_{\text{ext}}$:
\[ \mathcal{N}_{\text{ext}} = \{x(r,t) = \gamma(t) + r \xi(t) ; t \in (-\epsilon,\epsilon) \text{ and }0\leq r < +\infty\}.\]
On $\mathcal{N}_{\text{ext}}$ we define the minimal convex extension of $u|_{\mathcal{N}}$ by the formula \eqref{eq:general-form}
\[ u_1(x) = b(t) + rm(t) \text{ for $x =x(r,t) \in \mathcal{N}_{\text{ext}}$}. 
\]
Note there is some $\tempalpha > 0$ such that $B_\tempalpha(x_1) \subset \subset \mathcal{N}_{\text{ext}} \cap \Omega $. Moreover on $\mathcal{N}_{\text{ext}} \setminus \Omega_1$ we have, by \eqref{eq:discont-est}, $\Delta u_1 \leq 3-c_0$ for $c_0 > 0$ depending only on a lower bound for $R(t)$ and $(Du-x) \cdot \mathbf{n}$. Let $v =u - u_1$ and let $1_{\{v>0\}}$ denote the characteristic function of $\{v>0\}$. Then  $\Delta u = 3$ in $\Omega_2$ implies that
\begin{align}
 \label{eq:obs-1}  \Delta v &= f(x)1_{\{v>0\}} \geq c_0 1_{\{v>0\}} \text{ in }B_\tempalpha(x_1),\\
  \label{eq:obs-2}  v &\geq 0 \text{ in }B_\tempalpha(x_1),\\
v &= 0 \text{ in }B_\tempalpha(x_1)\cap \Omega_1 \text{ and }v > 0 \text{ in }B_\tempalpha(x_1)\cap \Omega_2
\end{align}
where $f(x) = 3 - \Delta u_1$. Thus $v$ solves the classical obstacle problem in $B_\tempalpha(x_1)$ with the same free boundary as our original problem. The regularity theory for the obstacle problem yields estimates for the measure of the free boundary. What prevents us from using higher regularity theory for the obstacle problem is that on $B_\tempalpha(x_1) \cap \{v > 0\}$,
\[ \Delta v = f(x) = 
\frac{3r-2R(t)}{r + \xi \times \dot{\gamma}/ |\dot{\xi}|},\]
may not be H\"older continuous --- which is the minimum regularity required to apply regularity results for the obstacle problem.  If, $R$ --- and hence the free boundary --- were Lipschitz, $\Delta v$ would also be Lipschitz in which case one can bootstrap to $C^\infty$ regularity of $R$; see Proposition \ref{P:smooth fb}. As a partial result in this direction we prove that $t \mapsto R(t)$ is continuous. We begin with Hausdorff dimension estimates for the free boundary. 

\begin{lemma}[Hausdorff dimension estimate for the free boundary]\label{lem:dim-ests}
  Let $x_0,x_1,\tempalpha$ be as given above,  so that $x_1\in \Omega$ is the endpoint of a tame ray; c.f. Lemma \ref{lem:foliates-neighbourhood}. Then the Hausdorff dimension of $B_{\tempalpha/2}(x_1) \cap \Gamma$ equals $2-\delta_1$ for some $\delta_1$ depending only on $\tempalpha$, $\inf_{B_\tempalpha(x_0) \cap \partial \Omega}(Du-x) \cdot \mathbf{n}$ and $\inf_{t \in (-\epsilon,\epsilon)}R(t)$.   
\end{lemma}
\begin{proof}
  This is a standard result for the obstacle problem once one notes that $f$ in \eqref{eq:obs-1} satisfies $0 < 2c_0 \leq f =  \Delta  v \leq 3 $ on $ \{v>0\}$ for $c_0$ depending only on $\inf_{B_\tempalpha(x_0)}(Du-x) \cdot \mathbf{n}$ and $\inf_{t \in (-\epsilon,\epsilon)} R(t)$. We follow the clear exposition 
 of Petrosyan, Shahgholian, and Uraltseva \cite[\S 3.1, 3.2]{PetrosyanShahgholianUraltseva12} to establish first quadratic detachment, then porosity.

  \textit{Step 1. (Quadratic detachment at free boundary points)} We claim if $x_2 \in B_{\tempalpha/2}(x_1) \cap \Gamma$ then
  \begin{equation}
    \label{eq:non-degeneracy}
    \sup_{B_\rho(x_2)}v \geq \frac{c_0}{2}\rho^2.
  \end{equation}
  Fix such an $x_2$ and  $\overline{x} \in \Omega_2 \cap B_{\tempalpha/2}(x_1)$; we will eventually take $\overline{x} \rightarrow x_2$. Set $w(x) = v(x) - c_0|x-\overline{x}|^2/2.$ On the set $\{v>0\}$ we have $\Delta w > 0$. Thus, the maximum principle implies
  \[ 
  0 \leq v(\overline{x})  = w(\overline{x})\le \sup_{B_\rho(\overline{x}) \cap \{v > 0\}} w  = \sup_{\partial [B_\rho(\overline{x}) \cap \{v > 0\}]} w.  \]
  Because $w < 0$ on $\partial \{v > 0\}$ the supremum is attained at some $x$ on $\partial( B_\rho(\overline{x})) \cap \{v > 0\}$. Because $|x-\overline{x}| = \rho$ we obtain
  \[ 0 \leq v(x) - \frac{c_0}{2}\rho^2,\]
  for some $x \in  \partial B_\rho(\overline{x})$.  We send $\overline{x} \rightarrow x_2$ to establish \eqref{eq:non-degeneracy}.

  \textit{Step 2. (Nondegeneracy implies porosity)} We recall a measurable set $E \subset \mathbf{R}^n$ is called porous with porosity constant $\delta$ if for all $x \in \mathbf{R}^n$ and $\rho>0$ there is $y \in B_\rho(x)$ with
  \[ B_{\delta \rho}(y) \subset B_\rho(x) \cap E^c.\]
  We prove that nondegeneracy, i.e. \eqref{eq:non-degeneracy}, and Caffarelli--Lions's $C^{1,1}_{\text{loc}}$ implies $\Gamma \cap B_{\tempalpha/2}(x_1)$ is porous. Take $x_2 \in \Gamma \cap B_{\tempalpha/2}(x_1)$. Note \eqref{eq:non-degeneracy} implies $\sup_{B_\rho(x_2)} |Dv| \geq c_0\rho/2.$ Indeed, with $\overline{x} \in B_\rho(x_2)$ such that $v(\overline{x}) \geq c_0\rho^2/2$ we have
  \[c_0\rho^2/2 \leq v(\overline{x})-v(x_2) = \int_{0}^{1}Dv(x_2+\tau(\overline{x}-x_2)) \cdot(\overline{x}-x_2) \, d\tau \leq \rho\sup_{B_\rho(x_2)}|Dv|  . \]
 Now, redefine $\overline{x}$ as a point in $B_{\rho/2}(x_2)$ where $|Dv(\overline{x})| \geq c_0\rho/4$. Using that $\Vert v \Vert_{C^{1,1}_{\text{loc}}} \leq M$ (where $\Delta u=3$, $\Delta u_1 \leq 3$ gives the obvious estimate $M=6$),  we have if $x \in B_{\delta \rho}(\overline{x})$ for $\delta = c_0/8M$ then
\[ |Dv(x)| \geq |Dv(\overline{x})| - |Dv(x) - Dv(\overline{x})| \geq c_0\rho/4 - M\delta \rho = c_0\rho/8.\]
Since $Dv \equiv 0$ along $\Gamma$ this proves $B_{\delta \rho}(\overline{x})$ lies in $\Gamma^c$. Thus we've established the porosity condition for balls centered on $\Gamma \cap B_{\tempalpha/2}(x_1)$. To establish the porosity condition for any ball in $\mathbf{R}^n$ we argue as follows. Let $B_\rho(x) \subset \mathbf{R}^n$. We take $\overline{x} \in B_{\rho/2}(x) \cap \Gamma \cap B_{\tempalpha/2}(x_1)$, noting if no such $\overline{x}$ exists we're done. Our porosity result applied on $B_{\rho/2}(\overline{x}) \subset B_{\rho}(x)$ gives porosity of $B_{\rho}(x)$. 

\textit{Step 3. (Conclusion)} As noted in \cite[\S 3.2.2]{PetrosyanShahgholianUraltseva12} by the work of \cite{Sarvas75} a porous set in $\mathbf{R}^n$ has Hausdorff dimension less than $n$. 
\end{proof}

For our next result we shall need the following maximum principle on an unbounded strip,
which is directly implied by \cite[Theorem 1.4, Remark 2.1]{MR1329831}.

\begin{lemma}[Maximum principle on a strip \cite{MR1329831}]
\label{lem:cabre-mp}
  Let $U =  \mathbf{R}\times[-1,0] \subset \mathbf{R}^{2}$. Let $u \in W^{2,n}(U)$ be a bounded solution of 
  \begin{align}
    \left\{
    \begin{array}{rll}
      \Delta u&\geq 0&\text{in }U,\\
      u &= 0&\text{on }\partial U.
    \end{array}
    \right.    
  \end{align}
  Then $u \leq 0$ in $U$.
\end{lemma}

  We now aim to establish continuity of $R$.
 
  \begin{theorem}[Continuity of the tame free boundary]
    \label{T:a.e.continuity}
Taking $\epsilon$ as in Lemma \ref{lem:foliates-neighbourhood},
  the function $R:(-\epsilon,\epsilon) \rightarrow \mathbf{R}$ defined by $R(t) = \text{diam}(\widetilde{\gamma(t)})$ is continuous on  $(-\epsilon/2,\epsilon/2)$.
  \end{theorem}

  The assumptions in the following lemma express precisely the requirements that (1) rays have a single endpoint on the boundary and, (2) there are no nearby rays with endpoint elsewhere on the boundary.

\begin{lemma}[Tame ray lengths vary continuously]
\label{L:R continuity}
Assume a smooth portion of $\partial \Omega_1 \cap \partial \Omega$ is parametrized by $\{\gamma(t)\}_{t \in [-\epsilon,\epsilon]}$ as in Lemma \ref{lem:foliates-neighbourhood} satisfying $\widetilde{\gamma(t)} \cap \partial \Omega = \{\gamma(t)\}$ and adopt the $(r,t)$ coordinates of Section \ref{sec:coord-argum-two}. Assume there is $\delta > 0$ such that
\[ \{x(r,t) ; t \in (-\epsilon,\epsilon), 0 < r < R(t)+\delta\} \subset \Omega_2 \cup \bigcup_{t \in (-\epsilon,\epsilon)} \widetilde{\gamma(t)}.\]
Then each $t_\infty \in [-\epsilon,\epsilon]$ satisfies $R(t_\infty)=\lim_{t \rightarrow t_\infty}R(t)$. 
\end{lemma}

\begin{proof}
  Up to redefining the interval or its orientation, it suffices to prove the result at $t_\infty = \epsilon$. 
  {(This shows a one sided-limit exists.  In the interior of $(-\epsilon,\epsilon)$ it can be upgraded to a two-sided limit by recalling the upper semicontinuity of $R$ from Lemma \ref{lem:upper-semicontinuity-diam}.  The possibility 
  $$2c := R(t_\infty)- \limsup_{t \to t_\infty} R(t)
  >0$$
  can then be excluded by using the strong maximum principle Lemma \ref{lem:mp-argument} after setting the center $x_\infty = x(R(t_\infty)-c, t_\infty)$ and
  reflecting the strictly convex function $u(x) - u(x_\infty) - Du(x_\infty)(x-x_\infty)$ on sufficiently small half-disc in $\Omega_2$ across the ray  $\widetilde{x_\infty}$ where it vanishes.
  At $t_\infty = \pm \epsilon$ the same argument shows the limit coincides with $R(t_\infty)=\text{diam} (\tilde x_\infty)$.)
  }
  
  For a contradiction we assume otherwise. This implies there exists a sequence $t_k \rightarrow \epsilon^-$ with
  \begin{equation}
    \label{eq:cont-for-contra}
    \underline{R}:= \lim_{k \rightarrow\infty} R(t_k) < \limsup_{t \rightarrow t_\infty}R(t)=:\overline{R}.
  \end{equation}
   The proof is three steps: First we show each $x = x(r,t_\infty)$ for $r \in [\underline{R},\overline{R}]$ lies in the free boundary. Then we perform a blow-up. Finally, we conclude the blow-up violates the maximum principle. 

  \textit{Step 1. } 
  To show $x(r,t_\infty) \in \partial \Omega_1$ for each $r \in [\underline{R},\overline{R}]$  we
suppose otherwise.  Then since no such point $x(r,t_\infty)$ can be interior to $\Omega_2$
, the only remaining possibility is
that there is $r \in (\underline{R},\overline{R})$ such that $x(r,t_\infty)$ is an interior
point of $\Omega_1$.
  
  Upper semicontinuity of $R$ implies rays sufficiently close to $x(r,t_\infty)$ have intersection with the boundary close to $x(0,t_\infty) =  \partial \Omega \cap \widetilde{x(r,t_\infty)}$. More precisely for every $\epsilon_{0} > 0$ there is $\delta_{0} > 0$ such that
  \[ \{ \tilde{x} ; x \in B_{\delta_0}(x(r,t_\infty)) \} \cap \partial \Omega \setminus \widetilde{x(r,t_\infty)} \subset B_{\epsilon_{0}}(x(0,t_\infty)). \]
  Thus, our planar foliation implies that because $x(r,t_\infty)$ is an interior point of $\Omega_1$ then $x(\rho,t_\infty)$ is also an interior point of $\Omega_1$ for each $\rho \leq r$ which contradicts that, by \eqref{eq:cont-for-contra}, $x(\underline{R},t_\infty) \in \partial \Omega_1$ and completes Step (1).

  \textit{Step 2. (Blow-up analysis)} Now we choose $r_{0} \in (\underline{R},\overline{R})$ satisfying $r_{0} \geq 4\overline{R}/5$ and set $x_0 = x(r_{0},t_{\infty})$. We perform a blow-up analysis and consider the behaviour of
  \[ u_r(x) := \frac{u(x_0+rx) - u_1(x_0+rx)}{r^2},\]
where we recall $u_1$ denotes the minimal convex extension of $u|_{\mathcal{N}}$. It will be helpful to choose coordinates such that $x_0$ is the origin,
  \[ \tilde{x_0} = \{t e_1 ; t \in [-\alpha,\beta] \text{ for some }\alpha,\beta>0\},\]
  and the positive $e_2$ direction is orthogonal to $\tilde{x_0}$ and satisfies $\gamma'(t_\infty) \cdot e_2 > 0$. At the outset we fix some half ball $B_{\delta}^{-}(x_0) = B_\delta(x_0) \cap \{x^{2} \leq0\}$ with
\[\delta \leq \text{dist}(x_0,\partial \Omega \cap \{x^{2} < 0\}).\]
 Note that $u_r$ is defined on, at least, the half ball $B_{\delta/r}^{-}(0)$ and equals $0$ along $B_{\delta/r}^-(0) \cap \{x^{2} = 0\}$ because $u = u_1$ along rays.

There exists a sequence of $r_k$ with $x_k = -r_ke_2 \in\Omega_1$ and $\text{diam}(\tilde{x_k}) \rightarrow \overline{R}$. Consider the sequence of functions $\{u_{r_k}\}$. We will establish that, up to a subsequence, these functions converge to a limit $u_\infty$ defined on $\mathbf{R}^{2}_{-} := \{(x_1,x_2) ; x_2 \leq 0\}$ satisfying
  \begin{enumerate}
  \item $|D^2u_{\infty}| \in L^{\infty}(\mathbf{R}^{2}_{-})$ with convergence $u_{r_k} \rightarrow u_\infty$ in $C^{1}_{loc}$,
  \item $u_\infty \equiv 0$ along the lines $l_0 = \{x^{2}=0\}$ and $ l_{-1} = \{x^{2} = -1\}$,
  \item there is $x = (x^{1},x^{2})$ between $l_0$ and $l_{-1}$, i.e. satisfying $-1 < x^{2} < 0$, with $u_\infty(x) > 0$,
  \item $\Delta u_\infty \geq 0$.  
  \end{enumerate}
  These facts combine to contradict the maximum principle (Step 3). 
  
  \textit{(1)} Because $r_0 \geq 4\bar{R}/5$, $u$ and $u_1$ satisfy a $C^{1,1}$ estimate in $B_{\delta}^{-}(x_0)$: In the nearby portion of $\Omega_1,$ we have $\Delta u, \Delta u_1 \leq 3$ by \eqref{eq:discont-est}, and in $\Omega_2$, $\Delta u =3$, $\Delta u_1 \leq 3$. Thus, Arzela--Ascoli implies for any fixed $B_{N}^{-} \subset \mathbf{R}^2$ there is $M \in \mathbf{N}$ sufficiently large that the family $\{u_{r_k}\}_{k \geq M}$ is precompact in $C^1(\overline{B_{N}^{-}})$. Hence, up to a subsequence, we obtain $u_\infty : \mathbf{R}^{2}_{-} \rightarrow \mathbf{R}$ satisfying
  \[ u_\infty(x) := \lim_{k\rightarrow \infty}u_{r_k}(x).\]
Moreover $u_{k} \rightarrow u_\infty$ in $C^1(\Omega')$ for every compact $\Omega'\subset \mathbf{R}^{2}_{-}$ and $\Vert D^2u_\infty \Vert_{L^\infty(\mathbf{R}^{2}_{-})} \leq C$   as in \cite{RealRosOton22},\cite{Figalli18}.

\textit{(2)} Clearly, $u_\infty = 0$ along $\{x^{2} = 0\}$, since $u-u_1$ equals $0$ along rays. Moreover we've chosen $r_k$ such that
\[\tilde{x_k} = \{x_k + t (\cos \theta_k,\sin \theta_k) ; t \in [-\alpha_k,\beta_k] \text{for some }\alpha_k,\beta_k>0\}, \]
where $\theta_k\rightarrow 0$ as $k\rightarrow \infty$ and $u-u_1 = 0$ on $\tilde{x_k}$. Note that because $u-u_1 \equiv 0$
along $\tilde{x_k}$, $u_{r_k}$ is equal to $0$ along the line
\[ \{-e_2 + t (\cos \theta_k,\sin \theta_k) ; t \in [-\alpha_k/r_k,\beta_k/r_k] \}.\]
Because $\theta_k \rightarrow 0$ and $u_{r_k}\rightarrow u_{\infty}$ locally uniformly we have $u_\infty = 0$ on $\{x^{2} = -1\}$, thereby completing the proof of \textit{(2)}.

\textit{(3)} This follows as a consequence of the quadratic separation argument in Lemma \ref{lem:dim-ests}. Indeed, that argument gives  a sequence of $z_k$ on $\partial B^{-}_{1/2}$ with $|z_k| = 1/2$ such that $u(x_0+r_kz_k) - u_1(x_0+r_kz_k) \geq c_0 r_k^2$ for a $c_0 > 0$ independent of $k$. Uniform convergence implies a limiting $x \in \partial B_{1/2}^{-}$ where $u_\infty(x) > c_0$ .

\textit{(4)} At points of second differentiability for $u$ and $u_1$ in $\Omega_1$ we have $\Delta u = \Delta u_1$. On the other hand at points of second differentiability in $\Omega_2$ we have $\Delta u_1 < 3$ by \eqref{eq:discont-est}, whereas here $\Delta u =3$, establishing (4).  

  \textit{Step 3. ($u_\infty$ contradicts the maximum principle) }Conclusions (1)--(4) show $u_\infty$ violates Lemma \ref{lem:cabre-mp}, the desired contradiction. 
\end{proof}

\begin{remark}[From Dini continuity to smoothness of the free boundary]
\label{R:Hoelder to smooth}
We have only proved the continuity of $R$ about tame rays. If $R$ were Dini continuous the regularity theory of the obstacle problem, detailed below, improves the Dini continuity of $R$ to Lipschitz regularity of the free boundary in the neighbourhood of any regular point (defined before \eqref{eq:obs-specific} below); see our work with O'Brien \cite[Remark 4.5]{McCannOBrienRankin26+} for caveats and details.  Next we  prove the fourth point of Theorem \ref{thm:fb-regularity}: that if the function $R$ is  Lipschitz then one can bootstrap to a smooth free boundary and minimizer on $\Omega_1$.
 We use this
to show smoothness of the free boundary outside a nowhere dense set in \cite{McCannOBrienRankin26+}; see also Chen, Figalli and Zhang's $C^1_{ loc}$ result \cite[Theorem 1.4]{ChenFigalliZhang26+}.
\end{remark}

\begin{proposition}[Tame part of free boundary is smooth where Lipschitz]
\label{P:smooth fb}
  Let $u$ solve \eqref{eq:monopolist}. Let $x_1\in \Gamma \subset \R^2$ be a tame point of the free boundary and $x_0 = \partial \Omega \cap \tilde{x_1}$. Assume $x\mapsto \text{diam}(\tilde{x})$ is Lipschitz on some $B_\epsilon(x_0) \cap \partial \Omega$. Then there is $\delta > 0$ such that $B_\delta(x_1) \cap \Gamma$ is a smooth curve. On the portion $\mathcal N$ of $\Omega_1$ consisting of rays which intersect $B_\delta(x_1)$,
  the transformation $(r,t) \to x(r,t)$ of \eqref{leafwise coordinates} is a smooth diffeomorphism and $u \in C^\infty(\mathcal N 
    )$.
\end{proposition}

\begin{proof}
  We prove by induction on $k=0,1,2,\dots$ that there is some neighbourhood of $x_0$ on which the curve $\partial \Omega_1$ 
  and the function $R$ are both $C^{k,\alpha}$ for some $0 < \alpha < 1$,  while the coordinate transformations and $u$ are 
  $C^{k+1,\alpha}$ in $B_\delta(x_1) \cap \Omega_1$.
  { For $k = 0$ our assumption is  that $t \mapsto R(t)$  is Lipschitz, and from 
  Corollary \ref{C:biLipschitz} the coordinate transformations are biLipschitz} on $\mathcal B_\delta(x_1) \cap \Omega_1$.  From the formula \eqref{eq:r-neumm},  reproduced here for the readers convenience 
\begin{align}
  \label{eq:r-neumm2} 
  R^2(t) |\dot{\xi}(t)| = 2|\dot{\gamma}(t)| (Du-x) \cdot \mathbf{n},
  \end{align}
Caffarelli and Lions $u\in C^{1,1}_{\text{loc}}$, and the fact that $|\xi|=1$,
{ it follows} that $t\mapsto\dot{\xi}(t)$ is also Lipschitz,
hence the coordinate transformations improve to $C^{1,1}$ by 
the Jacobian expressions \eqref{eq:jac-1}--\eqref{eq:jac-2};
here we have used Corollary \ref{C:biLipschitz} to propagate regularity information about $Du$ down the rays from the free to the fixed boundary.
From \eqref{eq:discont-est}, i.e.
\begin{equation}
\label{eq:discont-est2}
  3 - \Delta u   = \left| \frac{\dot{\xi}(t)}{J(r,t)}\right|(3r-2R),
\end{equation}
we see $\Delta u \in C^{0,\alpha}$, hence 
the regularity theory for Poisson's equation implies $u \in C^{k+2,\alpha}$ when $k=0$ (one derivative more than needed).

Now assume the inductive hypothesis for some fixed $k$. From
\eqref{eq:r-neumm2}--\eqref{eq:discont-est2} we again deduce $u$ has $C^{k,\alpha}$
Laplacian in $\Omega_1$. Thus the regularity theory for Poisson's equation implies
$u \in C^{k+2,\alpha}$. The regularity theory for the obstacle problem (where the
obstacle has a $C^{k,\alpha}$ Laplacian; (due to Caffarelli
\cite{Caffarelli77,Caffarelli98}, and Kinderlehrer \cite{Kinderlehrer76} with Nirenberg
\cite{KinderlehrerNirenberg77}, though the clearest statement we've found is by
Blank \cite{Blank00,Blank01}) implies the free boundary is $C^{k+1,\alpha}$. Note to
apply the classical regularity theory for the obstacle problem we are using
that the Lipschitz regularity of $R$ implies the set $\overline \Omega_1 = \{v =
0\}$ has positive density at each boundary point; here $v = u - u_{1}$ as in
\eqref{eq:obs-1} . Now that $R$ is $C^{k+1,\alpha}$ the same is again true for
$\dot{\xi}$ by equation \eqref{eq:r-neumm2} since the smoothness $Du\in C^{k+1,\alpha}$
established in $\Omega_1\cap B_\delta(x_1)$ { propagates} down the rays from the free to the
fixed boundary using the $C^{k+1,\alpha}$ coordinate transformations; these
transformations then improve to $C^{k+2,\alpha} $ by equations
\eqref{eq:jac-1}--\eqref{eq:jac-2} so the induction is established and the proof
is complete.
\end{proof}

Theorem \ref{thm:fb-regularity}(1) and (4) are obtained by combining Lemmas \ref{lem:dim-ests} and \ref{P:smooth fb}; parts (2) and (3) will be established in Theorem \ref{T:unimodal} of the next section.

\subsection{Criteria for the tame part of the free boundary to be Lipschitz}
Having deduced regularity of the free boundary when it is Lipschitz we now turn our attention to the question of characterising the set on which the free boundary is Lipschitz. We will rely on the well known Caffarelli dichotomy for the blow-up of solutions to the obstacle problem. We recall that blowing-up at the edge of the contact region in the classical obstacle problem (without convexity constraints) led Caffarelli to formulate his celebrated alternative \cite{Caffarelli77,Caffarelli98}: 
at each point in the free boundary,  the density of the contact region is either $0$ (called {\em singular}) or $\frac1 2$ (called {\em regular});
it cannot equal $1$ because of quadratic detachment
(as in e.g.\ the proof of Lemma~\ref{lem:dim-ests}). Furthermore, the following dichotomy holds for blowups of solution to equations of the form  $\Delta u(x) = f(x)1_{\{u>0\}}(x)$ in a domain $\Omega$ where $f$ is continuous in the following sense: Take $x_0 \in \Omega \cap \partial \{u = 0\}$ and a sequence $r_k \rightarrow 0$. Note that up to taking a sequence the limit 
\begin{equation}
    \label{eq:obs-specific}
u_0(x) := \lim_{k \rightarrow \infty} \frac{u\big(r_k (x-x_0) \big)}{r_k^2},
\end{equation}
exists and is a globally defined solution of $\Delta u(x) = f(x_0) 1 _{\{u > 0\}}(x)$. Moreover $u_0$ is {convex}, 
and either a {quadratic polynomial} or a rotation
of the {\em half-parabola solution}
$$
{w(x_1,\ldots,x_n) = 
\begin{cases}
\frac12 x_1^2 & {\rm if}\ x_1 >0 
\\ 0 & {\rm else}.
\end{cases}
}$$
Unfortunately, in our setting we only know $f \in L^\infty_{\text{loc}}$ and not the H\"older continuity required for higher regularity \cite{Kinderlehrer76};
see the sequelae \cite{ChenFigalliZhang26+}\cite{McCannOBrienRankin26+} for  improvements.

A real-valued function $S$ on an interval $J \subset \R$ is called {\em unimodal} if it is monotone,  or else if it attains its maximum on a 
(possibly degenerate interval) $I \subset J$, with $S$ being non-decreasing throughout the connected component of $J\setminus I$ to the left of $I$,
and non-increasing through the connected component 
 to the right of $I$. The following lemma shows lower semicontinuous functions are unimodal away from their local minima.

\begin{lemma}[Lower semicontinuous unimodality away from local minima]
\label{L:unimodal}
Let $S:E \longrightarrow [-\infty,\infty)$ be lower semicontinuous on an interval $E\subset \R$.  Let $T$
denote the subset of $E$ consisting of local minima for $S$, and $\overline T$
its closure.  Then $S$ is unimodal on each connected component $J$ of $E \setminus
\overline T$.
\end{lemma}

\begin{proof}
Fix any open interval $J\subset E \setminus \overline{T}$.  We claim $S$ is unimodal on~$J$.
Since $S$ is lower semicontinuous but has no local minima on $J$, for each $c \in
\R$ it follows that $J(c) := \{t \in J \mid S(t)>c\} = \cup_i (a_i,b_i)$ is a countable
union of open intervals on which $S>c$ with $S \le c$ on $J\setminus J(c)$.  If there
were more than one open interval in this union, say $(a_1,b_1)$ and $(a_2,b_2)$
with $b_1 \le a_2$, then $S$ would attain a local minimum on the compact set
$[b_1,a_2] \subset J$, contradicting the fact that $J\subset E\setminus \overline T$ is an interval.  Thus the set
$J(c)$ consists of at most one open interval, which is monotone nonincreasing
with $c \in \R$.  Let $c_0$ denote the infimum of $c \in \R$ for which $J(c)$ is
empty, and set $I = \cap_{c<c_0} J(c)$. 
Then $S$ is non-decreasing to the left
of~$I$, non-increasing to the right of~$I$, and attains its maximum value on~$I$.
\end{proof}

We apply this lemma to the diameter $R=-S$ of the rays along the tame part of the free boundary to deduce the free boundary
is Lipschitz away from its local maxima.

\begin{theorem}[Tame free boundary is Lipschitz away from local maxima]
\label{T:unimodal}
Let $\gamma:E \longrightarrow \p \Omega$ with $\dot{\gamma}(t)\ne 0$ for $t \in E:= (-\epsilon,\epsilon)$ smoothly parameterize a fixed boundary interval 
throughout which the Neumann condition \eqref{Neumann} is violated.  Let $T$ denote the subset of $E$ consisting of local maxima for 
$R(t):=\text{\rm diam}(\widetilde{\gamma(t)})$,  and $J$ any connected component of $E \setminus \overline T$, where $\overline T$ is the closure of $T$.
Then $R$ extends continuously to $\overline J$ and its graph is a Lipschitz submanifold
of $\overline J \times \R$.  Similarly,  the graph of $F(t):=\gamma(t) + R(t) \xi(t)$ over $\overline J$ is a Lipschitz submanifold of $\overline \Omega$ 
(except perhaps at $t = \pm \epsilon$), and $F$ is continuous on $\overline J$.
\end{theorem}

\begin{proof}
Corollary \ref{C:biLipschitz} shows the coordinates $x(t,r) = \gamma(t) + r\xi(t)$ are locally biLipschitz on $E \times (0,\infty)$,  
so the final sentence follows from showing $R$ has a continuous extension to $\overline J$ whose graph is a Lipschitz submanifold.

Lemma \ref{lem:upper-semicontinuity-diam} asserts $R$ is upper semicontinuous on $\overline J$.  Unless $R$ is monotone on~$J$,  Lemma \ref{L:unimodal} shows $J$ 
decomposes into two subintervals on which $-R$ 
is monotone and they overlap at least at one point $t'$.  
Although a monotone function need not be Lipschitz --- or even continuous --- its graph has Lipschitz constant at most~$1$.
A discontinuity in $R$ on the closure of either of these 
subintervals can be ruled out as in the proof of Lemma \ref{L:R continuity} (or by Theorem \ref{T:a.e.continuity} in the interior).
Thus $R$ is continuous on $\overline J$.   The graph of $R$ on $\bar J$ is obviously Lipschitz,  except perhaps when the minimum value of $R$ is uniquely 
attained at some $t' \in J$.  Since $t'$ is a local minimum,  $R$ is continuous at $t'$ hence Caffarelli's alternative holds for the blow-up at $F(t')$:
the Lebesgue density of $\Omega_1$ at $F(t')$ cannot be zero since $R(t')$ is a local minimum,  so it must be exactly $1/2$ \cite{Caffarelli98,Caffarelli77}.
The blow-up limits of $u_2-u_1$ at $F(t')$ all coincide with the same half-parabola,  and $R$ is differentiable at $t'$.  
The Lipschitz graph of $R$ to the left of $t'$ shares the same tangent at $t'$ as the Lipschitz graph of $R$ to the right of $t'$, which completes the proof.
\end{proof}

Our next corollary shows that the tame part of the free boundary can only fail to be locally Lipschitz when oscillations with unbounded frequency cause local maxima of $R$ to accumulate, or when an isolated local maximum forms a cusp.  In the latter case,  the tame free boundary is locally piecewise Lipschitz and the perimeter of
$\Omega_1$ is locally finite in this region.

\begin{corollary}[Is the tame free boundary piecewise Lipschitz?]
\label{C:unimodal}
Assume $T$ has only finitely many connected components 
in Theorem \ref{T:unimodal} 
and $R$ is constant on each of them
--- as when $R$ has only finitely many local maxima on $E$.
Then the graph of $F(t):=\gamma(t) + R(t)\xi(t)$ is a piecewise Lipschitz submanifold of $(-\epsilon+\delta,\epsilon-\delta) \times (0,\infty)$ 
for each $\delta>0$.  Moreover,  if the graph of $F$ fails to be Lipschitz at $F(t')$ for some $t' \in E$,
then $R$ has an isolated local maximum at $t'$ and $\Omega_1$ has Lebesgue density zero at $F(t')$.
\end{corollary}

\begin{proof}
Under the hypotheses of Theorem \ref{T:unimodal},  
assume $T$ has only finitely many connected components and $R$ is constant on each of them.
Then these components must be intervals which are relatively closed in $E$: 
otherwise the upper semicontinuous function $R$ has a jump increase at the end of one of them,
which leads to a segment in the graph of $u_2$ --- producing the same contradiction to Lemma \ref{lem:mp-argument} as in the proof of Lemma \ref{L:R continuity}.
Thus $\overline T=T$. For each of the open intervals $J$ comprising $E \setminus \overline T$,
 Theorem \ref{T:unimodal} already asserts that $R$ is continuous and has Lipschitz graph on $\overline J$;
 the only question is whether the graph extends past each endpoint of $\overline J$ in $E$ in a Lipschitz fashion.
If the endpoint of $\overline J$ belongs to a nondegenerate interval in $T$ this is obvious.
When the endpoint of $\overline J$ is an isolated point $t'$ in $T$,  then Lemma~\ref{L:unimodal} shows 
$R$ nearby is monotone on either side hence must be continuous at $t'$ to
avoid an affine segment in the graph of $u_2$ as before.  Now Caffarelli's alternative applies, so the density of $\Omega_1$ 
at $F(t')$ must be either $0$ or $1/2$.  In the latter case $R$ has a Lipschitz graph in a neighbourhood of $t'$,
as in the proof of Theorem~\ref{T:unimodal}, hence the corollary is established.
\end{proof}

\begin{remark}[A partial converse]
\label{R:unimodal}
If $\Omega_1$ fails to have Lebesgue density $1/2$ at some tame point $x'=F(t') \in \Omega \cap \p \Omega_1$, 
then there is no neigbourhood of $x'$ whose intersection with $\Omega_1$ is a Lipschitz domain.
This follows from the Caffarelli alternative,  which requires $\Omega_1$ to have Lebesgue density $0$ at 
$x'$ \cite{Caffarelli98,Caffarelli77}.
\end{remark}

 \section{Bifurcations to bunching in 
 the family of square examples}
 \label{sec:rc-example}

In this section we apply our results and techniques to a concrete example and
completely describe $C^{1,1}(\overline{\Omega})$ solutions on the domain $\Omega = (a,a+1)^2$. We prove
Theorem \ref{thm:description-on-square}. Indeed, throughout this section all results are conditional on the assumption $u \in C^{1,1}(\overline{\Omega})$. For $a\ge \frac{11}{2} -\sqrt{2}$ the solution is as
hypothesized in the earlier work of McCann and Zhang
\cite{McCannZhang23+}. A key strategy involves showing first unimodality and then monotonicity of the normal distortion $(x-Du(x))\cdot \mathbf n$ along each edge of the square. From this we deduce the leaves $\Omega_{1}^{0}$ which enter $\Omega$ and have more than one endpoint on the boundary can only have one
endpoint on $\Omega_{W}$ and the other on $\Omega_{S}$; on these leaves the solution is given
explicitly in~\cite{RochetChone98,McCannZhang23+}. We let $\Omega_{1}^-$ denote the
set of
leaves with one endpoint in $\Omega$ and the other on $\Omega_{W}$, and, finally,
let $\Omega_{1}^+$ denote the set of leaves with one endpoint in $\Omega$ and the other on
$\Omega_{S}$. Because, in the course of our proof, we prove rays $\Omega_1$ 
which enter $\Omega$ do not intersect
$\Omega_N$ or $\Omega_E$, we have $\Omega_1 = \Omega_{1}^{0} \cup \Omega_1^-\cup \Omega_1^+ \cup \Omega_1^\infty$  where $\Omega_1^\infty \subset \Omega_N \cup \Omega_E$ and consists of at most two rays.

Our main tool to study the minimizer is the coordinates introduced in Section
\ref{sec:coord-argum-two}.  Let us consider a component of $\Omega_1^-$ consisting
of leaves with one endpoint on the boundary $\Omega_{W} = \{a\} \times [a,a+1]$ and the
other interior to $\Omega$. The argument is similar on each side. We may take the
angle $\theta$ made by leaves with the horizontal (that is, with the vector $(1,0)$), as
the parametrization coordinate of our boundary (i.e. $t=\theta$ and $\xi(t) =
(\cos\theta,\sin\theta)$). Then $\gamma(\theta) = (a,h(\theta))$ and $(r,\theta)$ satisfy
 \[ x(r,\theta) = (a + r \cos \theta, h(\theta) + r \sin\theta), \] where $h(\theta)$ is the height at
which the leaf that makes angle $\theta$ with the horizontal intersects $\Omega_{W}$. We
work in a connected subset of $\Omega_1$
 \[ \mathcal{N} = \mathcal{N} \cap \Omega_1^{-} = \{(r,\theta) ;\underline{\theta} \leq \theta \leq \overline{\theta} \text{ and } 0 \leq r \leq R(\theta)\}. \]
 In this setting \eqref{eq:zeroth moment} and \eqref{eq:first moment} become
\begin{align}
\label{eq:spec-zeroth}
 0&= (3h'\cos \theta - m''-m)R(\theta)+\frac{3}{2}R^2(\theta) + h'(\theta)(Du({x}) - {x}) \cdot \mathbf{n} \\
0 &=   (3h'\cos \theta - m''-m)\frac{R^2(\theta)}{2} + R^3(\theta) 
\label{eq:spec-first}
\end{align}
where we use the prime notation for derivatives of roman characters as opposed
to the dot notation for derivatives of greek characters, and equation
\eqref{eq:r-neumm} becomes
\begin{equation}
  \label{eq:spec-r-neumm}
  R^2(\theta) = 2h'(\theta)(Du-x) \cdot \mathbf{n}.
\end{equation}
Note when we parameterize with respect to $\theta$, $|\dot{\xi}| = 1$ so $m'$ is
Lipschitz by \eqref{eq:m-dash-comp}. We've used that $h'(\theta) > 0$ as is easily seen by first working with the parametrization $\gamma(t) = (a,t)$ and
the angles $\xi(t) = (\cos \theta(t), \sin\theta(t) )$, for which the identity $\xi \times \dot{\xi}
>0$ derived in Section \ref{sec:coord-argum-two} implies $\theta'(t) >0$.  Equations
\eqref{eq:spec-zeroth} -- \eqref{eq:spec-r-neumm} yield a new, expedited, proof
of the Euler--Lagrange equations in $\Omega_{1}^\pm$ originally derived by the first
and third author \cite{McCannZhang23+} via a complicated perturbation
argument. Solving \eqref{eq:m-comp} and \eqref{eq:m-dash-comp} gives
\begin{align}
  \label{eq:u1theta}  u_1(x) &=m(\theta) \cos \theta - m'(\theta) \sin \theta\\
  \label{eq:u2theta} \text{and}\quad u_2(x) &= m(\theta) \sin \theta + m'(\theta) \cos \theta.
\end{align}
Thus, along $\Omega_W$
\[ (Du(x_0) - x_0) \cdot \mathbf{n} = a - u_1(x_0) = a + m'(\theta)\sin(\theta) - m(\theta)\cos\theta\]
Substituting into \eqref{eq:spec-r-neumm} we obtain
\[ R^2(\theta) = 2h'(\theta) (a + m'\sin(\theta) - m(\theta)\cos\theta).\]
After multiplying by $\cos \theta$ and solving \eqref{eq:spec-first} for $h'\cos \theta$
we obtain \eqref{slope E-L}, which coincides precisely with equation (4.22) of
\cite{McCannZhang23+}.

We obtain Theorem \ref{thm:description-on-square} as a combination of Lemmas.  As required by the theorem, we henceforth make the tacit assumption $a\ge 0$.

\begin{lemma}[Exclusion includes right-angled triangle in lower left corner]
\label{L:lower left}
  Let $u$ minimize \eqref{eq:monopolist} with $\Omega = (a,a+1)^2$. Then
$\mathcal{H}^2(\Omega_0)>0$ and  $\Omega_0 \subset [a,c]^2$ for some $c>a$ 
satisfying $\Omega_0\cap \partial\Omega = [a,c]^2 \cap \partial
\Omega$.
\end{lemma}

\begin{proof}
Whenever $\Omega$ is a subset of the first quadrant, symmetry shows the minimizer
satisfies $D_i u \ge 0$. Since the inclusion $\Omega_0 \subset \{u=0\}$ of Theorem~\ref{thm:structure-convexity} becomes an equality if $\Omega_0$ is nonempty,
monotonicity of convex gradients implies $\Omega_0 = \{u=0\} \subset [a,c]^2$ for some
$c>a$ such that $\Omega_0 \cap \partial \Omega$ = $[a,c]^2\cap \partial \Omega$ in this case; here symmetry across
the diagonal and the fact that $\{u=0\}$ is closed have been used. 
Armstrong has proved that $\Omega_0$ has positive measure whenever $\Omega$ is strictly
convex \cite{Armstrong96} and this result has been extended to general benefit
functions by Figalli, Kim, and McCann \cite{FigalliKimMcCann11}. It is
straightforward to adapt their proof to our setting.  Indeed, convexity of
$\Omega_0$ and symmetry across the diagonal means $\mathcal{H}^{1}(\partial \Omega_0 \cap \partial \Omega) > 0$ implies
$\mathcal{H}^2(\Omega_0) > 0$ and this implication is the only place strict convexity is used
in \cite[Theorem 4.8]{FigalliKimMcCann11}.
\end{proof}

Now we use the following lemma to establish that along $\Omega_N \cup \Omega_E$ there is no outer distortion, that is $(Du - x)\cdot \mathbf{n} = 0$.

\begin{lemma}[Neumann condition on boundary rays]
\label{lem:no-nonzero-neumann-boundary-ray}
  If there is a nontrivial ray $\tilde{x} \subset \partial \Omega$ on which $u \neq 0$,
  then $(Du - x) \cdot \mathbf{n} = 0$ holds throughout $\tilde{x}$
  (except possibly at its endpoints).
\end{lemma}
\begin{proof}
  Observe $\tilde{x} \subset \partial \Omega$ implies the boundary ray is either a subset of a single edge or has one endpoint at a vertex. 
   If along a $\mathcal{H}^1$ positive measure portion of $\tilde{x}$ we have $(Du - x) \cdot \mathbf{n} > 0$ then the push forward $Du_{\#}\sigma$ of the variational derivative,  considered in Corollary \ref{lem:Du-balances}, has an atom at $Du(x)$ contradicting Corollary \ref{lem:Du-balances}. 
\end{proof}

\begin{proposition}[No normal distortion along top right boundaries]
\label{prop:no-outward}
  Let $u$ solve \eqref{eq:monopolist} with $\Omega = (a,a+1)^2$. Then $(Du-x)\cdot
\mathbf{n} = 0$ throughout $\Omega_N$ and~$\Omega_E$. 

\end{proposition}

 \begin{proof}
  For a contradiction we assume (without loss of generality, by Proposition~\ref{prop:neumann-sign}) there is $x_0 = (a+1,t_0)\in \Omega_{E}$, 
at which $(Du(x_0) - x_0) \cdot \mathbf{n} > 0$;
 when $x_0 \in \Omega_E$ is a vertex of $\Omega$ we interpret $\mathbf{n}=(1,0)$. With this interpretation the continuity of $Du$ implies we may in fact assume, without loss of generality, that $x_0$ lies in the relative interior of $\Omega_{E}$. Thus $\tilde{x_0}$ has
positive diameter and the same is true in a relatively open portion of the
boundary, by Proposition~\ref{P:strictly-convex-implies-neumann}. 
Note Lemma \ref{lem:no-nonzero-neumann-boundary-ray} rules out 
$\tilde{x_0} \subset \Omega_E$.

Working on $\Omega_{E}$, it is
convenient to let $\theta$ denote the clockwise angle a ray makes with the inward normal $(-1,0)$. Thus $\theta > 0$ corresponds to a ray with nonpositive
slope. Parametrizing the boundary as $\gamma(\theta) = (a+1,h(\theta))$ using this $\theta$, the
derivation of the equation \eqref{eq:spec-r-neumm} is unchanged along $\Omega_E$. In
particular $h'(\theta) > 0$, which is most easily seen by beginning with the
clockwise oriented parametrization $\gamma(t) = (a+1,a+1-t)$ and $\xi(t) = (-\cos
\theta(t),\sin \theta(t))$ in the coordinate arguments of Section
\ref{sec:coord-argum-two} and recalling $\xi \times \dot{\xi} > 0$.

Clearly $\theta|_{\tilde{x_0}} \ge 0$ or $\theta|_{\tilde{x_0}} < 0$; we will derive a contradiction in either case. 

 \textit{Case 1. (Nonpositively sloped leaf).} First let's assume the leaf has
nonpositive slope (i.e. $\theta \ge 0$). The inequality $h'(\theta) > 0$ implies leaves above, but in the same connected
component of $\Omega_1$, as $x_0$ with one endpoint on $\{a+1\}\times[t_0,a+1]$ are also
nonpositively sloped.

At the endpoint of each leaf the Neumann condition is not satisfied, that is
$(Du(x)-x)\cdot \mathbf{n} > 0$, equivalently $D_1u(x) > a+1$ (by the sign
condition on the Neumann value). On the boundary portion where leaves have
nonpositive slope,
$t \mapsto u_1(a+1,t)$ is a nondecreasing function (see Figure \ref{fig:angle-determines-u1}(A)). 
Thus
$D_1u(a+1,t) > a+1$ for all $t \in
[t_0,a+1]$ and, by Theorem \ref{thm:structure-convexity} and
Proposition \ref{P:strictly-convex-implies-neumann},
 each $x \in \{a+1\} \times [t_0,a+1]$ is the endpoint of
a nontrivial leaf of nonpositive slope; 
(such leaves cannot accumulate onto the convex set $\Omega_0=\{u=0\}$ in view of Lemma \ref{L:lower left}).
We consider the following dichotomy and
derive a contradiction in either case: there is a sequence of leaves
approaching $(a+1,a+1)$ with one endpoint on $\Omega_E$ and the other in $\Omega$ or else
there is not, in which case all sequences of leaves approaching $(a+1,a+1)$
have one end on $\Omega_E$ and the other on $\Omega_N$.

  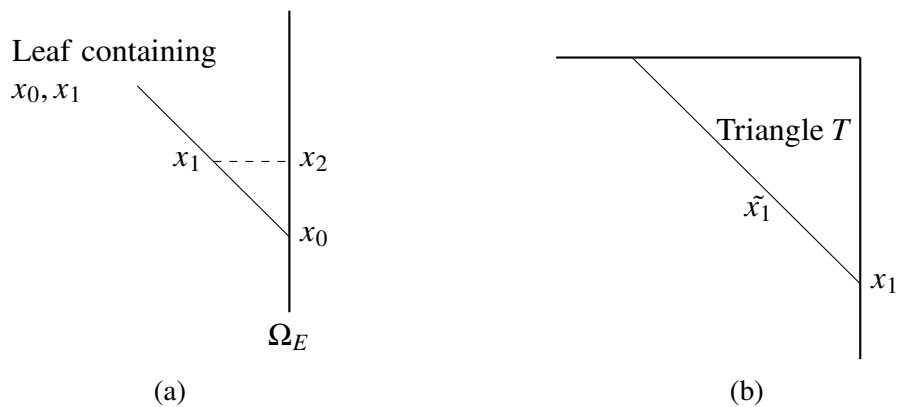
\begin{figure}
     \centering
     \begin{subfigure}[b]{0.4\textwidth}
           \begin{tikzpicture}
      \draw [thick] (0,0) -- (0,4);
      \node [below] at (0,0) {$\Omega_E$};
      \draw (0,1) -- (-2,3);
      \node [above, left,text width=3cm] at (-0.5,3.2) {Leaf containing\\ $x_0,x_1$};
      \draw [dashed] (-1,2) -- (0,2);
      \node [right] at (0,1) {$x_0$};
      \node [left] at (-1,2) {$x_1$};
      \node [right] at (0,2) {$x_2$};
    \end{tikzpicture}
         \centering
         \caption{}
%         \label{fig:y equals x}
     \end{subfigure}
     \hfill
     \begin{subfigure}[b]{0.4\textwidth}
       \begin{tikzpicture}
        % \draw[step=1,help lines] (0,0) grid (4,4);
         \draw [thick] (0,4) -- (4,4);
         \draw [thick] (4,0) -- (4,4);
         \draw (4,1) -- (1,4);
         \node at (3,3) {Triangle $T$};
         \node [below,left] at (3,2) {$\tilde{x_1}$};
         \node [right] at (4,1) {$x_1$};
       \end{tikzpicture}
         \caption{}
%         \label{fig:three sin x}
       \end{subfigure}
           \caption{(A) Explanation of why $t \mapsto u_1(a+1,t)$ is monotone nondecreasing when leaves make positive angle with the horizontal (i.e. have nonpositive slope). Because $x_1 \in \tilde{x_0}$, $Du(x_0) = Du(x_1)$. Then monotonicity of the gradient of a convex function implies $D_1u(x_2) \geq D_1u(x_1) = D_1u(x_0)$. Thus $t \mapsto D_1u(a+1,t)$ is nondecreasing. \\
    (B) Since $Du$ is constant along $\tilde{x_1}$, monotonicity of the gradient implies $D_1u(x) \ge D_1u(x_1)$ and $D_2u(x)  \ge D_2u(x_1)$ for all $x \in T$.}
    \label{fig:angle-determines-u1}
\end{figure}

\textit{Case 1a. (All leaves approaching the vertex have one end in the interior).}  In the first case
take a sequence $(x_k)_{k \geq 1} \subset \Omega_E$ with $x_k = (x_k^1,x_k^2) = (a+1,x^2_k)$
increasing along $\Omega_E$ to a limit
$x_\infty$ with $(a+1,a+1) \in \tilde x_\infty$.
{These rays can not limit to a ray contained in the boundary; such a ray would violate 
Lemma~\ref{lem:no-nonzero-neumann-boundary-ray}.} Thus, we obtain $\tilde x_\infty=\{(a+1,a+1)\}$ in view of Lemma~\ref{L:lower left}. Moreover,
we can take $\tilde x_k$ to contain points of Alexandrov second differentiability of $u$ 
since the leaves occupy positive area by Corollary~\ref{C:biLipschitz} and Fubini's theorem.
Let the corresponding angles be $\theta_k$. Because the leaves don't intersect other sides
of the square, symmetry and the sign of the angle yield $R(\theta_k) \rightarrow0$.  

{But this contradicts the fact that $(Du(x)-x) \cdot \mathbf{n}$ is positive and increasing,
given Proposition \ref{P:strictly-convex-implies-neumann} and
our assumption \eqref{Assumption 2} that $u \in C^{1,1}(\overline\Omega)$.} 

\textit{Case 1b. (There exists a leaf crossing the domain).}  In the second case we pick
any leaf with one endpoint (call it $x_1$) on $\Omega_E$ and the other on $\Omega_N$.
Then $|\theta|=\pi/4$ by symmetry.  Note $D_1u(x_1),D_2 u(x_1) > a+1$ (by the Neumann
inequality on $\Omega_E$ and $\Omega_N$).  Also this leaf bounds a right triangle $T$
with sides $\tilde{x_1}$, and segments of $\Omega_N,\Omega_E$ (Figure
\ref{fig:angle-determines-u1}(B)). Define
\begin{equation}
  \label{eq:triangle-extension}
   \bar{u}(x) :=
  \begin{cases}
    u(x_1) + Du(x_1) \cdot (x-x_1) &\text{for }x \in T\\
    u(x) &x \in \Omega \setminus T
  \end{cases}
  .
\end{equation}
  Because $\bar{u}$ is defined by extension of an affine support for $u$, for
all $x \in \text{int}\,T$, $\bar{u}(x) < u(x)$. Moreover for $x \in T$ we have
  \[ |D\bar{u}(x) - x| \leq |Du(x) - x|,\] this is because monotonicity of the
gradient and the Neumann condition implies for $x \in T$, $D_iu(x) \ge D_iu(x_1) >
a+1$, whereas $x^i \leq a+1$.  Thus $\bar{u}$ is admissible for
\eqref{eq:monopolist} and strictly decreases $L[u] = \int_{\Omega} |Du-x|^2/2+u \, dx$,
a contradiction, given that $u$ minimizes $L$.

    \textit{Case 2. (Positively sloped leaf).} If our originally chosen leaf
has positive slope (i.e. $\theta<0$) the proof is similar, with slight modifications
in the lower right corner.  Indeed, $h'(\theta) > 0$ implies all leaves below our
chosen leaf also have positive slope
and on such leaves $t \mapsto D_1 u(a+1,t)$ is a decreasing function (by monotonicity
of $Du$, as in Case 1). Thus the Neumann value $(Du-x) \cdot \mathbf{n} =  D_1 u(a+1,t)-a-1$ increases as we
move towards the lower right corner. Proposition \ref{P:strictly-convex-implies-neumann} then
implies each $x \in \{a+1\}\times[a,t_0]$ is the endpoint of a nontrivial leaf with
positive slope (since Lemma \ref{L:lower left} again prevents such rays from accumulating onto the convex set $\Omega_0$, and rays in the boundary are ruled out by Lemma \ref{lem:no-nonzero-neumann-boundary-ray} as above).  
Consider the same two alternatives as above: there is a
sequence of leaves whose endpoints on $\Omega_E$ converge to $(a+1,a)$ and whose other endpoint is interior to $\Omega$, or there is not.

  In the first case the contradiction is the same as in Case 1a above. In the second case
choose a leaf with endpoint $x_1$ on $\Omega_E$ and other endpoint on $\Omega_S$.  By the
Neumann inequality, Proposition \ref{prop:neumann-sign}, $D_1 u(x_1) > a+1$ while $D_2
u(x_1) \le a$.  For $x$ in the  interior of the right triangle $T$ formed by $\tilde{x_1}$ and
segments of $\Omega_{E}$, $\Omega_{S}$ monotonicity of the gradient implies
  \begin{align}
    \label{eq:2}
    D_{1}u(x) > D_{1}u(x_1) > a+1,\\
    D_{2}u(x)  \le D_2u(x_1)  \le a.
  \end{align}
  Thus the affine extension as in \eqref{eq:triangle-extension} once again
satisfies $|D\bar{u} - x| < |Du - x|$ (because $x \in [a,a+1]^2$) and
$\bar{u} < u$ in $T$. Thus $L[\bar{u}] < L[u]$ --- the same contradiction as in Case 1b
above.
\end{proof}

\begin{remark}[No ray has positive slope]
\label{rem:ow-angle}
A similar argument to the above implies no leaf intersecting $\Omega_{W}$ or $\Omega_S$
has positive slope. Indeed, if a leaf on $\Omega_{W}$ has 
{finite} positive slope its other
endpoint is interior to $\Omega$ (since Proposition~\ref{prop:no-outward} shows the leaf cannot intersect $\Omega_N$ or $\Omega_E$). The
same argument as Case 1 above implies as one moves vertically up $\Omega_{W}$ each boundary
point remains the endpoint of a nontrivial leaf of positive slope. 
{Lemma \ref{lem:no-nonzero-neumann-boundary-ray} combined with the vertically nondecreasing Neumann condition in a region of positively sloped rays on $\Omega_W$} ensures
leaves must have length shrinking to $0$ as they approach
$(a,a+1)$ and thus we obtain the same contradiction as in Case 1a above.
As a result along $\Omega_{W}\cap \Omega_1$ the function $\theta \mapsto u_1(x(0,\theta))$ is nondecreasing (equivalently $\theta \mapsto (Du-x)\cdot\mathbf{n}\big\vert_{x=x(0,\theta)}$ is nonincreasing; see again Figure \ref{fig:angle-determines-u1}(A)). 
\end{remark}

{
\begin{lemma}[Counting stray rays]
\label{L:Neumann unimodality}
Let $u \in C^{1,1}(\Omega)$ minimize \eqref{eq:monopolist} on a convex domain $\Omega \subset [0,\infty)^2$ whose boundary is smooth except perhaps on a closed set $S \subset \p\Omega$. Then every (one-ended) stray ray ending at a point outside $S$ is a limit of tame rays (from both the left and the right).
When $\Omega$ is also a polygon,  then each such stray ray forms the boundary between two intervals of tame rays,  hence
there are only countably many of these stray rays.
\end{lemma}

\begin{proof}
Let us first recall that a point
 $x_0 \in \p \Omega\setminus S$ belongs to a one-ended stray ray if and only if
$\tilde x_0 \ne \{x_0\}=\tilde x_0\cap \p \Omega$ but $(Du(x_0)-x_0)\cdot \mathbf{n}(x_0)=0$.
Fix such a point $x_0$.  
From $\tilde x_0 \ne \{x_0\}$ it follows that $x_0 \not\in \Omega_2$.
Recall that $\Omega_0$ is a closed convex set containing everything in $\Omega$ 
below and/or to its left.  Unless $\Omega_0$ is empty, it follows that $\Omega_0 \cap \p \Omega$ contains an interval of positive length, by Theorem \ref{thm:structure-convexity}(1).
Thus $x_0 \not\in \Omega_0$.   
Convexity of $\Omega$ implies the remaining set $\Omega_1=\Omega_1^1 \cup \Omega_1^0 \cup \Omega_1^\infty$ is partitioned into rays that have one, two,  or infinitely many points on $\p\Omega$. From $\tilde x_0 \ne \{x_0\}=\tilde x_0\cap \p \Omega$
it follows that $x_0 \in \Omega_1^1$,  and because both boundary endpoints of any two-ended rays accumulating to $\tilde{x_0}$ must converge to $x_0$, $\tilde{x_0} \setminus \{x_0\}$ lies outside {the closure of} $\Omega_1^0$.
Since $\Omega_1^\infty \subset \p \Omega$, each point
$x_1 \in \tilde x_0 \setminus \{x_0\}$ has a neighbourhood whose closure $V\subset \Omega$ is exhausted by $\Omega_1^1 \cup \Omega_2$.
We claim $x_1$ is a limit of tame rays.  Otherwise $V \setminus \Omega_2$ consists exclusively of stray rays, which we now show leads to a contradiction.  Taking $V$ smaller if necessary (to avoid rays which intersect $S$) ensures the set $U$ of stray rays which intersect $V$ is closed.  Since the rays in $U$ are stray,  the disintegration  
$\sigma = \int f \hat \nu_{y} d\tilde \sigma_+(y)$ provided by Lemma \ref{L:identifying disintegration}
must vanish on $U \cap \p \Omega$, hence $f=3-\Delta u$.  Since the rays in this family are not too short,  their directions are a Lipschitz function of their midpoints.  {As in the proof of Proposition \ref{P:no distortion implies strict convexity}, Lemmas \ref{lem:convex-laplacian} and \ref{L:absolute continuity} make $f=3 - \Delta u $ a permissible convex test function on $\tilde \sigma_+$-almost every ray in $U$.}  It then follows that $f=0$ since  $f_+ \hat{\nu_y}$ dominates $f_- \hat{\nu_y}$ in convex order, according to Theorem~\ref{T:strictly convex disintegration}.   The identity
$\Delta u=3$ therefore holds $\mathcal H^2$-a.e. on $V$ (as $f=0$ on $\tilde\sigma_+$-a.e.\ ray in $U$ gives $\sigma_+(U)=0$, we have $\sigma_-(U)=-\sigma(U) \le 0$ by Corollary \ref{lem:Du-balances}) \label{rem:extending disintegration}. 

Since Lemma \ref{lem:leq/geq} implies $\Delta u=3$ on $\Omega_2$,  we conclude $V$ is disjoint from $\Omega_2$ by Lemma 
\ref{lem:mp-argument} (recalling that $u$ fails to be strictly convex at $x_1 \in V$).
The entire neighbourhood $V$ is therefore foliated by stray rays,  whose endpoints foliate 
a neighbourhood of $x_0$ on $\p \Omega$.  But this contradicts Proposition \ref{P:no distortion implies strict convexity}.

Since we can use a reflection across the stray ray $\tilde x_1$ to make the same argument on the half of the neighbourhood $V$ which lies to either side of $\tilde x_1$, we have now shown $x_1$ (and hence $x_0$) to be a limit of tame rays from both the left and the right.
Let $\Omega$ be a polygon and $S$ its vertex set.
By definition, 
 the tame rays intersect $\p\Omega$ precisely in the relatively open set 
$T:=\{ x \in \p \Omega \setminus (S \cup Z \cup \overline{\Omega_1^0} ); (Du(x)-x)\cdot \mathbf{n}>0\}$,  in view of Proposition~\ref{P:no distortion implies strict convexity} and 
Lemma \ref{lem:no-nonzero-neumann-boundary-ray}.  Here $Z :=\{u=0\}$.
Now $T$ consists of countably many open segments in $\p \Omega$.  
On each such segment, 
the argument use to prove Proposition \ref{prop:no-outward} shows the monotonic growth of
$(Du(x)-x)\cdot \mathbf{n}$ cannot end except at $S\cup Z \cup \Omega_1^0$. 
Since $x_0 \not \in S \cup Z \cup \overline{\Omega_1^0}$, no sequence of segments in $T$ can accumulate at $x_0$;  we must have $x_0$ being the endpoint of one segment of $T$ to its left and another to its right.  This shows there can be no more stray ray ends on $\p \Omega \setminus S$ than there are segments in $T$, as desired.
\end{proof}}

Another key point that it will be helpful to have at our disposal is the following.

\begin{lemma}[Concave nondecreasing profile of stingray's tail]
\label{lem:boundary-portions-convex}
  Let $\omega$ be some connected subset of $\Omega_{1}^-$. Then $Du(\omega) = \{Du(x) =
(y^1,y^2) ; x \in \omega \}$ is such that $y^2$ is a strictly convex increasing
function of $y^1$ lying above the diagonal whose monotonicity \eqref{stingray
monotone} and convexity \eqref{stingray convex} are easily quantified below in terms
of the parameters $\theta(t)=t$ and $m(t)=m(\theta)$ from \eqref{eq:general-form}.
\end{lemma}

\begin{proof}
Connectivity of $\omega\subset \Omega_1^-$ combines with Theorem
\ref{thm:structure-convexity}(2) and the definition of $\Omega_1^-$ to imply the set
$\hat \omega = \{\tilde x \cap \partial \Omega; x \in \omega\}$ of fixed boundary endpoints of rays
intersecting $\omega$
is also connected --- hence forms an interval on the left
boundary $\Omega_W$ of the square; except possibly for zero slope ray including $(a,a+1)$, it cannot intersect the relative interior of
$\Omega_N$ according to Proposition \ref{prop:no-outward}.  On the set of rays whose
endpoints lie in the relative interior of this interval $\hat \omega$, the
Lipschitz regularity of $m'$ from Assumption \ref{a:localc11}, i.e. \eqref{Assumption 2}, and Corollary \ref{C:biLipschitz} justifies the following computations.

We wish to consider the convexity of the curve $y(\theta) = (y^1(\theta),y^2(\theta))$ where by \eqref{eq:u1theta} and \eqref{eq:u2theta}
\begin{align*}
y^1(\theta)&=m(\theta) \cos \theta - m'(\theta) \sin \theta\\
y^2(\theta)&= m(\theta) \sin \theta + m'(\theta) \cos \theta.
\end{align*}
Using Lipschitz regularity of $m$ 
we have for $\mathcal{H}^1$ almost every $\theta$
\begin{equation}
  \label{stingray monotone}
  \frac{dy_2}{dy_1} = \frac{\dot{y}^2(\theta)}{\dot{y}^1(\theta)} = -\frac{\cos\theta}{\sin\theta}>0. 
\end{equation}
Here the sign condition comes from Remark \ref{rem:ow-angle}. 
Note this implies the curve $y(\theta)$ is such that $y^2$ is an increasing function of $y_1$. We see $\frac{dy_2}{dy_1}$ is an increasing function of $\theta$, namely
\begin{equation}
  \label{stingray convex}
  \frac{d}{d\theta}\frac{dy_2}{dy_1} = \frac{1}{\sin^2\theta}
\end{equation}
and subsequently, by Remark \ref{rem:ow-angle} which implies $\theta$ is an increasing function of $y^1$, $\frac{dy_2}{dy_1}$ is an increasing function $y_1 = u_1$.Thus we have the required monotonicity of the derivative to conclude $y_2$ is a convex function of $y_1$. 
\end{proof}

Symmetry implies connected components of $Du(\Omega_{1}^+)$ are curves with $y^2$ a
concave function of $y^1$.

Now we can combine all the foregoing lemmas to complete the proof of
Theorem \ref{thm:description-on-square}.

\begin{proof}(Theorem \ref{thm:description-on-square}).  By symmetry about the
  diagonal and $(Du -x )\cdot \mathbf{n} = 0$ on $\Omega_E $
we can prove each point of the theorem by an
analysis of the function $t \mapsto (Du-x) \cdot \mathbf{n}\big\vert_{x = (a,t)}$.  Lemma
\ref{L:lower left} asserts $(a,a)$ is in $\Omega_0$ as is $\{(a,a+t) ; 0 \le t \le \alpha\}$ 
for some $\alpha=c-a \in (0,1]$. On $\Omega_0 \cap \Omega_W$ we have $(Du-x)\cdot \mathbf{n} =
a$. whereas on $\Omega_2 \cap \Omega_W$ we have $(Du-x)\cdot \mathbf{n} = 0$. Thus, for $a>0$, our assumption $u \in
C^1(\overline{\Omega})$ implies some
portion of $\Omega_{1}$ must abut $\Omega_0$ as one moves up $\Omega_{W}$.

  Now we consider the configuration of $\Omega_1$. Since leaves in $\Omega_{1}^{0}$ reach
the diagonal, by symmetry they are orthogonal to it and $Du(x) = (b,b)$ on such
leaves, i.e. the product $Du(x)$ selected lies on the diagonal.

 \textit{Step 1. (Configuration of domains)} We claim as one moves vertically
up $\Omega_W$ there is, in order, (i) a closed interval of $\Omega_0$ with positive
length, (ii) a half-open interval of $\Omega_{1}^{0}$ which is empty for $a$
sufficiently small and nonempty for $a$ sufficiently large, (iii) a nonempty
open interval of $\Omega_{1}^-$, and 
{possibly} an interval of $\Omega_2$. All we must show
is there is  at most a single component of $\Omega_{1}^{0}$, and it is followed by $\Omega_{1}^-$.  This is
because, if $\Omega_{1}^{0}$ and $\Omega_{1}^-$ exist their ordering follows from Lemma
\ref{lem:boundary-portions-convex}. Indeed if a portion of $\Omega_{1}^-$ is
preceded by $\Omega_0$ or $\Omega_{1}^{0}$ then followed by $\Omega_{1}^{0}$ we have the
contradiction of a strictly convex curve lying above the diagonal with a start
and endpoint on the diagonal in the stingray's profile.
  
 \textit{Step 2. (Blunt bunching (i.e.\ $\Omega_{1}^{0}\ne \emptyset$) for $a \ge \frac{11}{2} -\sqrt{2}$).}
 Recall $u$ cannot be affine on any segment in the closure of $\Omega_2$, by 
 a reflection argument combined with the strong maximum principle of Lemma \ref{lem:mp-argument}. It follows that
$\Omega_1^-$ is nonempty whenever $\Omega_{1}^{0}$ is nonempty; this was
previously established by a different 
approach in~\cite{McCannZhang24+}. Next we
assume $\Omega_{1}^{0}$ is empty and show $a < \frac{11}{2} -\sqrt{2}$. Let $(a,\underline{x}_2)$ be
the upper endpoint of $\Omega_0 \cap \Omega_W$ and let $(a,\overline{x}_2)$ be the lower
endpoint of $\Omega_2 \cap \Omega_W$ if the latter is nonempty and  $(a,a+1)$ otherwise. The segment $\{a\} \times (\underline{x}_2,\overline{x}_2)$
consists of endpoints of leaves in $\Omega_{1}^{-}$. By the Neumann condition we
have $D_1u(a,\underline{x}_2) = 0$ and $D_1u(a,\overline{x}_2)  \geq a-2$: the last inequality follows because if $\overline{x_2} < a+1$, the bound $D_1 u(a,\overline{x}_2) \le a$ of Proposition \ref{prop:neumann-sign} is saturated, while otherwise $\Omega_N$ consists of a boundary ray followed by a portion in $\Omega_2$ along which $D^2_{11}u \le \Delta u =3$; subsequently the condition $Du(a+1,a+1) = (a+1,a+1)$ implies $D_1u(a,a+1)\geq a-2$.  The same estimates also establish that $(\Omega_E \cup \Omega_N) \cap B_{1/3}(a+1,a+1)$ is disjoint from $\Omega_1$. Thus,
  \[ \int_{\underline{x}_2}^{\overline{x}_2}\p^2_{12} u(a,x_2) \, dx_2 \geq a-2.  \]
  As in (4.17) of \cite{McCannZhang23+},   using the $(r,\theta)$ coordinates we compute 
  \[ \p^2_{12} u(a,x_2) = -\sin(\theta) \frac{m''(\theta)+m(\theta)}{h'(\theta)}.\]
  From \eqref{eq:spec-first}, which reads $m''(\theta)+m(\theta)-3h'(\theta)\cos\theta = 2R(\theta) ,$  we have
  \begin{align*}
 \p^2_{12} u(a,x_2) &= -\sin(\theta)\frac{2R(\theta)}{h'(\theta)} - 3 \sin(\theta)\cos(\theta) 
\\&= -2\sin(\theta) R(\theta) \theta'(x_2) -\frac{3}{2}\sin(2\theta).  
  \end{align*}
  We've used the inverse function theorem to rewrite $\frac{1}{h'(\theta)} = \theta'(x_2)$. Using $-\sin(\theta) \geq 0$ from Remark \ref{rem:ow-angle} and $R \leq 1$ and $\theta \ge -\frac\pi4$ for convexity of $\Omega_0$ we conclude
  \begin{align*}
    a-2 &\leq \int_{\underline{x}_2}^{\overline{x}_2} \p^2_{12} u(a,x_2) \, dx_2 \\
      &=  -\int_{\underline{x}_2}^{\overline{x}_2}   [2\sin(\theta) R(\theta) \theta'(x_2) +\frac{3}{2}\sin(2\theta)] \, dx_2\\
  & < 2\|R\|_\infty [\cos (0) - \cos (-\frac \pi 4)]  + \frac32 [\bar x^2 - \underline x^2]\\
    &\le \frac72 -\sqrt2.
\end{align*}

  \textit{Step 3. (No blunt bunching (i.e. $\Omega_{1}^{0}=\emptyset$) for $a \ll 1$
sufficiently small)} Suppose for a contradiction that there is a sequence $a_k
\downarrow 0$ such that the minimizer on $ \Omega^{(k)} = (a_k,a_k+1)^2$ has $\Omega_{1}^{0}$
nonempty. Let $u_k$ be the minimal convex extension to $\mathbf{R}^n$ of the
corresponding minimizer to \eqref{eq:monopolist}. Let, for example
$(\Omega^0_1)^{(k)}$ denote $\Omega^0_1$ 
for the problem on $(a_k,a_k+1)^2$, and domains
with no superscript denote the corresponding domain for $a=0$.  The convergence
result \cite[Corollary 4.7]{FigalliKimMcCann11} implies $u_k|_{B_\epsilon} \rightarrow u|_{B_\epsilon}$
locally uniformly for any $B_\epsilon \subset (0,1)^2$ where $u$ is the minimizer for $a=0$.

  It is clear that $\Omega_{1}$ is empty when $a = 0$. Indeed, for $a = 0$ the
solution on $[0,1]^2$ is the restriction of the solution on $[-1,1]^2$.
The solution on $[-1,1]^2$ satisfies $\Omega_1 = \emptyset$: Theorem
\ref{thm:structure-convexity}(2) asserts the rays all extend to the boundary
but Proposition \ref{prop:no-outward}, which is valid also on $[-1,1]^2$
asserts there can be no ray intersecting the boundaries $\{1\}\times[0,1]$ and
$[0,1]\times \{1\}$. By symmetry there are no rays intersecting anywhere on $\partial
[-1,1]^2$ and hence no rays whatsoever.

Recall (4.18) of \cite{McCannZhang23+} asserts $\{u_k = 0\}$ is the triangle $(x^1,x^2) \in [a_k,a_k+1]^2$ defined by 
$x^1 + x^2 \le 2a_{k} + \frac{2a_{k}}3 (\sqrt{1+ \frac3{2a^2_{k}}}-1)$.
The limit $\Omega_0:=\{u=0\}$ from Theorem \ref{thm:structure-convexity} must therefore contain
the triangle $x^1 +x^2 \le \sqrt{2/3}$ of area $\frac13$ in $[0,1]^2$.  Nor can $\Omega_0$ be larger than this triangle, since Proposition \ref{prop:neumann-sign} implies
both integrands are non-negative in the identity
$$
1 = \int_{\Omega_0} 3 d\mathcal{H}^2 + \int_{\Omega_0 \cap \p\Omega} (Du-x) \cdot \mathbf{n} d\mathcal{H}^1
$$
asserted by Corollary \ref{lem:Du-balances}.
But the previous paragraph implies $\Omega_1=\emptyset$,  so outside the triangle $\Omega_0$ the minimizing $u \in C^{1,1}_{\text{loc}}((0,1)^2)$ is a strictly convex solution to Poisson's
equation $\Delta u=3$.  Reflecting this solution across the line $x^1+x^2 =\sqrt{2/3}$ where it vanishes contradicts the strong maximum principle 
(Lemma \ref{lem:mp-argument}).  This is the desired contradiction which establishes Step 3.

\textit{Step 4. (No further $\Omega_1$ components)} From Remark \ref{rem:ow-angle}
any leaves intersecting $\Omega_{W}$ have nonpositive slope. Thus, recalling Figure
\ref{fig:angle-determines-u1}(A), the Neumann value $(Du-x)\cdot \mathbf{n}$ is a
positive decreasing function of $x_2$ along $\Omega_1 \cap \Omega_{W}$ and $0$ along $\Omega_2 \cap
\Omega_{W}$. Thus there cannot exist an $\Omega_1$ component above an $\Omega_2$ component on
$\Omega_W$.
\end{proof}

\begin{remark}[Estimating the bifurcation point] It is clear that the criterion $a \ge \frac{11}{2} -\sqrt{2}$ for the existence of $\Omega_{1}^{0}$ is not sharp. However the presence of a bifurcation
reflects the radically different behavior we have shown the model to display
for small and large $a$. We expect there is a single bifurcation value $a_0$
such that $\Omega_{1}^0$ is nonempty for $a>a_0$ while $\Omega_{1}^0$ is empty for
$0<a \le a_0$.  It would be interesting to confirm this expectation, and to find or
estimate $a_0$ more precisely.
\end{remark}

\appendix
\section{Localization and sweeping 
\`a la Rochet and Chon\'e}
\label{sec:rochet-chones-use}

A main goal of this appendix is to establish the localization result Theorem~\ref{T:strictly convex disintegration}. For $u \in C^1(\overline\Omega)$, Theorem \ref{T:strictly convex disintegration} below
plays a central role in the analysis of Rochet and Chon\'e~\cite{RochetChone98}.
However, as explained in Remark \ref{R:ChenFigalliZhang}, the best regularity
known for the minimizer \eqref{eq:monopolist} is currently $u \in (C^{1,1}_{loc} \cap
C^{0,1})(\Omega)$. It is therefore worthwhile to extend the theorem 
result to this less regular setting, in which \eqref{equivalence classes}
defines convex equivalence classes $\tilde x_0$ which partition the set $\dom D^-u$ 
 that fills $\mathcal H^{n-1}$ almost all of $\overline \Omega$. Care
is then needed to disintegrate the variational derivative $\sigma$ from
\eqref{eq:mu-def} --- whose positive and negative parts $\sigma_\pm$ need not vanish on
$\p \Omega$ a priori. Therefore, we prove Theorem \ref{T:strictly convex
disintegration} using $u \in (C^1\cap C^{0,1})(\Omega)$.  See Ciosmak~\cite{Ciosmak23+}
and its references for a more abstract approach to similar issues.

For completeness, let us first give a proof of the Lipschitz bound; it is
attributed to Chon\'e thesis \cite{Chone99} in \cite{CarlierLachand--Robert01},
but we were unable to access a copy.

\begin{lemma}[Bound on types of products selected]
\label{L:Lipschitz}
  Let $u:\Omega \rightarrow \mathbf{R}$ be the minimizer of \eqref{eq:monopolist} on a bounded, open, convex domain $\Omega\subset \R^n$. Then $2\sup_{x \in \Omega}|x|$ is a Lipschitz constant for $u$.
\end{lemma}

\begin{proof}
  Note that $u$ is locally Lipschitz in $\Omega$ and differentiable almost everywhere. Thus, it suffices to establish $|Du(x)| \leq 2 |x|$ at any point of differentiability $x \in \Omega$. Let $v$ denote the Legendre transform of $u$ restricted to $\partial u(\Omega)$. We claim  $v(y) \geq |y|^2/2$ .  For a contradiction assume otherwise and set $\tilde{v}(y) = \max\{v(y),|y|^2/2\}$ and let $\tilde{v}^*$ denote its Legendre transform which is an admissible competitor. We have $\tilde{v}^* \leq u$ with strict inequality on a positive measure subset of $\Omega$. Moreover, at points of differentiability with $\tilde{v}^*(x) \neq u(x)$, $D\tilde{v}^*(x) = x$. Thus $\tilde{v}^*$ decreases $|Du(x)-x|^2/2+u$ on a set of positive measure (and increases it nowhere). This contradicts that $u$ minimizes \eqref{eq:monopolist} and thus establishes $v(y) \geq |y|^2/2$. Evaluating the inequality $v(y) \geq |y|^2/2$ at $y=Du(x)$ implies
\[ 0 \leq x \cdot Du(x) - u(x) - \frac{|Du(x)|^2}{2} \leq x \cdot Du(x) - \frac{|Du(x)|^2}{2}. \]
This implies $ |Du(x)|^2 \leq 2 |x||Du(x)|$ as required. 
\end{proof}

We recall as in \eqref{eq:perturb-2} the minimizer $u$ of \eqref{eq:monopolist}
satisfies the variational inequality
\begin{equation}
  \label{eq:min} 0 \leq \int_{\Omega} (n+1 - \Delta u )v(x) \, dx + \int_{\partial \Omega} v(x) (D^-u-x) \cdot
\textbf{n} \, d\mathcal{H}^{n-1} ,
\end{equation}
for all convex $v$ with $\text{spt}\, v_{-}$ disjoint from
$\{u = 0\}$.  We set
\begin{equation}\label{eq:mu-def} 
  d \sigma := (n+1 - \Delta u) \,  d\mathcal{H}^n\mres {\Omega} + (D^{ -}u-x) \cdot \textbf{n} \, d \mathcal{H}^{n-1}
\mres \partial \Omega,
\end{equation} and can rewrite
\eqref{eq:min} as
\begin{equation}
  \label{eq:order} 0 \leq \int_{\overline{\Omega}} v(x) \,d \sigma(x).
\end{equation}
In what follows we need to remove the condition $\text{spt}\, v_{-}$ disjoint from $\{u = 0\}$ which we do via the following lemma.

\begin{lemma}[Restoring neutrality]\label{lem:neutral}
\label{lem:decomp} Let $\sigma$ be as in \eqref{eq:mu-def}. There exists a nonnegative
measure $\lambda \ge 0$ supported on $ \{u=0\}
\cap \overline\Omega$ 
with $\lambda(\R^n)=|\Omega|$ such that for every continuous convex $v : \R^n \rightarrow \mathbf{R} $ we have
  \begin{equation}
    \label{eq:simpler-conc}
    0 \leq \int v d (\sigma-\lambda).
  \end{equation}
\end{lemma}

\begin{proof}  First note for all $t \geq 0$, $tu$
is admissible for the minimization problem. Thus by minimality
  \begin{equation}
    \label{eq:mu-u-0} 0 = \frac{d}{dt}\Big\vert_{t = 1} L[tu] = \int_{\overline{\Omega} }
u(x)\, d \sigma.
\end{equation}
Combined with the minimality condition \eqref{eq:order} we have
\[ 0 = \inf\left\{ \int_{\overline{\Omega}} v \, d \sigma ; v \text{ is convex with $v \geq0$}\right\},\]
with $v = u$ realizing the infimum. A classical theorem in the
calculus of variations says the constraint $v \geq 0$ may be realized by a
Lagrange multiplier (see, for example, \cite[Theorem 1,pg. 217]{Luenberger69}).
Thus there is a nonnegative Radon measure $\lambda$ such that
  \begin{equation}
    \label{eq:lagrange} 0 = \inf\left\{ \int_{\overline{\Omega}} v \, d (\sigma-\lambda) ; v
\text{ is convex}\right\},
\end{equation}
with $v = u$ still realizing the infimum. Using that $u$
attains the infimum along with \eqref{eq:mu-u-0} yields
\[ \int u \, d \lambda = 0.\]
Since 
$\lambda$ is nonnegative we
conclude $\text{spt}\lambda \subset \{u=0\}$. Moreover, \eqref{eq:lagrange} implies \eqref{eq:simpler-conc}.
Now $L(u+t)=L(u)+t|\Omega|$ for $t \ge 0$ implies
$\sigma(\overline \Omega)=|\Omega|$; convexity of $v=\pm 1$
yields $\lambda(\overline \Omega)=|\Omega|$.  
\end{proof}

Let $u$ denote the minimal convex extension of $u$ from $\Omega$ to $\R^n$,
$\dom Du$ its points of differentiability, and $u^*$ its Legendre-Fenchel transform.  From the duality of points $x$ with slopes $y$, it follows that $y \in \p u(x)$ if and only if $x \in \p u^*(y)$.

By Rademacher's theorem,
the restriction of $u$ to $\p \Omega$ is differentiable
at $\mathcal H^{n-1}$ a.e. point $x$ of differentiability for $\p\Omega$.  
For such $x$,
Gangbo and McCann \cite[Lemmas 1.6 and 2.2]{GangboMcCann00} show
$\p u(x)=[T^-(x),T^+(x)]$ with $T^\pm$ Borel and
$T^+(x)-T^-(x)$ either vanishing or else being outer normal to $\p \Omega$ at $x$.  We set $D^\pm u(x):=T^\pm(x)$
and denote by $\dom D^-u$ the set $\dom Du$ together with
the points $x$ of differentiability for $\p \Omega$ at which  
$\p u(x) =[D^-u(x),D^+u(x)]$ as above. We call $D^-u$ 
the {\em inner} 
and $D^+u$ the {\em outer extension} of $Du$.  
Noting that $Du$
has bounded variation,  these extensions also coincide with the inner and outer boundary traces of $Du$ defined, e.g. as in \cite[Theorem 3.77]{AmbrosioFuscoPallara00} or \cite[Theorem 5.6]{EvansGariepy92}.

Theorem \ref{T:strictly convex disintegration} states there is a (unique) point $z \in \{u=0\}$ such that the positive and negative parts $\bar \sigma_\pm$ of $\bar \sigma :=\sigma-|\Omega|\delta_z$ are in convex order. 
Moreover, they have the same push-forward $\tilde \sigma_+$ under $D^-u$, and their 
probabilities $\bar \sigma_{\pm,y}$ conditioned on the value $y$ of $D^-u$ are also in convex order.
We abbreviate this conditioning \eqref{B disintegrate+} by writing $\bar \sigma_\pm=\int \sigma_{\pm,y} d\tilde \sigma_+(y)$. As is standard, we define the conditioning via disintegrations. The following two background results are key tools in the coming proofs.

\begin{theorem}[{Abstract disintegration \cite[Theorem 5.3.1]{AmbrosioGigliSavare05}}]
\label{T:abstract disintegration}
Let $F:X\longrightarrow Y$ be a Borel map between
two Radon separable spaces,  and $\mu\ge 0$ 
a constant multiple $ c \in (0,\infty)$ of a 
Borel probability measure $\hat \mu \in \mathcal P(X)$,
and introduce the push-forward $\tilde \mu = F_\#\mu
\in c \mathcal P(Y)$ of $\mu$ by $F$.  
Then there is family of probability 
measures $\{\mu_y\}_{y \in Y} \subset \mathcal P(X)$
satisfying $\mu_y[F^{-1}(y)]=1$ for $\tilde \mu$-a.e. $y \in Y$, and such that every bounded continuous $\phi\in C(X)$ satisfies
\begin{equation}
    \label{B disintegrate-}
\int_{X} \phi(x) d\mu(x) = \int_{Y} 
\left(\int_X
\phi(x) d\mu_{y}(x)\right) d\tilde \mu(y),
\end{equation}
and the integral in parenthesis is a Borel function of $y$.   If a second family $\{\bar \mu_y\}_{y \in Y}$ satisfies the same conclusions,  then $\bar \mu_y=\mu_y$
holds for $\tilde \mu$-a.e. $y$.
\end{theorem}

\begin{definition}[Disintegration; essential uniqueness]
\label{D:disintegrate}
The preceding theorem remains true as long as $F$ is defined $\mu$-a.e. Under this hypothesis we call $\{\mu_y\}_{y \in Y}$ the {\em disintegration} of $\mu$ with respect to $F$.
We abbreviate this by writing $\mu = \int \mu_y d\tilde \mu(y)$.  
We encapsulate the last sentence of the theorem by saying disintegration is {\em essentially unique}.
\end{definition}

Henceforth,  we shall apply Theorem \ref{T:abstract disintegration} to various measures $\mu$ but always 
with $X =\overline \Omega$ and 
$F=D^-u$,  where $\Omega$ and $u \in (C^1\cap C^{0,1})(\Omega)$ are convex.
Since $F$ may fail to be defined on a boundary subset having negligible surface area,  we'll require $\mu \mres \p \Omega$ to be absolutely continuous with respect to $\mathcal H^{n-1}$.
As the following example shows,  when we have a smooth enough change of variables at our disposal, it is possible to identify the disintegration more explicitly;
in this example $u \not\in C^1(\Omega)$ but this causes no harm since differentiability of $u$ fails only on a $\mu$ neglible set.

\begin{example}[Spherical polar coordinates]
\label{E:spherical polar coordinates}
Let $\mu = \mathcal H^n\mres \Omega$ 
be Lebesgue measure on the Euclidean ball $\Omega = B_1(0) \subset \Rn$ and 
 $u(x)=|x|$,  so $F(x)=x/|x|$ represents outward
radial projection of the ball to its boundary $Y=\p B_1(0)$.    From the classical change of variables to spherical polar coordinates
\begin{equation}\label{spherical coordinates}
\int_{B_1(0)} \phi(|x|,\frac{x}{|x|}) d\mathcal H^n(x)
=\int_{\p B_1(0)} \left(\int_0^1 \phi(r,y) r^{n-1}dr\right) d\mathcal H^{n-1}(y)
\end{equation}
for all $\phi \in C(\overline \Omega)$,
we deduce that (i)
$\tilde \mu = F_\#\mu = \frac1n \mathcal H^{n-1}\mres \p\Omega$ and (ii)
essential uniqueness of the disintegration 
$\{\mu_y\}_{y\in Y}$ 
provided by Theorem~\ref{T:abstract disintegration} 
yields that $\tilde \mu$ a.e. $y$ satisfies: absolute continuity of $\mu_y$ with respect to $\mathcal H^1$ on the interval $F^{-1}(y)=[0,y]$, and a density $\frac{d \mu_y}{d \mathcal H^1}(x) = n |x|^{n-1}$, i.e. $\frac{d \mu_y}{d r} = n r^{n-1}$.
\end{example}

\begin{theorem}[A disintegration respecting convex order]
\label{T:strictly convex disintegration}
Let $u:\R^n\longrightarrow\R$ be the minimal convex extension of the minimizer of \eqref{eq:monopolist} on a bounded, open, convex domain $\Omega\subset \R^n$,  and $\sigma$ from \eqref{eq:mu-def}.
Then there exists $z \in \{u=0\}$ such that 
the integral of $\bar \sigma :=\sigma - |\Omega|\delta_z$
against every real convex function is nonnegative.
Its positive and negative parts $\bar \sigma_+=\sigma_+$ and $\bar \sigma_-=\sigma_- + |\Omega|\delta_z$ both vanish outside $\dom D^-u$,
and $\tilde \sigma_\pm = (D^-u)_\#\sigma_+$ where $\tilde \sigma_\pm:=(D^-u)_\#\bar \sigma_\pm$. 
{The essentially unique disintegrations $\{\bar \sigma_{\pm,y}\}_{y \in \R^n}$
of $\bar \sigma_\pm$ with respect to $D^-u$ 
from Definition \ref{D:disintegrate} satisfy:
for $\tilde \sigma_+$-a.e. $y$,}
{
\begin{equation}
    \label{B disintegrate+}
\int_{\R^n} \phi(x) d\bar\sigma_\pm(x) = \int_{\R^n} 
\left(\int_{\R^n}
\phi(x) d\bar\sigma_{\pm,y}(x)\right) d\tilde \sigma_+(y)
\end{equation}
and moreover
$\tilde \sigma_+$-a.e.  
$y$ satisfies
}
\begin{equation}\label{consistency and local dominance+}
{\bar \sigma_{\pm,y}(\R^n\setminus \p u^{*}(y))=0\ {\rm and}\ }
\int_{\R^n} w(x) d\bar \sigma_{-,y}(x) \le \int_\Rn w(x) d\bar \sigma_{+,y}(x)
\end{equation}
for all continuous bounded $\phi \in (C \cap L^\infty)(\R^n)$ and all convex $w:\R^n\longrightarrow\R$.
\end{theorem}

\begin{proof}
Although $u \in (C_{loc}^{1,1} \cap C^{0,1})(\Omega)$ by Lemma \ref{L:Lipschitz} and
\cite{mccann2023c}, we shall only need $u \in (C^1 \cap C^{0,1})(\Omega)$.  We see $\sigma_\pm$
are absolutely continuous with respect to $\mathcal H^{n-1}$ --- indeed,
${d\sigma_\pm}/{d\mathcal H^{n-1}} \le \|u\|_{C^{0,1}(\Omega)}$ --- and have finite mass, hence
the Borel map $D^-u=T^-$ is defined $\sigma_\pm$ a.e.  by \cite[Lemmas 1.6 and
2.2]{GangboMcCann00}.

Let $\lambda \ge 0$ be the measure supported on $\{u=0\}$
from Lemma \ref{lem:decomp} and denote its barycenter by $z$.
Since both $\lambda - |\Omega|\delta_z$ and $\sigma - \lambda$ have nonnegative integrals against any real convex function,  the
same is true for their sum --- whose negative and positive parts
$\bar \sigma_- := \sigma_- + |\Omega|\delta_{z}$ and $\bar \sigma_+=\sigma_+$
are therefore in convex order.
Let $S$ denote the relative interior of the convex set $\{u=0\}\cap\overline\Omega$.

Note that employing the Armstrong perturbation $\overline{u} = \text{max}\{u-h,0\}$ in \eqref{eq:perturb-0} and sending $h \rightarrow 0$ implies $\sigma(\overline{\Omega}\setminus \overline{S}) \leq 0$. Since Lemma \ref{lem:neutral} yields $\sigma(\overline{\Omega})=|\Omega|$ we deduce $\sigma(\overline{S}) > 0$. 
Thus 
{$\sigma_+$} is nonvanishing on $\overline S$ and subsequently $\overline S \cap \dom D^-u$ is non-empty;
 {\label{rem:parenthetical}(although not needed until Corollary~\ref{lem:u=0_set_dimension}, it is evident that $\overline S \subset \p \Omega$ must be an $n-1$ dimensional boundary facet unless $S$ intersects $\Omega$).}
  Since every hyperplane which supports $u$ (respectively $\overline \Omega$) at one point in $S$ supports $u$ (respectively $\overline \Omega$) at all points in $\overline S$, it follows that $S \subset \dom D^-u$ and moreover $D^-u$ is constant on $S$.  Recalling that $\p u(x)$ parallels the normal to $\p \Omega$ at each $x \in \dom D^-u$,  we see that $D^-u$
takes the same constant value --- call it $\bar y$ --- 
throughout $\overline S \cap \dom D^-u$.

(i) Strassen's theorem \cite{Strassen65} asserts the existence of random variables $Z$ distributed like $\sigma_+$ and $X$ distributed like $\bar \sigma_-$, such that $X=E(Z|X)$ a.s.; in other words, such that $(X,Z)$ is a martingale. (For concreteness, we may take the underlying probability space to be Lebesgue measure on the unit interval, and $X,Z:[0,1]\longrightarrow \R^{n}$ to be Borel measurable.)

Define the Borel functions $\bar Y := D^-u \circ X$ and $Y:=D^-u \circ Z$; {in case $z
\not\in \dom D^-u$, we set $\bar Y = \bar y$ whenever $X=z$.}  Define $\bar
\sigma_{+,y}$ to be the law of $Z$ conditioned on $Y=y$, and define $\bar \sigma_{-,y}$ to
be the law of $X$ conditioned on $\bar Y=y$; these exist and depend Borel
measurably on $y$ by Theorem \ref{T:abstract disintegration} (or \cite[Theorem
1.7]{Panchenko19}). Now \eqref{B disintegrate+} holds for $\bar \sigma_+$; it also
holds for $\bar \sigma_-$, except with the law 
$\frac{\tilde {\sigma}_-}{\tilde {\sigma}_-(\R^n)}$
of $\bar Y$ replacing the law
$\frac{\tilde \sigma_+}{\tilde \sigma_+(\R^n)}$ of $Y$.  From part (ii) (resp. (iv)) of
the proof it will follow that $\bar Y=Y$ a.s. (resp. $z \in \dom D^-u$) hence
${\tilde {\sigma}_-}=\tilde \sigma_+ = (D^-u)_\#\sigma_+$ since $\bar \sigma_+=\sigma_+$.

(ii) We now turn to the first statement in
\eqref{consistency and local dominance+}.
Since we have minimized $L(u)$ over a cone of functions,
both $w= \pm u$ yield nonnegative convex perturbations of $u$ whence $0=\int u d\sigma = \int u d\bar\sigma$,  where $u(z)=0$ has been used. The expectation of 
$$
u(Z) - u(X) - (Z-X)\cdot \bar Y \ge 0
$$ 
therefore vanishes by the martingale property,  
hence almost surely $X,Z \in \p u^*(\bar Y)$.
Temporarily 
suppose realizations $X,Z$ both in $\p u^*(\bar Y)$
are fixed, with $Z \in \dom D^-u$ and $X \in \{z\} \cup \dom D^-u$. 
If $\p u^*(\bar Y)$ is a singleton then $X=Z \in \dom D^-u$ and $Y=\bar Y$ as desired. Otherwise let $S_0 \ne \emptyset$ 
be the relative interior of the convex set $\p u^*(\bar Y)$.
 Since every hyperplane which supports $u$ (respectively $\overline \Omega$) at one point in $S_0$ supports $u$ (respectively $\overline \Omega$) at all points in $\overline S_0$, it follows that $S_0 \subset \dom D^-u$ and 
 that $y_0=D^-u(x_0)$ is independent of $x_0 \in S_0$. A priori, one might worry that $\p u(X)$ and $\p u(Z)$ are 
 larger than $\p u(x_0)$, 
but we now show they are not.
Since $\p u(Z)$ 
parallels the unique normal at $Z$ to $\p \Omega$, 
each supporting hyperplane to $u$ at $Z$
also supports $u$ on $S_0$.  Hence $D^-u(Z)=y_0$.
Similarly $D^-u(X)=y_0$ 
unless $X=z \not \in \dom D^-u$.
In the latter case $\bar Y=\bar y$ lies in 
$\p u(Z)= [y_0,D^+u(Z)]$ and $\mathbf n :=\bar y-y_0\ne 0$
must be the unique outer normal direction at 
$Z \in \p \Omega$ unless $\bar Y=Y$.
Unless $Y=\bar Y$, monotonicity $\mathbf n \cdot (x-Z)\ge 0$ shows each
$x \in \overline S \cap \dom D^-u \subset \p u^*(\bar y)$
to lie in the facet of $\p\Omega$ having unique outer normal direction $\mathbf n$.
Thus $\bar y =D^-u(x)= y_0$ after all. 
Either way, $\bar Y = Y$ and
we conclude \eqref{B disintegrate+}, and 
the first part of \eqref{consistency and local dominance+} holds.

(iii) We still need to deduce the second assertion in \eqref{consistency and local dominance+}.
This follows by applying $w$ and Jensen's inequality to the martingale identity $X=E(Z|X)$. 
Indeed, for $\tilde \sigma_+$ a.e. $y$ and $\bar \sigma_{y,-}$ a.e. value of $X$ 
we have
$$
w(X) = w(E( Z|X)) \le E(w(Z)|X).
$$
Averaging this identity against $d\bar \sigma_{y,-}(X)$ yields
inequality \eqref{consistency and local dominance+}
since the distribution of $Z$ conditioned on $Y=y$ is $\bar \sigma_{+,y}$. 

(iv) Finally $\tilde \sigma_+(\{\bar y\})\ge|\Omega|>0$, hence $\sigma_{-,\bar y}$
is dominated by $\sigma_{+,\bar y}$ in convex order.
It follows that $z$ lies in the convex hull of the support $\sigma_{+,\bar y}$; otherwise choosing $v$ linear to be positive at $z$ but negative $\sigma_{+,\bar y}$-a.e. yields a contradiction.
Hence $z \in \dom D^-u$ a posteriori,
since any supporting hyperplane to $u$ (respectively $\Omega$) at $z$ must then also support $u$ (respectively $\Omega$)
at spt$\sigma_{+,\bar y} \cap \dom D^-u \ne \emptyset$.
\end{proof}

A number of corollaries are straightforward consequences of Theorem~\ref{T:strictly convex disintegration}. The first two concern the zero set: both its configuration and variational derivative. The  convex set $\overline S:=u^{-1}(0) \cap \overline \Omega$ where the nonnegativity constraint binds represents the agents who receive their reservation utility (zero) at equilibrium.  On a domain $\overline \Omega \subset [0,\infty)^n$ whose convexity is not strict,  the following result shows that 
the relative interior $S$ of $\overline{S}$ occupies a positive fraction of either the domain interior or of its boundary.
In the former case $Du$ vanishes on $S$, but in the latter case one-dimensional examples show this need not remain true: in some but not all cases agents who receive their (zero) reservation utility may also receive a nonzero product $Du(S)=\{\bar y\} \ne \{0\}$ in addition~\cite{MussaRosen78}. 
This 
refines the dichotomy between the desirability of exclusion in
one versus higher dimensions
observed by Armstrong \cite[Proposition 1]
{Armstrong96}:
when the domain is strictly convex we recover his conclusion;
when it has no $n-1$ dimensional facets we recover the stronger implication of Figalli, Kim and McCann \cite[Theorem 4.8]{FigalliKimMcCann11}.

\begin{corollary}[Exclusion versus possible nonexclusion]
\label{lem:u=0_set_dimension}
Either (i) the relative interior $S$ of the convex set $\overline{S}:=\{x\in \overline{\Omega} : u(x)=0\}$ has positive volume and $D^-u(S)=\{0\}$,  or else (ii) $S \subset \p \Omega$ and has positive area
$\mathcal{H}^{n-1}(S)>0$.  

\end{corollary}
\begin{proof}
  The first parenthetical remark on page \pageref{rem:parenthetical} indicates that if $\overline{S} \subset \partial \Omega$ it is an $n-1$ dimensional boundary facet. On the other hand, if $\overline{S} \cap \Omega$ is nonempty then $\sigma(\overline{S}) = |\Omega|$  from Lemma \ref{lem:neutral} and \eqref{eq:mu-def} couples with convexity of $S$ and $u \in C^{1,1}_{loc}(\Omega)$ to
  imply $S$ is an open set. Since at any $x \in S \cap \Omega$ the function $u$ attains an interior minimum we have $D^-u=0$ throughout $S$ in this case. 
\end{proof}

\begin{corollary}
[Objective responds proportionately to uniform utility increase]
\label{lem:Du-balances}
  Let $\sigma$ denote the variational derivative (recall equation \eqref{eq:variational-derivative}). Then 
  $(D^-u)_{\#}\sigma =  \mathcal{H}^n(\Omega)\delta_{\bar y},$ 
  that is $(D^-u)_{\#}\sigma$ 
  is a Dirac mass at the point $\bar y = D^-u(z)$ where $z \in u^{-1}(0)$ is from Theorem \ref{T:strictly convex disintegration}. 
\end{corollary}

\begin{proof} This follows by applying \eqref{consistency and local dominance+} with $w = \pm 1$ and $y = D^{-}u(z)$. Using $\bar{\sigma}_- = \sigma_- +|\Omega|\sigma_z$ we obtain   $(D^-u)_{\#}\sigma =  \mathcal{H}^n(\Omega)\delta_{\bar y}$ . 
\end{proof}

  We conclude by establishing the extent to which the abstract objects provided by the disintegration may be related to concrete objects such as the Radon--Nikodym derivative $(n+1-\Delta u)1_\Omega + (D^{-}u-x)\cdot \mathbf{n}1_{\partial \Omega}$ of $\sigma$ with respect to the sum $\nu := \mathcal H^n\mres \Omega + \mathcal{H}^{n-1}\mres \p \Omega$
  of the volume and surface area measures.  {This is delicate since $Du$ is not biLipschitz}: Counterexamples first considered in light of Sudakov's work \cite{Sudakov76}, imply absolute continuity need not be inherited even under disintegration with respect to lines in $\mathbf{R}^{3}$. We overcome these difficulties by using that the disintegration is with respect to the faces of a convex graph. Related interior results to part of Lemma \ref{L:absolute continuity} were proved by Caravenna and Daneri \cite{CaravennaDaneri2010}.

\begin{remark}[Extending the disintegration]\label{rem:extending-disintegration}
We now come to a subtle point.  We would like to relate the abstract disintegrations $\bar \sigma_\pm = \int \bar \sigma_{\pm,y } d\tilde\sigma_+(y)$ more concretely to 
$\Delta u$ on $\Omega$ and $D^-u$ on $\p \Omega$.  Although the disintegrations are essentially unique,  they tell us nothing on unions $U \subset \overline\Omega$ of leaves  
which turn out to be $|\sigma|$ negligible. (The customization region $\Omega_n$ is one example of a such a union.)
In other words,  because $\tilde \sigma := (D^-u)_\#\sigma_+$ vanishes on $V:=D^-u(U)$,  for $y \in V$ the choice of 
$\bar \sigma_{\pm,y}$ on such leaves is arbitrary.  For such leaves, it proves convenient to {\em define} $\bar \sigma_{\pm,y}:=0$.  In particular, $\bar \sigma_{\pm,y}:=0$ on any leaf $\tilde{x}$ on which both $n+1-\Delta u = 0$ $\mathcal{H}^{\dim\tilde{x}}$ a.e. on $\tilde{x}$ and $(D^-u - x ) \cdot \mathbf{n} = 0$ $\mathcal{H}^{\dim\tilde{x}-1 }$ a.e. on $\overline{\tilde{x}} \cap \partial \Omega$.
\end{remark}

\begin{lemma}[Identifying the disintegration]\label{L:identifying disintegration}
Let $u:\R^n\longrightarrow\R$ be the minimal convex extension of the minimizer of \eqref{eq:monopolist} on a bounded, open, convex domain $\Omega\subset \R^n$,  and $\sigma$ from \eqref{eq:mu-def}.  Let 
$\nu = \int \nu_y d\tilde \nu(y)$ denote the disintegration of $\nu:= \mu^{n-1} + \mu^{n}$ with respect to $D^-u$
{as in Definition \ref{D:disintegrate}},  where
$\mu^{n-1} = \mathcal H^{n-1} \mres \p \Omega$ and $\mu^{n} = \mathcal H^{n} \mres \Omega$ are the surface and volume measures of $\Omega$.
Then $\sigma = \int (f \hat \nu_y) d\tilde \sigma_+(y)$,  where 
$$
f:=(n+1-\Delta u)1_{\Omega} + (D^-u-id)\cdot \mathbf n 1_{\p \Omega}
$$ 
and $ \hat \nu_y := \nu_y / \|f_+\|_{L^1(d\nu_y)}$ when $ \|f_+\|_{L^1(d\nu_y)} \neq 0$ and $\hat \nu_y:=0$ otherwise. 
\end{lemma}
 \begin{proof}
  It suffices to establish that for every continuous bounded $\phi$ there holds
  \begin{equation}
    \label{eq:disint-equality}
    \int \phi \, d \sigma = \int_{\mathbf{R}^{n}}\int_{\mathbf{R}^{n}} \phi(x){f}(x) \, d \hat\nu_y(x) d \tilde{\sigma}_+(y).
  \end{equation}
  Because we may write
  \[ \int \phi \, d \sigma = \int\phi(x) f(x) \, d \nu, \]
after using $\tilde \nu = (D^-u)_\# \nu$ to disintegrate $\nu = \int \nu_y d\tilde \nu$ with respect to $D^-u$ 
   we have
  \begin{align*}
    \int \phi \, d \sigma &= \int_{\mathbf{R}^{n}}\int_{\mathbf{R}^{n}} \phi(x) f(x)  d \nu_y(x)d\tilde \nu(y) \\
    &= \int_{\mathbf{R}^{n}}\int_{\mathbf{R}^{n}} \phi(x) {f}(x)  d \hat \nu_y(x) \|f_+\|_{L^1(d\nu_y)} d\tilde \nu(y).
  \end{align*}
  Thus, we've reduced \eqref{eq:disint-equality} to establishing $d \tilde{\sigma}_+(y) =  \|f_+\|_{L^1(d\nu_y)} d\tilde \nu(y)$. Fix arbitrary Borel $B \subset \R^n$ and compute via a disintegration 
  \begin{align*}
    \int_{B}d \tilde{\sigma}_+(y) &= \int_{(D^{-}u)^{-1}(B)} f_+ \, d \nu\\
                          &= \int_{\mathbf{R}^{n}}\int_{\mathbf{R}^{n}}f_+(x) 1_{(D^{-}u)^{-1}(B)}(x) d \nu_y(x)d\tilde \nu(y) \\
    &= \int_{B} \|f_+\|_{L^1(d\nu_y)} d\tilde \nu(y), 
  \end{align*}
  as required. To move from the second to third equality we've used that  $\nu_y[(D^-u)^{-1}(B)]=1_B(y)$ holds $\tilde \nu$-a.e.
  by Theorem \ref{T:abstract disintegration}.
\end{proof} 
%\pagebreak

\begin{lemma}[Disintegrating planar length and area absolutely continuously]\label{L:absolute continuity}
For $\Omega \subset \R^2$ open, bounded, and convex, disintegrate
the boundary length and area on $\Omega$, 
\begin{align*}
\mu^{1} := \mathcal H^{1} \mres \p\Omega &= \int_{\R^2} \mu^{1}_y d\tilde \mu^{1}(y), 
\\ \mu^2 := \mathcal H^2 \mres \Omega &= \int_{\R^2} \mu^2_y d\tilde \mu^2(y),
\\ \nu := \mu^{1} + \mu^2 &= \int_{\R^2} \nu_y d\tilde \nu(y)
\end{align*} 
with respect to $D^-u$ {as in Definition \ref{D:disintegrate}}. Note $\tilde \nu := (D^-u)_\# \nu = \tilde \mu^{1}+\tilde \mu^2$.
Outside of a $\tilde \mu^2$-negligible set, 
we have 
$\mu^2_y$ {is mutually absolutely continuous} with respect to 
$\mathcal H^k \mres (\tilde x \cap \Omega)$
where $\tilde x = (D^-u)^{-1}(y)$ and $k=\dim \tilde x$.
Similarly, for each $y$ outside of a $\tilde \mu^1$-negligible set, we have $\mu^1_y \ll \mathcal H^j\mres (\tilde x \cap \p \Omega)$, with $\tilde x$ as above and  
$j=\dim (\tilde x \cap \p \Omega)$.
Finally, $d\nu_y (\cdot) = \frac{d\tilde \mu^1}{d\tilde \nu}d\mu^1_y (\cdot) + \frac{d\tilde \mu^2}{d \tilde \nu} d\mu^2_y(\cdot)$ 
outside a $\tilde\nu$ negligible set of $y$.
\end{lemma}

\begin{proof}
Recall that 
$\mu^i_y$ vanishes outside  $\tilde x = (D^-u)^{-1}(y)$
for $i=1,2$ and $\tilde \mu^i$ almost all $y$;  similarly,
$\mu^1_y$ vanishes outside $\p \Omega$ and $\mu^2_y$ vanishes outside $\Omega$.
We first consider the region $\Omega_1$ of rays in the partition
$\dom D^-u = \Omega_0 \cup \Omega_1 \cup \Omega_2$.  Since the rays comprising $\Omega_1$ are disjoint intervals in $\R^2$,
as long as the rays are not too short their directions vary in a Lipschitz fashion as a function of their midpoints \cite[Remark 6.1]{Ambrosio03};  (this simply quantifies the observation that the closer the midpoint of a unit segment approaches a point of distance at least $1/2$
from the endpoint of any longer disjoint segment in the same plane,  the more the segments must align).
Also, the distance of each ray end from its midpoint varies upper semicontinuously,  so the set of ray ends $E \subset \Omega_1$ has zero area \cite[Lemma 22--25]{CaffarelliFeldmanMcCann00}
(being the image of a zero area set under the extention to $\overline V_k$ of the Lipschitz map $G_k$ indicated below).  As in Caffarelli, Feldman, and McCann,  it is therefore possible to partition 
$\Omega_1 \setminus E=\cup_k U_k$ into countably many Borel unions $U_k$ of (relative interiors of) rays,  each of which has a locally Lipschitz 
(ray-flattening) 
map $F_k:U_k \longrightarrow \R^2$ which makes $F_k(\tilde x)$ parallel to the vertical axis for each $x \in U_k$ and acts as an isometry on $\tilde x$.  Moreover, $F_k$ has a globally Lipschitz inverse map $G_k$ defined on  $V_k:=F(U_k)$. 
The push-forward of $\mathcal H^2 \mres U_k$ through $F_k$ is absolutely continous with respect to Lebesgue \cite[Lemma 5.5.3]{AmbrosioGigliSavare05}, with Radon-Nikodym derivative $g_k = \frac{d(F_k)_\#(\mathcal H^2 \, \mres \,  U_k)}{d \mathcal H^2}$ given by $g_k = 1_{V_k} \det DG_k$.  For $\tilde \mu^2$ a.e. $y$,
setting $\tilde x = (D^-u)^{-1}(y)$,  the co-area
formula \cite[\S 3.4.3]{EvansGariepy92} shows the 
(essentially unique) disintegration of $\mu^2$ 
with respect to $D^-u$ is 
given by
$
d\mu^2_y(\cdot) = 
\hat g_k(\cdot) d\mathcal H^1 \mres\tilde x
$
where
$$
\hat g_k := \frac{g_k \circ F_k}{\int_{\tilde x} g_k\circ F_k d\mathcal H^1.}
$$
In particular,  $\mu^2_y \ll \mathcal H^1 \mres \tilde x$ 
and the positivity of $\hat g_k$ on the relative interior of $\tilde x$ shows $\mathcal H^1 \mres \tilde x \ll \mu^2_y$ desired.

We now consider the restriction of $\mu^1=\mathcal H^1\mres \p \Omega$ to $\Omega_1$.
There are (at most) countably many rays contained in $\p \Omega$.  On each of these rays,  $\mu^1_y$ is uniform,
with a density inversely proportional to the length of the ray.
In addition, there are potentially uncountably many rays which intersect
$\p \Omega$ only in one or two endpoints.  However the probability
measure $\mu^1_y$ is divided between these endpoints,  it is absolutely continuous with respect to 
$\mathcal H^0 \mres (Du^{-1}(y) \cap \p\Omega)$, as desired.

The analogous claims on $\Omega_0$ and $\Omega_n$ are easy to establish.  Each leaf in $\Omega_n$ consists of a point,
hence $k=0$ and the conditional probability measure
$\mu^2_{D^-u(x)}$, being $\tilde \mu^2$ a.s. supported on $\tilde x = \{x\}$, is mutually absolutely continuous with respect to $\mathcal H^0 \mres \tilde x$.
Similarly $\mu^1_{D^-u(x)}$, being $\tilde \mu^1$ a.s. supported on $\tilde x\cap \p \Omega$, is absolutely
continuous with respect to $\mathcal H^0 \mres \tilde x$ if non-vanishing.  Each of the (at most) countably many leaves $\tilde x$ in $\Omega_0$ consists of a convex set with non-empty interior.
For $y = D^-u(x)$ with $x \in \Omega_0$,  the conditional measures $\mu^i_y$
are irrelevant unless $c_i^{\tilde x} =  \mu^i[\tilde x]>0$, in which case 
$c_i^{\tilde x} \mu^i_y = {\mu^i \mres \tilde x}$ verifies the absolute continuity 
asserted for $i=1,2$. The claims concerning $\nu$ follow from the essential uniqueness of its disintegration.
\end{proof}

\begin{remark}[Disintegrating surface area and volume in higher dimensions]
Although not needed here, we believe it should be possible to extend the absolute continuity assertions of Lemma~\ref{L:absolute continuity} to higher dimensional
convex sets $\Omega \subset \Rn$.   For the volume
the relevant result is in Caravenna and Daneri~\cite{CaravennaDaneri2010}. For the surface
area
the remaining obstacle is to countably partition each 
$\Omega_i \setminus E= \cup_k U^i_k$ into Borel unions $U^i_k$ of relative leaf interiors which can be made parallel through a locally Lipschitz map $F_k:U^i_k \longrightarrow \R^n$ with 
globally Lipschitz inverse $G_k:F_k(U^i_k)\longrightarrow U^i_k$,
where $\mathcal H^n(E)=0$.
We expect it is possible to construct such maps using the fact that the Legendre transform $u^*$ is countably $C^2$ (outside a set of zero volume),  as is its restriction to the 
countably $C^2$ submanifold (outside an $\mathcal H^{n-i}$ negligible set) of dimension $n-i$ where $\dim \p u^*(y) \ge i$
\cite{Zajicek79}.
\end{remark}

\bibliographystyle{plain}
\bibliography{bibliography}

\end{document}

%% file: neumann-tikz.tex
\begin{tikzpicture}[scale=1.5]
          %\draw[help lines,opacity=0.1] (-4,-3) grid (8,3);
          %\draw[thick] (0,-2) -- (0,2);
  \draw [thick] (0.5,1.9) arc [start angle=152, end angle=198.5, radius=5];
            \node[text width=5cm, align=center] at (0,-2.5) {Boundary portion\\$\gamma([-2\alpha,2\beta]) \subset \partial \Omega$};
       %   \node[below] at (0,-2) {$\gamma([-\alpha,\beta]) \subset \partial \Omega$};

          \draw (-0.07,0) -- (4,-1);
          \node[left] at (0,0) {$\tilde{x_0}$};
          
          \draw (0.12,1) -- (4,0.5);
          \node[left] at (0,1) {$\widetilde{\gamma(\beta)}$};
          
          \draw (-0.06,-1) -- (4,-2.5);
          \node[left] at (0,-1) {$\widetilde{\gamma(-\alpha)}$};

          \draw [dashed] (0,0.1) -- (4.5,0.4);
          \node[below] at (4.5,0.4) {$\tilde{x_b}$};

          \draw [dashed] (-0.05,-0.6) -- (1,-0.6);
          \draw [dashed] (-0.05,-0.5) -- (0.7,-0.3);
          \draw [dashed] (-0.05,-0.7) -- (0.7,-0.9);
           \draw [dashed] (-0.05,-0.35) -- (0.4,-0.15);
           %\draw [dashed] (0,-0.4) -- (0.4,-0.2);
           \draw [dashed] (-0.05,-0.8) -- (0.4,-1);

           \draw [dotted] (0.5,1.9) -- (5,1);
           \draw [dotted] (0.2,-2.05) -- (5,-3);
           \draw [dotted] (5,1) -- (5,-3);
           \node[right,text width=3cm, align=center] at (4.8,-2) {vertical height--- \\ $2(\alpha+\beta)$};
           \node[above,rotate=-10] at (2.5,1.5) {length $5R(x_0)/4$};
           \node[text width=5cm, align=center] at (6.3,0.6) {Containment set for\\ $Du^{-1}\big(Du(\gamma[-\alpha,\beta])\big)$};
\draw [thick] (0.5,1.9) arc [start angle=152, end angle=198.5, radius=5];
        \end{tikzpicture}

%%% Local Variables:
%%% mode: LaTeX
%%% TeX-master: "main"
%%% End: